\documentclass[10pt,oneside,a4paper]{amsart}
\usepackage{amsmath,amsfonts,amsthm, amssymb}
\usepackage{array}
\usepackage[all]{xy}
\usepackage{portland}
\usepackage{color}
\usepackage[orange]{xcolor}

\theoremstyle{plain}
\newtheorem{theorem}{Theorem}[section]
\newtheorem{proposition}[theorem]{Proposition}
\newtheorem{lemma}[theorem]{Lemma}
\newtheorem{corollary}[theorem]{Corollary}

\theoremstyle{definition}
\newtheorem{definition}[theorem]{Definition}

\theoremstyle{remark}
\newtheorem{remarks}[theorem]{Remarks}
\newtheorem{remark}[theorem]{Remark}
\newtheorem{example}[theorem]{Example}

\newcommand{\Kern}{\mathrm{Ker}}

\newcommand{\Beeld}{\mathrm{Im}}

\newcommand{\Mor}{\mathrm{Mor}}
\renewcommand{\lim}{\mathrm{lim}}
\newcommand{\colim}{\mathrm{colim}}

\newcommand{\Ext}{\mathrm{Ext}}

\newcommand{\Hom}{\mathrm{Hom}}

\newcommand{\Def}{\mathrm{Def}}

\newcommand{\op}{^{\mathrm{op}}}
\newcommand{\Ob}{\mathrm{Ob}}
\newcommand{\Z}{\mathbb{Z}}
\newcommand{\Q}{\mathbb{Q}}

\newcommand{\N}{\mathbb{N}}

\newcommand{\AAA}{\mathfrak{a}}
\newcommand{\BBB}{\mathfrak{b}}

\newcommand{\CCC}{\mathfrak{c}}

\newcommand{\CC}{\mathbf{C}}

\newcommand{\Alg}{\ensuremath{\mathsf{Alg}} }

\newcommand{\Mod}{\ensuremath{\mathsf{Mod}} }
\newcommand{\Bimod}{\ensuremath{\mathsf{Bimod}} }

\newcommand{\Pre}{\ensuremath{\mathsf{Pr}} }

\newcommand{\mmod}{\ensuremath{\mathsf{mod}} }

\newcommand{\Qch}{\ensuremath{\mathsf{Qch}} }
\newcommand{\QPr}{\ensuremath{\mathsf{QPr}} }
\newcommand{\Cat}{\ensuremath{\mathsf{Cat}} }

\newcommand{\Des}{\ensuremath{\mathrm{Des}}}
\newcommand{\PDes}{\ensuremath{\mathrm{PDes}}}

\newcommand{\lra}{\longrightarrow}

\newcommand{\aaa}{\ensuremath{\mathcal{A}}}
\newcommand{\bbb}{\ensuremath{\mathcal{B}}}
\newcommand{\ccc}{\ensuremath{\mathcal{C}}}
\newcommand{\ddd}{\ensuremath{\mathcal{D}}}

\newcommand{\fff}{\ensuremath{\mathcal{F}}}
\newcommand{\GGG}{\ensuremath{\mathcal{G}}}

\newcommand{\nnn}{\ensuremath{\mathcal{N}}}
\newcommand{\ooo}{\ensuremath{\mathcal{O}}}
\newcommand{\ppp}{\ensuremath{\mathcal{P}}}

\newcommand{\ttt}{\ensuremath{\mathcal{T}}}
\newcommand{\uuu}{\ensuremath{\mathcal{U}}}

\newcommand{\can}{\mathsf{can}}
\allowdisplaybreaks
\numberwithin{equation}{section}

\title{Non-commutative deformations and quasi-coherent modules}
\author{Hoang Dinh Van}
\address[Hoang Dinh Van]{Universiteit Antwerpen, Departement Wiskunde-Informatica, Middelheimcampus,
Middelheimlaan 1,
2020 Antwerp, Belgium}
\email{hoang.dinhvan@uantwerpen.be}

\author{Liyu Liu}
\address[Liyu Liu]{Universiteit Antwerpen, Departement Wiskunde-Informatica, Middelheimcampus,
Middelheimlaan 1,
2020 Antwerp, Belgium}
\email{liyu.liu@uantwerpen.be}

\author{Wendy Lowen}
\address[Wendy Lowen]{Universiteit Antwerpen, Departement Wiskunde-Informatica, Middelheimcampus,
Middelheimlaan 1,
2020 Antwerp, Belgium}
\email{wendy.lowen@uantwerpen.be}
\thanks{The authors acknowledge the support of the European Union for the ERC grant No 257004-HHNcdMir and the support of the Research Foundation Flanders (FWO) under Grant No G.0112.13N}

\begin{document}
\maketitle

\begin{abstract}
We identify a class of \emph{quasi-compact semi-separated (qcss)} twisted presheaves of algebras $\aaa$ for which well-behaved Grothendieck abelian categories of quasi-coherent modules $\Qch(\aaa)$ are defined. This class is stable under algebraic deformation, giving rise to a 1-1 correspondence between algebraic deformations of $\aaa$ and abelian deformations of $\Qch(\aaa)$. For a qcss presheaf $\aaa$, we use the Gerstenhaber-Schack (GS) complex to explicitly parameterize the first order deformations. For a twisted presheaf $\aaa$ with central twists, we descibe an alternative category $\QPr(\aaa)$ of quasi-coherent presheaves which is equivalent to $\Qch(\aaa)$, leading to an alternative, equivalent association of abelian deformations to GS cocycles of qcss presheaves of commutative algebras. Our construction applies to the restriction $\ooo$ of the structure sheaf of a scheme $X$ to a finite semi-separating open affine cover (for which we have $\Qch(\ooo) \cong \Qch(X)$). Under a natural identification of Gerstenhaber-Schack cohomology of $\ooo$ and Hochschild cohomology of $X$, our construction is shown to be equivalent to Toda's construction from \cite{toda} in the smooth case.
\end{abstract}

\section{introduction}

\subsection{Motivation}
In algebraic geometry, the category $\Qch(X)$ of quasi-coherent sheaves on a scheme $X$ is of fundamental importance. In non-commutative contexts, an important task is to find suitable replacements for this category. For a possibly non-commutative algebra $A$, the category $\Mod(A)$ of (right) $A$-modules is a natural replacement since for $A$ commutative, we have $\Mod(A) \cong \Qch(\mathrm{Spec}(A))$. This way, algebraic structures naturally giving rise to well-behaved ``categories of quasi-coherent sheaves'' can be considered to be non-commutative algebraic counterparts of schemes (corresponding to a ``choice of coordinates''). An important instance constitutes the basis for non-commutative projective geometry, in which to a sufficiently nice $\Z$-graded algebra $A$, one associates a category $\mathsf{QGr}(A)$ of ``quasi-coherent graded modules'', obtained as the quotient of the graded modules by the torsion modules. This is motivated by Serre's well-known result that, for $A$ commutative, we have $\mathsf{QGr}(A) \cong \Qch(\mathrm{Proj}(A))$ (see \cite{staffordvandenbergh}, \cite{vandenbergh}, \cite{vandenbergh2}, \cite{dedekenlowen} for more details and generalizations).

In this paper, we take the ``local approach'' to non-commutative schemes and consider a presheaf of algebras $\aaa\colon \uuu^{\op} \lra \Alg(k)$ on a small category $\uuu$ as a kind of ``non-commutative structure sheaf on affine opens''. Mimicking the process by which a quasi-coherent sheaf is obtained by glueing modules on affine opens, we define the category of \emph{quasi-coherent modules} over $\aaa$ to be
\begin{equation}\label{eqqch}
\Qch({\aaa}) = \Des(\Mod_\aaa),
\end{equation}
the descent category of the prestack $\Mod_{\aaa}$ of module categories on $\aaa$.

\subsection{Twisted deformations and the Gerstenhaber-Schack complex}\label{intro1}
Our main aim is to identify a class of presheaves of algebras $\aaa$ which behave well with respect to deformation. Let $k$ be a field. For simplicity, in this paper we focus on first order deformations (in the direction of $k[\epsilon]$). In order to understand what to expect, let us look at the situation for algebras. For a $k$-algebra $A$, the Gerstenhaber deformation theory of $A$ is controlled by the Hochschild cohomology $HH^{\ast}(A) = \Ext^{\ast}_{A^{\op} \otimes A}(A,A)$. In \cite{lowenvandenberghab}, a deformation theory for abelian categories $\ccc$ was developed. This theory is controlled by an intrinsic notion of Hochschild cohomology $HH^{\ast}_{\mathrm{ab}}(\ccc)$ \cite{lowenvandenberghhoch}. For first order deformations, we obtain a commutative square of isomorphisms:
\begin{equation}\label{square}
\xymatrix{ {HH^2(A)} \ar[d]  \ar[r]^{\Psi_1^A} & {\Def_{\mathrm{alg}}(A)} \ar[d]^{\Psi_2^A} \\  {HH^2_{\mathrm{ab}}(\Mod(A))} \ar[r] & {\Def_{\mathrm{ab}}(\Mod(A))} }
\end{equation}
in which the ``upper route'' $\Psi^A = \Psi^A_2 \Psi^A_1$ is given by
\begin{equation}\label{eqamod}
\Psi^A\colon HH^2(A) \lra \Def_{\mathrm{ab}}(\Mod(A)),\quad \phi \longmapsto \Mod(\bar{A}_{\phi})
\end{equation}
where $\bar{A}_{\phi}$ is the first order algebra deformation of $A$ canonically associated to a Hochschild $2$-cocycle $\phi$ in the Hochschild complex $\CC(A)$.

The first main aim in this paper, is to construct the upper route for presheaves of algebras $\aaa$. Let $\aaa\colon \uuu^{\op} \lra \Alg(k)$ be a presheaf of $k$-algebras on a small category $\uuu$.
In \S \ref{pargstw}, we describe an explicit isomorphism
\begin{equation}\label{hhtw}
\Psi^{\aaa}_1\colon HH^2(\aaa) \lra \Def_{\mathrm{tw}}(\aaa),\quad \phi \longmapsto \bar{\aaa}_{\phi}.
\end{equation}
Here, $HH^{\ast}(\aaa) = \Ext^{\ast}_{\aaa^{\op} \otimes \aaa}(\aaa, \aaa)$ is computed by the Gerstenhaber-Schack complex $\CC_\mathrm{GS}(\aaa)$ \cite{gerstenhaberschack1}, and $\Def_{\mathrm{tw}}(\aaa)$ denotes first order deformations of $\aaa$ \emph{as a twisted presheaf} (Theorem \ref{proptwist}). The fact that an isomorphism exists follows for instance from combining \cite{lowenvandenberghCCT} and \cite{lowenmap}, but for our purpose we are interested in the explicit description of \eqref{hhtw} based upon $\CC_\mathrm{GS}(\aaa)$. 

The complex $\CC_\mathrm{GS}(\aaa)$ is the total complex of a first quadrant double complex, with
$$\CC^{p,q}(\aaa) = \prod_{\sigma \in \nnn_p(\uuu)} \Hom_k(\aaa(c\sigma)^{\otimes q}, \aaa(d\sigma)),$$
where $\nnn(\uuu)$ is the simplicial nerve of $\uuu$, and $d\sigma$ (resp.\ $c\sigma$) is the domain (resp.\ codomain) of a simplex $\sigma$. The differential is obtained from
vertical Hochschild differentials and horizontal simplicial differentials. In particular, the bottom row $\CC^{p,0}(\aaa)$ is the simplicial cohomology complex $\CC_{\mathrm{simp}}(\aaa)$.
We have
\begin{equation}\label{3pieces}
\CC^2_{\mathrm{GS}}(\aaa) = \CC^{0,2}(\aaa) \oplus \CC^{1,1}(\aaa) \oplus \CC^{2,0}(\aaa),
\end{equation}
and in \eqref{hhtw}, a cocycle $\phi = (m_1, f_1, c_1)$ gives rise to a twisted deformation $\bar{\aaa}_{\phi}$ of $\aaa$ in which $m_1$ deforms the individual algebras $\aaa(U)$, $f_1$ deforms the restriction maps, and $c_1$ introduces twists.

If $\aaa$ is a presheaf of \emph{commutative} algebras, the bottom row $\CC_{\mathrm{simp}}(\aaa)$ splits off as a direct summand of the complex $\CC_\mathrm{GS}(\aaa)$, corresponding to the fact that every deformation $\bar{\aaa}_{\phi}$ has central twists and an underlying presheaf $\underline{\bar{\aaa}}_{\phi} = \bar{\aaa}_{(m_1, f_1, 0)}$ associated to the cocycle $(m_1, f_1, 0)$. For a twisted presheaf $\aaa$ with central twists, we describe a category $\QPr(\aaa)$ of \emph{quasi-coherent presheaves}, closer in spirit to the twisted sheaves considered for instance in \cite{caldararu}, and equivalent to the category $\Qch(\aaa)$ (Theorem \ref{thmqmodqpr}).

\subsection{Twisted versus abelian deformations}\label{intro2}
Motivated by the fact that twisted presheaves of algebras naturally occur as deformations of presheaves, we consider definition \eqref{eqqch} a priori for arbitrary twisted presheaves (or prestacks) $\aaa$.
In order for $\Qch(\aaa)$ to be well-behaved under deformation, we impose a number of ``geometric'' conditions upon $\aaa$ in \S \ref{parparqcoh}.

The first condition is actually independent of deformation theory: in order that $\Qch(\aaa)$ is a Grothendieck abelian category, we impose that the restriction functors $u^{\ast}\colon \aaa(U) \lra \aaa(V)$ for $u\colon V \lra U$ in $\uuu$ give rise to the induced $$- \otimes_{\aaa(U)} \aaa(V)\colon \Mod(\aaa(U)) \lra \Mod(\aaa(V))$$ being exact (Theorem \ref{thmgroth}). Our framework encompasses the framework considered in \cite{enochsestrada}.

A fundamental notion in deformation theory is flatness, and in \cite{lowenvandenberghab}, a suitable notion of flatness for abelian categories was introduced. In order that $\Qch(\aaa)$ becomes flat (over a commutative ground ring --- which in our setup will be $k[\epsilon]$), we further impose the following conditions:
\begin{enumerate}
\item $\uuu$ is a finite poset with binary meets;
\item The functors $- \otimes_{\aaa(U)} \aaa(V)$ are the exact left adjoints of compatible locali\-zation functors.
\end{enumerate}

A prestack $\aaa$ satisfying these conditions is called a \emph{quasi-compact semi-separated (qcss)} prestack. In Theorem \ref{defeq}, we prove that for a qcss prestack $\aaa$, there is an isomorphism
\begin{equation} \label{eqtwabintro}
\Psi^{\aaa}_2\colon \Def_{\mathrm{tw}}(\aaa) \lra \Def_{\mathrm{ab}}(\Qch(\aaa)),\quad \bar{\aaa} \longmapsto \Qch(\bar{\aaa}).
\end{equation}
The notion of qcss prestack is preserved under deformation, in particular the isomorphism $\Psi^{\aaa}_2$ inductively extends to an isomorphism between higher order deformations.

Combining \eqref{hhtw} and \eqref{eqtwabintro}, we have now explicitly constructed the upper route $\Psi^{\aaa} = \Psi^{\aaa}_2 \Psi^{\aaa}_1$ for a qcss presheaf of algebras $\aaa$ as
\begin{equation}\label{eqeq1}
\Psi^{\aaa}\colon HH^2(\aaa) \lra \Def_{\mathrm{ab}}(\Qch(\aaa)),\quad \phi \longmapsto \Qch(\bar{\aaa}_{\phi}).
\end{equation}

\subsection{Relation with Toda's construction}\label{introtoda}
The prime example of a qcss presheaf is the restriction $\ooo = \ooo_X|_{\uuu}$ of the structure sheaf of a scheme $X$ to a finite semi-separating cover $\uuu$ (i.e. an open affine cover closed under intersections). Suppose from now on that $\Q \subseteq k$. In \cite{toda}, Toda describes, for a smooth quasi-compact separated scheme $X$, a construction which associates to an element $u \in HH^2(X) = \Ext^2_{X \times X}(\ooo_X, \ooo_X)$, a certain $k[\epsilon]$-linear abelian category $\Qch(X, u)$ which he considers to be a ``first order deformation'' of $\Qch(X)$. This category is not a priori an abelian deformation in the sense of \cite{lowenvandenberghab}. The second main aim in this paper is to clarify the relation between Toda's construction and the map $\Psi^{\ooo}$ from \eqref{eqeq1}. Concretely, Toda's starting point is an element $u$ in
\begin{equation}\label{HKRintro}
 HH^2_\mathrm{HKR}(X, \uuu) = \check{H}^2(\uuu, \ooo_X) \oplus \check{H}^1(\uuu, \ttt_X) \oplus \check{H}^0(\uuu, \wedge^2 \ttt_X).
 \end{equation}
 The resemblance between \eqref{3pieces} and \eqref{HKRintro} is no coincidence. In order to understand it properly, we devote \S \ref{parparhodgeHKR} to a discussion of Gerstenhaber and Schack's Hodge decomposition of $HH^n(\aaa)$ for presheaves of commutative algebras $\aaa$, from which we deduce an HKR decomposition in the smooth case (Theorem \ref{thmHKR}). When applied to the restricted structure sheaf $\ooo = \ooo_X|_{\uuu}$ of a smooth semi-separated scheme, the decomposition translates into \eqref{HKRintro} and we obtain
 \begin{equation}\label{transl}
 HH^n(\ooo) \cong HH^n_\mathrm{HKR}(X, \uuu)
 \end{equation}
 (the application to smooth complex projective varieties is treated by Gerstenhaber and Schack in \cite{gerstenhaberschack1}, and motivated their work). Note that when combined with the isomorphism $HH^n(X) \cong HH^n(\ooo)$ proved for a quasi-compact separated scheme in \cite{lowenvandenberghhoch}, \eqref{transl} yields the classical HKR decomposition formula for schemes, proved for instance by Swan \cite{swan} and Yekutieli \cite{yekutieli} under stronger finiteness assumptions. The possibility to interpret the Hochschild cohomology $HH^2(X)$ in terms of non-commutative deformations of $X$ is an important ingredient in the Homological Mirror Symmetry setup \cite{kontsevich2}.

In Theorem \ref{thmequiv},
we show that if $\phi \in HH^2(\ooo)$ corresponds to $u \in  HH^2_\mathrm{HKR}(X, \uuu)$ under \eqref{transl}, then for the image $\Psi^{\ooo}(\phi) = \Qch(\bar{\ooo}_{\phi})$ of $\phi$ under \eqref{eqeq1}, there are equivalences of abelian categories
$$\Qch(\bar{\ooo}_{\phi}) \cong \QPr(\bar{\ooo}_{\phi}) \cong \Qch(X, u).$$
The intermediate category of quasi-coherent presheaves $\QPr(\bar{\ooo}_{\phi})$ is essential in the proof of the theorem.

In case $X$ is a quasi-compact semi-separated scheme, it follows in particular that $\Qch(X,u)$ is an abelian deformation of $\Qch(X)$ in the sense of \cite{lowenvandenberghab} and the general theory, including the obstruction theory for lifting objects \cite{lowen2}, applies. This is used for instance by Macr\`i and Stellari in the context of an infinitesimal derived Torelli theorem for K3 surfaces \cite[\S 3]{macristellari}.

\subsection{Broader context}\label{introbroader}
We end this introduction by situating the present paper in a broader context. Twisted presheaves and, more generally, (algebroid) prestacks play an important role in deformation quantization, a role which dates back to the work by Kashiwara in the context of contact manifolds \cite{kashiwara}. In the algebraic context, the first proposal to use stacks in deformation quantization was made by Kontsevich in \cite{kontsevich1}. Since the appearance of that paper, several (groups of) people have elaborated and elucidated parts of the suggested approach, and many other directions have been investigated since. The present paper has two main parts, each one starting from the basic notion of twisted deformations as given in Definition \ref{deftwistedpresheaf} (first order case):
\begin{enumerate}
\item the relation with the Gerstenhaber-Schack complex (see \S \ref{intro1});
\item the relation with abelian deformations of ``quasi-coherent modules'' (see \S \ref{intro2}).
\end{enumerate}

Concerning (1), there are several works in which some notion of twisted deformations is studied by means of a Hochschild type deformation complex.
However, in these works, both the precise notion of twisted deformation and the deformation complex tend to differ from the ones we consider.

Firstly, on the level of deformations, the notions under consideration are often adapted to a geometric picture which involves sheaves of (commutative) algebras, and twisted sheaves are obtained by glueing actual sheaves together in a ``twisted'' way, see the work by Caldararu \cite{caldararu}. For instance, in \cite{yekutieli2}, \cite{yekutieli3}, Yekutieli uses crossed groupoids as the technical tool to capture such twists on the level of deformation groupoids. Note that our approach is philosophically quite different, as it deals with twisted presheaves of algebras as algebraic objects in their own right, living a priori on an arbitrary base category $\uuu$. In the commutative case, the relation with the higher viewpoint on the level of associated abelian categories is detailed in \S \ref{parqchmodule}.

Secondly, the deformation complexes associated to specific twisted deformations in the literature are quite different from the Gerstenhaber-Schack complex. For instance, in smooth geometric setups - often considered in the context of deformation quantization - it is natural to replace Hochschild complexes by subcomplexes of polydifferential operators, in order to arrive at a sheaf of structured complexes which can be globalized  \cite{kontsevich}, \cite{kontsevich1}, \cite{vandenbergh4}, \cite{calaquevandenbergh}, \cite{yekutieli1}, \cite{yekutieli2}, \cite{yekutieli3}. This method is detailed for instance in \cite[Appendix 4]{vandenbergh4}, which also treats the relation with a construction by Hinich \cite{hinich1} - the most refined formality result in terms of higher structure being obtained by Calaque and Van den Bergh in \cite{calaquevandenbergh}. A very different type of globalization of Hochschild complexes on quasi-compact opens was described in \cite{lowensheafhoch}. 

Let us comment a bit further upon the fact that since its appearance in \cite{gerstenhaberschack} \cite{gerstenhaberschack1}, \cite{gerstenhaberschack2}, the Gerstenhaber-Shack complex has not been more intensively used as deformation complex in concrete applications, although it perfectly describes first order twisted deformations of a presheaf (Theorem \ref{proptwist}). The complex itself was discovered well before twisted sheaves were studied in deformation theory, and whereas Gerstenhaber and Schack identified a subcomplex which describes deformations of presheaves, they did not pay attention to the more general deformations described by their complex. Further, the complex is not naturally endowed with a dg Lie algebra stucture, making it at first sight less suited for higher order deformation theory. It turns out that this situation can in fact be remedied, as we show in \cite{dinhvanlowen}: a Gerstenhaber-Schack complex can be defined for an arbitrary twisted presheaf (or prestack), and can naturally be endowed with an $L_{\infty}$-structure. In particular, this generalized Gerstenhaber-Schack complex may be a valid alternative complex in the setup of the work by Bressler et. al. \cite{bressler1}, \cite{bressler2}, \cite{bressler3}, where deformations of algebroid prestacks are studied in various differential geometric setups.

In our opinion, the main advantage of the Gerstenhaber-Schack complex over some of the alternatives is that it does not require any smoothness hypothesis whatsoever, and in the commutative case (in characteristic zero) it possesses a purely algebraic Hodge decomposition which under natural smoothness conditions reduces to the HKR decomposition (Theorem \ref{thmHKR}). In the forthcoming \cite{liulowen}, we use the complex and its Hodge decomposition to compute deformations of some concrete singular schemes.

Concerning (2), several works like \cite{kontsevich1} and \cite{toda} use an intuitive concept of abelian deformations, an actual formalized theory of such deformations was developed in \cite{lowenvandenberghab}, \cite{lowenvandenberghhoch}. In the current paper, we place Toda's construction from \cite{toda} within the realm of \cite{lowenvandenberghab} (see \S \ref{introtoda}). Our categories of quasi-coherent modules, obtained as descent categories of module categories, correspond to the abelian categories considered in \cite{kontsevich1}. For now, their use seems mostly limited to the (non-commutative) algebraic geometric setup. In contrast, in differential setups, larger sheaf categories are often taken as starting point, with no well-behaved category of quasi-coherent sheaves a priori available. 
The main body of \cite{lowen8} deals with abelian deformations of sheaf categories, and Morita theory of such categories was further developed in work by D'Agnolo and Polesello \cite{polesello1}, \cite{polesello2}, \cite{polesello3}. In the future, it would be interesting to understand to what extent the descent method used in this paper may be of use in non-algebraic contexts, to capture notions like DQ modules \cite{kashiwarashapira} or cohesive modules \cite{block}.

\vspace{0,5cm}

\noindent \emph{Acknowledgement.} The authors are greatly indebted to Michel Van den Bergh for the idea of using prestacks in order to capture abelian deformations.

The authors thank the referee for pointing out some additional related literature which helped putting our results in a broader context as described under \S \ref{introbroader}.

\section{The Gerstenhaber-Schack complex and twisted deformations}\label{pargstw}

Let $\aaa$ be a presheaf of algebras on a small category $\uuu$. In \cite{gerstenhaberschack1}, \cite{gerstenhaberschack2}, Gerstenhaber and Schack introduce the Hochschild cohomology of presheaves of algebras to be
$HH^n(\aaa) = \Ext^n_{\Bimod(\aaa)}(\aaa, \aaa)$, computed in the category $\Bimod(\aaa)$ of bimodules over $\aaa$, and they describe an explicit complex $\CC_{\mathrm{GS}}(\aaa)$ computing this cohomology.
In contrast with the situation for algebra deformations studied by Gerstenhaber in \cite{gerstenhaber}, \cite{gerstenhaber1}, there is no perfect match between the second Hochschild cohomology and the natural first order deformations of $\aaa$  as a presheaf of algebras. In order to describe the latter, one has to restrict the attention to a subcomplex of the complex $\CC_{\mathrm{GS}}(\aaa)$ \cite{gerstenhaberschack}.

In \cite{lowenmap}, \cite{lowenvandenberghCCT}, new light is thrown on the situation, showing that the Hochschild cohomology of $\aaa$ is in fact naturally related to deformations of $\aaa$ not as a presheaf, but rather as a \emph{twisted} presheaf of algebras. This is argued in two steps. In step one, the presheaf $\aaa$ is turned into an associated fibered $\uuu$-graded category $\AAA$ by a $k$-linear version of the Grothendieck construction \cite{SGA1}, with a graded Hochschild complex $\CC_{\uuu}(\AAA)$. By a new version of the Cohomology Comparison Theorem \cite{lowenvandenberghCCT}, $\CC_{\uuu}(\aaa)$ also computes $HH^n(\aaa)$. In step two, described in \cite{lowenmap}, the complex $\CC_{\uuu}(\AAA)$ is seen to control the deformation theory of the fibered category $\AAA$, which in turn is equivalent to the deformation theory of $\aaa$ as a twisted presheaf.

An advantage of this approach, is the fact that the complex $\CC_{\uuu}(\AAA)$ (in contrast to $\CC_{\mathrm{GS}}(\aaa)$) is readily seen to be endowed with a $B_{\infty}$-structure controlling the higher order deformation theory.

There is however a disadvantage: in the associated fibered category to a twisted presheaf, the typical three pieces of data determining the algebraic structure (the multiplications of the algebras, the restriction maps, and the twist elements) are mixed into a single algebraic operation (the composition of the category), and by deforming this operation at once, we lose our immediate grip on how the three individual pieces of structure deform.

In contrast, in degree two of the Gerstenhaber-Schack complex we have
$$\CC^2_{\mathrm{GS}}(\aaa) = \CC^{0,2}(\aaa) \oplus \CC^{1,1}(\aaa) \oplus \CC^{2,0}(\aaa),$$
where the three pieces of this decomposition allow us to encode the three pieces of data necessary to describe a first order deformation of $\aaa$ as a twisted presheaf.
In this section, we give a direct proof that $H^2\CC_\mathrm{GS}(\aaa)$ classifies deformations of $\aaa$ as a twisted presheaf of algebras (Proposition \ref{proptwist}). In \cite{dinhvanlowen}, we go a step further defining a Gerstenhaber-Schack complex for arbitrary twisted presheaves (or prestacks), and endowing it with an $L_{\infty}$-structure controlling higher order deformation theory.

Throughout, let $k$ be a commutative ring with unit and let $k[\epsilon]$ be the ring of dual numbers. We always assume that algebras have units, morphisms between algebras preserve units, modules are unital.

\subsection{Hochschild cohomology of algebras}\label{paralg}
In this section, we briefly recall the basic definition of the Hochschild complex of an algebra, as well as its classical relation with first order deformations.

Let $A$ be a $k$-algebra and $M$ an $A$-bimodule. The Hochschild complex $\CC(A, M)$ has $\CC^n(A, M) = \Hom_k(A^{\otimes n}, M)$ and the Hochschild differential $d^n_{\mathrm{Hoch}}\colon \CC^n(A, M) \lra \CC^{n+1}(A,M)$ is given by
$$\begin{aligned}
d^n_{\mathrm{Hoch}}(\phi)(a_n, a_{n-1}, \dots, a_0) = &  a_n\phi(a_{n-1}, \dots, a_0) \\ & + \sum_{i = 0}^{n-1} (-1)^{i+1}\phi(a_n, \dots, a_{n-i}a_{n-i-1}, \dots, a_0) \\
& + (-1)^{n+1} \phi(a_n, \dots, a_1) a_0.
\end{aligned}$$
In particular, $d^0_{\mathrm{Hoch}}\colon \CC^0(A,M) \lra \CC^1(A,M)$ is given by
\begin{equation}\label{eqcentr}
M \lra \Hom_k(A,M),\quad m \longmapsto (a \longmapsto am - ma).
\end{equation}
A cochain $\phi \in \CC^n(A,M)$ is called \emph{normalized} if $\phi(a_{n-1}, \dots, a_0) = 0$ as soon as $a_i = 1$ for some $0 \leq i \leq n-1$. The normalized cochains constitute a subcomplex $\bar{\CC}(A,M)$ of $\CC(A,M)$, for which the inclusion $\bar{\CC}(A,M) \lra \CC(A,M)$ is a quasi-isomorphism.

The Hochschild complex of $A$ is the complex $\CC(A) = \CC(A,A)$. Note that the multiplication $m$ on $A$ is an element $m \in \CC^2(A)$. The Hochschild cohomology of $A$ is $HH^n(A) = H^n\CC(A)$.

\begin{definition}
Let $(A,m)$ be a $k$-algebra. A \emph{first order deformation} of $A$ is given by a $k[\epsilon]$-algebra $(\bar{A}, \bar{m}) = (A[\epsilon], m + m_1 \epsilon)$ where $m_1 \in \CC^2(A)$, such that the unit of $\bar{A}$ is the same as the unit of $A$.

For two deformations $(\bar{A}, \bar{m})$ and $(\bar{A}', \bar{m}')$ of $A$ an \emph{equivalence of deformations} is given by an isomorphism of the form $1 + g\epsilon\colon \bar{A} \lra \bar{A}'$ with $g \in \CC^1(A)$.
\end{definition}

Let $\Def_{\mathrm{alg}}(A)$ denote  the set of first order deformations of $A$ up to equivalence of deformations.

\begin{proposition}
Let $A = (A,m)$ be a $k$-algebra.
\begin{enumerate}
\item For $m_1 \in \CC^2(A)$, we have that $(A[\epsilon], m + m_1\epsilon)$ is a first order deformation of $A$ if and only if $m_1 \in \bar{\CC}^2(A)$ and $d_{\mathrm{Hoch}}(m_1) = 0$.
\item For $m_1, m_1' \in Z^2\bar{\CC}(A)$ and $g_1 \in \CC^1(A)$, we have that $1 + g_1\epsilon$ is an isomorphism between $\bar{A}$ and $\bar{A}'$ if and only if $g_1 \in \bar{\CC}^1(A)$ and $d_{\mathrm{Hoch}}(g_1) = m_1 - m'_1$.
\item We have an isomorphism of sets
$$H^2\bar{\CC}(A) \lra \Def_{\mathrm{alg}}(A),\quad m_1 \longmapsto (A[\epsilon], m + m_1\epsilon).$$
Hence $HH^2(A) \cong H^2\bar{\CC}(A)$ classifies first order deformations of $A$ up to equivalence.
\end{enumerate}
\end{proposition}

The following easy observation will be important later on:

\begin{lemma}\label{lemcentr}
Let $A = (A,m)$ be a $k$-algebra with first order deformation $\bar{A} = (A[\epsilon], m + m_1\epsilon)$. We have $Z(A)\epsilon \subseteq Z(\bar{A})$.
\end{lemma}

\begin{proof}
We have $(m + m_1\epsilon)(a\epsilon, b + b_1\epsilon) = m(a, b)\epsilon$ and $(m + m_1\epsilon)(b + b_1\epsilon, a\epsilon) = m(b, a)\epsilon$ so if $a$ is central in $A$, $a\epsilon$ is central in $\bar{A}$.
\end{proof}

Finally, we discuss the operation of taking opposites. For a $k$-algebra $A = (A,m)$, let $\Mod^r(A)$, $\Mod^l(A)$, $\Bimod(A)$ be the categories of right modules, left modules and bimodules respectively. Let $A^{\op} = (A^{\op}, m^{\op})$ be the opposite algebra, i.e. $A^{\op} = A$ as $k$-modules and $m^{\op}(a,b) = m(b,a)$. Taking the opposite algebra defines a self inverse automorphism of the category $\Alg(k)$, sending a morphism $f\colon A \lra B$ of $k$-algebras to the morphism $f^{\op}\colon A^{\op} \lra B^{\op}$, $a \longmapsto f(a)$. The identity map $1_A\colon A \lra A^{\op}$ is a morphism (and hence an isomorphism) of $k$-algebras if and only if $A$ is commutative.

We have an isomorphism of categories $(-)^{\op}\colon\Mod^r(A) \lra \Mod^l(A^{\op})$, $M \lra M^{\op}$ with the left action on $M^{\op}$ given by the right action on $M$. Similarly, we have isomorphisms $(-)^{\op}\colon \Mod^l(A) \lra \Mod^r(A^{\op})$ and $(-)^{\op}\colon \Bimod(A) \lra \Bimod(A^{\op})$.

The idea of taking opposites can be extended to the Hochschild complex in the following way. For a $k$-algebra $A$ and $A$-bimodule $M$, we have an isomorphism
\begin{equation}\label{eqop}
(-)^{\op}\colon \CC^n(A,M) \lra \CC^n(A^{\op}, M^{\op}),\quad \phi \longmapsto \phi^{\op} = (-1)^{\lambda(n)}\phi^{\sharp}
\end{equation}
with
$$
\phi^{\sharp}(a_{n-1}, \dots, a_1, a_0) = \phi(a_0, a_1, \dots, a_{n-1})
$$
and $$\lambda(n) = \frac{(n-1)(n+2)}{2}.$$
These operations are compatible with the Hochschild differential, whence they define an isomorphism of complexes $(-)^{\op}\colon \CC(A, M) \lra \CC(A^{\op}, M^{\op})$ resulting in isomorphisms $HH^n(A,M) \cong HH^n(A^{\op}, M^{\op})$ and $HH^n(A) \cong HH^n(A^{\op})$ for all $n$.

\begin{example}\label{exop}
Let $a$, $b$, $c \in A$. For $n = 0$, we have
$$(-)^{\op}\colon M \lra M^{\op},\quad m \longmapsto -m.$$
For $n=1$, $\phi \in \Hom_{k}(A,M)$, we have $\phi^{\op}(a) = \phi(a)$. \\
\noindent For $n = 2$, $\phi \in \CC^2(A,M)$, we have $\phi^{\op}(a,b) = \phi(b,a)$.\\
\noindent For $n = 3$, $\phi \in \CC^3(A,M)$, we have $\phi^{\op}(a,b,c) = - \phi(c, b,a)$.
\end{example}

In particular, for a $k$-algebra $A = (A, m)$ and morphism of $k$-algebras $f\colon A \lra B$, $m^{\op}\colon A^{\op} \otimes A^{\op} \lra A^{\op}$ and $f^{\op}\colon A^{\op} \lra B^{\op}$ are as defined before.
We further obtain the following almost tautological relation with deformations:

\begin{proposition}\label{propopalg}
Consider a $k$-algebra $(A,m)$ with $m_1 \in Z^2\bar{\CC}(A)$ and consider $m_1^{\op} \in Z^2\bar{\CC}(A^{\op})$. The corresponding deformations $(A[\epsilon], m + m_1\epsilon)$ and $(A^{\op}[\epsilon], m^{\op} + m_1^{\op}\epsilon)$ are such that $A^{\op}[\epsilon] = A[\epsilon]^{\op}$.
\end{proposition}

\subsection{Simplicial cohomology of presheaves}\label{parsimp}

In this section we introduce a simplicial cohomology complex associated to two arbitrary presheaves of $k$-modules. If we take the first presheaf equal to the constant presheaf $k$, we recover the usual simplicial cohomology of the second presheaf.

Let $\uuu$ be a small category and let $\fff = (\fff, f)$ and $\GGG = (\GGG, g)$ be presheaves of $k$-modules with restriction maps $f^u\colon \fff(U) \lra \fff(V)$ and $g^u\colon \GGG(U) \lra \GGG(V)$ for $u\colon V \lra U$ in $\uuu$.

Let $\nnn(\uuu)$ be the simplicial nerve of $\uuu$. Our standard notation for a $p$-simplex $\sigma \in \nnn_p(\uuu)$ is
\begin{equation}\label{eqsigma0}
\sigma = (\xymatrix{ {d\sigma = U_0} \ar[r]^-{u_1} & {U_{1}} \ar[r]^-{u_2} & {\cdots} \ar[r]^-{u_{p-1}} & {U_{p-1}} \ar[r]^-{u_p} & {U_p = c\sigma}}).
\end{equation}
If confusion can arise, we write $U_i = U^{\sigma}_i$ and $u_i = u^{\sigma}_i$ instead.
For $\sigma \in \nnn_p(\uuu)$, we obtain a map
$$f^{\sigma} = f^{u_p \dots u_2 u_1}\colon \fff(U_p) \lra \fff(U_0).$$
As part of the simplicial structure of $\nnn(\uuu)$, we have maps
$$\partial_i\colon \nnn_{p+1}(\uuu) \lra \nnn_p(\uuu),\quad \sigma \longmapsto \partial_i \sigma$$
for $i = 0, 1, \dots, p+1$. For $\sigma = (\xymatrix{ {U_0} \ar[r]^-{u_1} & {U_{1}} \ar[r]^-{u_2} & {\cdots} \ar[r]^-{u_{p}} & {U_{p}} \ar[r]^-{u_{p+1}} & {U_{p+1}}})$, we have
\begin{gather*}
\partial_{p+1}\sigma = (\xymatrix{ {U_0} \ar[r]^-{u_1} & {U_{1}} \ar[r]^-{u_2} & {\cdots} \ar[r]^-{u_{p-1}} & {U_{p-1}} \ar[r]^-{u_p} & {U_p}}),\\
\partial_{0}\sigma = (\xymatrix{ {U_1} \ar[r]^-{u_2} & {U_{2}} \ar[r]^-{u_3} & {\cdots} \ar[r]^-{u_{p}} & {U_{p}} \ar[r]^-{u_{p+1}} & {U_{p+1}}}),
\end{gather*}
and
$$\partial_i\sigma = (\xymatrix{ {U_{0}} \ar[r]^-{u_1} & {\cdots} \ar[r] & {U_{i-1}} \ar[r]^{u_{i+1} u_{i}} & {U_{i+1}} \ar[r] & {\cdots} \ar[r]^-{u_{p+1}} & {U_{p+1}}})$$
for $i = 1, \dots, p$.
Now put $$\CC^p_{\mathrm{simp}}(\GGG, \fff) = \CC^p(\GGG, \fff) = \prod_{\sigma \in \nnn_p(\uuu)} \Hom_k(\GGG(c\sigma), \fff(d\sigma)).$$
Every $\partial_i$ gives rise to a map
$$d_i\colon \CC^{p}(\GGG, \fff) \lra \CC^{p+1}(\GGG, \fff)$$
which we now describe.

Consider $\phi = (\phi^{\tau})_{\tau} \in \CC^{p}(\GGG, \fff)$.
We are to define $d_i\phi = (d_i\phi^{\sigma})_{\sigma} \in \CC^{p+1}(\GGG, \fff)$, so consider fixed $\sigma$.

For $i = 1, \dots, p$, note that $c \sigma = c \partial_i \sigma$ and $d \sigma = d \partial_i \sigma$ so we can put
$$d_i\phi^{\sigma} = \phi^{\partial_i \sigma}.$$

Next we define $d_{0}\phi^{\sigma}$ as the following composition:
$$\xymatrix{ {\GGG(U_{p+1})} \ar[r]^-{\phi^{\partial_{0} \sigma}} & {\fff(U_1)} \ar[r]^-{f^{u_1}} & {\fff(U_0).}}$$
Finally we define
$d_{p+1}\phi^{\sigma}$
as the following composition:
$$\xymatrix{ {\GGG(U_{p+1})} \ar[r]^-{g^{u_{p+1}}} & {\GGG(U_p)} \ar[r]^-{\phi^{\partial_{p+1}\sigma}} & {\fff(U_0).}}$$
We define $$d_{\mathrm{simp}} = \sum_{i = 0}^{p+1} (-1)^i d_i \colon \CC^{p}(\GGG, \fff) \lra \CC^{p+1}(\GGG, \fff).$$

\begin{lemma}
$d_{\mathrm{simp}}^2 = 0$.
\end{lemma}

\begin{example}
Take $\GGG = k$ the constant presheaf. Then we obtain
$$\CC^p_{\mathrm{simp}}(\fff) = \CC^p(k, \fff) = \prod_{\sigma \in \nnn_p(\uuu)} \fff(d \sigma).$$
The cohomology of this complex is called the \emph{simplicial presheaf cohomology} of $\fff$, and is denoted by
$$H^p(\uuu, \fff) = H^p \CC_{\mathrm{simp}}(\fff).$$
\end{example}

When $p\geq 1$ ,the simplex $\sigma$ in \eqref{eqsigma0} is said to be \textit{degenerate} if $u_i=1_{U_i}$ for some $i$. A $p$-cochain $\phi = (\phi^{\sigma})_{\sigma} \in \CC^{p}(\GGG, \fff)$ is said to be \textit{reduced} if $\phi^\sigma=0$ whenever $\sigma$ is degenerate. All $0$-cochains are reduced by convention. It is easy to see that $d_\mathrm{simp}$ preserves reduced cochains and hence we obtain a subcomplex $\CC'^\bullet(\GGG, \fff)\subseteq \CC^\bullet(\GGG, \fff)$ consisting of reduced cochains.

Equip $\CC^{p}(\GGG, \fff)$ with a filtration $\cdots\subseteq F^p\CC^{p}\subseteq F^{p-1}\CC^{p}\subseteq\cdots\subseteq F^0\CC^{p}=\CC^{p}(\GGG, \fff)$ by setting
\[
F^j\CC^{p}=\{\phi = (\phi^{\sigma})_{\sigma} \in \CC^{p}(\GGG, \fff)\mid\text{$\phi^\sigma=0$ whenever $u^\sigma_i=1_{U^\sigma_i}$ for some $i\leq j$}\}.
\]
Since $d_\mathrm{simp}(F^j\CC^{p})\subseteq F^j\CC^{p+1}$, $F^j\CC^{\bullet}$ is  a complex. There is a sequence of complexes
\[
\cdots\hookrightarrow F^j\CC^{\bullet}\hookrightarrow F^{j-1}\CC^{\bullet}\hookrightarrow \cdots\hookrightarrow F^0\CC^{\bullet}.
\]
For $p\geq j-1\geq 0$, define $\chi\colon \nnn_p(\uuu)\lra\nnn_{p+1}(\uuu)$ by
\[
\chi(\sigma)=(\xymatrix@C=6mm{ U_0 \ar[r]^-{u_1} &  \cdots \ar[r] & U_{j-2} \ar[r]^-{u_{j-1}} & U_{j-1} \ar[r]^-{1_{U_{j-1}}} & U_{j-1} \ar[r]^-{u_j} & U_{j} \ar[r] & \cdots  \ar[r]^-{u_p} & U_p}).
\]

\begin{lemma}\label{lemincl}
The inclusion $l\colon F^j\CC^{\bullet}\hookrightarrow F^{j-1}\CC^{\bullet}$ is a quasi-isomorphism.
\end{lemma}

\begin{proof}
First of all, let us prove $H^p(l)$ is injective. Suppose $\phi\in Z^p(F^j\CC^{\bullet})\cap B^p(F^{j-1}\CC^{\bullet})$, say $\phi=d_\mathrm{simp}(\psi)$ for some $\psi\in F^{j-1}\CC^{p-1}$. We want to show $\phi\in B^p(F^j\CC^{\bullet})$.

If $p\leq j$, then $\psi\in F^{j-1}\CC^{p-1}=F^{j}\CC^{p-1}$, so $\phi\in B^p(F^j\CC^{\bullet})$.

If $p> j$, define $\psi_1=(\psi^{\chi(\zeta)})_\zeta\in \CC^{p-2}(\GGG, \fff)$. It is obvious that $d_\mathrm{simp}(\psi_1)\in F^{j-1}\CC^{p-1}$. Note that for any $\sigma\in\nnn_{p-1}(\uuu)$ with $u^\sigma_j=1_{U_j^\sigma}$,
\[
\psi^{\sigma}-(-1)^jd_\mathrm{simp}(\psi_1)^{\sigma}=-(-1)^jd_\mathrm{simp}(\psi)^{\chi(\sigma)}=-(-1)^j\phi^{\chi(\sigma)}=0.
\]
So $\psi-(-1)^jd_\mathrm{simp}(\psi_1)\in F^{j}\CC^{p-1}$ and $\phi=d_\mathrm{simp}(\psi-(-1)^jd_\mathrm{simp}(\psi_1))\in B^p(F^j\CC^{\bullet})$.

Next we will prove $H^p(l)$ is surjective. It suffices to show that for any $\theta\in Z^p(F^{j-1}\CC^{\bullet})$ there exists $\theta'\in Z^p(F^{j}\CC^{\bullet})$ such that $\theta-\theta'\in B^p(F^{j-1}\CC^{\bullet})$.

If $p<j$, then since $F^{j-1}\CC^{p}=F^{j}\CC^{p}$, we take $\theta'=\theta$.

If $p\geq j$, define $\theta_1=(\theta^{\chi(\sigma)})_\sigma\in \CC^{p-1}(\GGG, \fff)$. By the same argument as above, we have $\theta_1\in F^{j-1}\CC^{p-1}$ and $\theta-(-1)^jd_\mathrm{simp}(\theta_1)\in Z^{p}(F^{j}\CC^\bullet)$. Therefore we take $\theta'=\theta-(-1)^jd_\mathrm{simp}(\theta_1)$.
\end{proof}

Notice that for every fixed $p$, the filtration $F^\bullet\CC^{p}$ is stationary. By Lemma \ref{lemincl}, we have

\begin{proposition}\label{propreduced}
The inclusion $\CC'^\bullet(\GGG, \fff)\hookrightarrow \CC^\bullet(\GGG, \fff)$ is a quasi-isomorphism.
\end{proposition}

\begin{definition}
Let $(\fff, f)$ be a presheaf of $k$-modules. A \emph{first order deformation } of $\fff$ is given by a presheaf of $k[\epsilon]$-modules
$$(\bar{\fff}, \bar{f}) = (\fff[\epsilon], f + f_1 \epsilon)$$
where $f_1\in\CC_\mathrm{simp}^1(\fff,\fff)$.

For two deformations $(\bar{\fff}, \bar{f})$ and $(\bar{\fff}', \bar{f}')$ an \emph{equivalence of deformations} is given by an isomorphism of the form $g=1+g_1\epsilon$ where $g_1\in\CC_\mathrm{simp}^0(\fff,\fff)$.
\end{definition}

\begin{proposition}
Let $\fff = (\fff, f)$ be a presheaf of $k$-modules.
\begin{enumerate}
\item For $f_1 \in \CC_\mathrm{simp}^1(\fff,\fff)$, we have that $(\fff[\epsilon], f + f_1\epsilon)$ is a first order deformation of $\fff$ if and only if $f_1$ is reduced and $d_{\mathrm{simp}}(f_1) = 0$.
\item For $f_1$, $f_1'\in Z^1\CC'_\mathrm{simp}(\fff,\fff)$, and $g_1\in \CC_\mathrm{simp}^0(\fff,\fff)$, $g=1+g_1\epsilon$ is an isomorphism between $\bar{\fff}$ and $\bar{\fff'}$ if and only if $d_\mathrm{simp}(g_1)=f_1-f_1'$.
\item The first cohomology group $H^1(\fff, \fff)$ classifies first order deformations of $\fff$ up to equivalence.
\end{enumerate}
\end{proposition}

\begin{proof}
(1) $(\fff[\epsilon], \bar{f}=f + f_1\epsilon)$ is a first order deformation of $\fff$ if and only if for any $v\colon W \lra V$, $u\colon V\lra U$ in $\uuu$ the equations $\bar{f}^{uv}=\bar{f}^v\bar{f}^u$ and $\bar{f}^{1_U}=1_{\fff[\epsilon](U)}$ hold. These equations are equivalent to $d_{\mathrm{simp}}(f_1) = 0$ and $f_1^{1_U}=0$ respectively.

(2) $g=1+g_1\epsilon$ is an isomorphism between $\bar{\fff}$ and $\bar{\fff'}$ if and only if for any $u\colon V \lra U$ the equation $\bar{f}'^ug^U=g^V\bar{f}^u$ holds. This equation is equivalent to $d_\mathrm{simp}(g_1)=f_1-f_1'$.

(3) This follows from (1), (2) and Proposition \ref{propreduced}.
\end{proof}

\subsection{Simplicial presheaf complex}\label{paracechlikecomp}

In this section, we introduce a simplicial complex of presheaves associated to an arbitrary presheaf of algebras.

Let $(\aaa, m, f)$ be a presheaf on $\uuu$. For $U \in \uuu$, we obtain an induced presheaf $\aaa|_U$ on $\uuu/U$ with $\aaa|_U(V \lra U) = \aaa(V)$. We identify $\sigma \in \nnn_{n}(\uuu/U)$ with the object $d\sigma\lra U$ in $\uuu/U$ by composing all morphisms of $\sigma$. For $n \geq 0$, we define a presheaf $\aaa^n = (\aaa^{n}, m^{n}, \rho^{n})$ on $\uuu$ by $$\aaa^{n}(U) = \prod_{\sigma \in \nnn_{n}(\uuu/U)} \aaa|_U(\sigma)$$
endowed with the product algebra structure $m^{n, U}$.
For $u\colon V \lra U$, there is a natural functor $\uuu/V \lra \uuu/U$ sending $v\colon W \lra V$ to $uv\colon W \lra U$ and hence a natural map
$\nnn_{n}(\uuu/V) \lra \nnn_{n}(\uuu/U)$, $\sigma \longmapsto u\sigma$. For $\sigma \in \nnn_n(\uuu/V)$, we further have $\aaa|_V(\sigma) = \aaa|_U(u\sigma)$.
Thus, we obtain restriction maps
$$\rho^{n,u}\colon \aaa^{n}(U) \lra \aaa^{n}(V),\quad (a^{\tau})_{\tau} \longmapsto (a^{u\sigma})_{\sigma}$$
finishing the definition of $\aaa^n$.

Next, we define a morphism of presheaves $\varphi^{n}\colon \aaa^n \lra \aaa^{n+1}$. The maps
\[
\partial_i\colon \nnn_{n+1}(\uuu/U) \lra \nnn_n(\uuu/U)
\]
give rise to the desired maps
\begin{align*}
\varphi^{n,U}\colon \prod_{\tau \in \nnn_{n}(\uuu/U)} \aaa|_U(\tau) &\lra \prod_{\sigma \in \nnn_{n+1}(\uuu/U)} \aaa|_U(\sigma)\\
(a^{\tau})_{\tau} &\longmapsto \biggl(u_1^{\sigma*}(a^{\partial_0\sigma})+\sum_{i=1}^{n+1}(-1)^ia^{\partial_i\sigma}\biggr)_{\sigma}.
\end{align*}

Clearly, $(\aaa^\bullet,\varphi^\bullet)$ is a complex. In particular, $\varphi^{0, U}\colon \aaa^0(U) \lra \aaa^1(U)$ is defined in the following way. For
\[
a = (a^u)_u \in  \prod_{u\colon V \lra U} \aaa(V) \text{ and } \sigma = (\xymatrix{{V_0} \ar[r]^-{v} & {V_1} \ar[r]^-{u} & U}),
\]
we have
$$\varphi^{0,U}(a)^{\sigma} = v^*(a^u)-a^{uv}.$$

\begin{lemma}
$\Kern(\varphi^0) \cong \aaa$.
\end{lemma}

\begin{proof}
Let $\varepsilon\colon \aaa\lra \aaa^0$ be the obvious map induced by $f$. It is easy to verify $\Beeld(\varepsilon)=\Kern(\varphi^0)$, and $\varepsilon$ is injective since $f^{1_U}$ is identity map.
\end{proof}

\subsection{The Gerstenhaber-Schack complex}\label{paragersschackcomp}
Let $\uuu$ be a small category and let
$$\aaa\colon \uuu^{\op} \lra \Alg(k),\quad U \longmapsto \aaa(U)$$ be a presheaf of $k$-algebras with restriction maps denoted by $f^u\colon \aaa(U) \lra \aaa(V)$ for $u\colon V \lra U$ in $\uuu$.
For $n \geq 0$, we consider the presheaf $\aaa^{\otimes n}$ on $\uuu$ with $\aaa^{\otimes n}(U) = \aaa(U)^{\otimes n}$ (in particular $\aaa^{\otimes 0}(U) = k$).
For $\sigma \in \nnn_p(\uuu)$, we consider $\aaa(d \sigma)$ as an $\aaa(c \sigma)$-bimodule through $f^{\sigma}\colon \aaa(c \sigma) \lra \aaa(d \sigma)$.

In \cite{gerstenhaberschack1}, Gerstenhaber and Schack define the complex $\CC_\mathrm{GS}(\aaa)$ which combines Hochschild and simplicial complexes.
For $p$, $q \geq 0$, we put
$$\CC^{p,q}(\aaa) = \prod_{\sigma \in \nnn_p(\uuu)} \Hom_k(\aaa(c \sigma)^{\otimes q}, \aaa(d \sigma)).$$
We obtain a double complex in the following way. For fixed $q$, we have
$$\CC^{\bullet, q} =  \CC_{\mathrm{simp}}(\aaa^{\otimes q}, \aaa)$$
which we endow with the (horizontal) simplicial differential $d_{\mathrm{simp}}$ from \S \ref{parsimp}. For fixed $p$, we have
$$\CC^{p, \bullet} =  \prod_{\sigma \in \nnn_p(\uuu)} \CC(\aaa(c \sigma), \aaa(d \sigma))$$
which we endow with the (vertical) product Hochschild differential $d_{\mathrm{Hoch}}$.
One can easily check $d_{\mathrm{Hoch}}d_{\mathrm{simp}}=d_{\mathrm{simp}}d_{\mathrm{Hoch}}$. As usual, we endow the associated total complex $\CC_\mathrm{GS}^\bullet(\aaa)$ with differentials $d^n_\mathrm{GS}\colon\CC_\mathrm{GS}^{n}(\aaa)\lra \CC_\mathrm{GS}^{n+1}(\aaa)$ as follows
$$d^n_\mathrm{GS}  =(-1)^{n+1}d_{\mathrm{simp}} +d_{\mathrm{Hoch}} .$$

Recall that a cochain $\phi=(\phi^\sigma)_\sigma\in\CC^{p,q}(\aaa)$
is called \textit{normalized} if for any $p$-simplex $\sigma$, $\phi^\sigma$ is normalized, and it is called \textit{reduced} if $\phi^\sigma=0$ whenever $\sigma$ is degenerate. The normalized cochains form a subcomplex $\bar{\CC}_\mathrm{GS}^\bullet(\aaa)$ of $\CC_\mathrm{GS}^\bullet(\aaa)$, the \emph{normalized Hochschild complex} of $\aaa$. The normalized reduced cochains form a further subcomplex $\bar{\CC}_\mathrm{GS}'^\bullet(\aaa)$ of $\bar{\CC}_\mathrm{GS}^\bullet(\aaa)$, the \emph{normalized reduced Hochschild complex} of $\aaa$.

\begin{proposition}
The inclusions $\bar{\CC}_\mathrm{GS}'^\bullet(\aaa) \lra \bar{\CC}_\mathrm{GS}^\bullet(\aaa)\lra \CC_\mathrm{GS}^\bullet(\aaa)$ are quasi-isomorphisms.
\end{proposition}

\begin{proof}
For the inclusion $\bar{\CC}_\mathrm{GS}^\bullet(\aaa)\lra \CC_\mathrm{GS}^\bullet(\aaa)$, the proof can be found in \cite{gerstenhaberschack1}. For the inclusion $\bar{\CC}_\mathrm{GS}'^\bullet(\aaa) \lra \bar{\CC}_\mathrm{GS}^\bullet(\aaa)$, using Lemma \ref{lemincl}, each row of the double complex $\bar{\CC}'^{\bullet, \bullet}(\aaa)$ is quasi-isomorphic to the corresponding row of $\bar{\CC}^{\bullet, \bullet}(\aaa)$, so the spectral sequences (filtrations by rows) of $\bar{\CC}'^{\bullet, \bullet}(\aaa)$ and $\bar{\CC}^{\bullet, \bullet}(\aaa)$ are isomorphic.
\end{proof}

For the complex $\CC_\mathrm{GS}(\aaa)$, we obtain a subcomplex $\CC_\mathrm{tGS}(\aaa)$ by eliminating the bottom row. Precisely, we have $\CC_\mathrm{tGS}(\aaa)^n = \oplus_{p + q = n, q \geq 1} \CC^{p,q}(\aaa)$ with the induced differential $d_\mathrm{GS}$. We call the resulting complex the \emph{truncated Hochschild complex}. Similarly, we obtain the \emph{truncated normalized Hochschild complex} $\bar{\CC}_\mathrm{tGS}(\aaa)$ and the \emph{truncated normalized reduced Hochschild complex} $\bar{\CC}'_\mathrm{tGS}(\aaa)$.

There is an exact sequence of complexes
\begin{equation}\label{seqhoch}
0 \lra \CC_\mathrm{tGS}(\aaa) \lra \CC_\mathrm{GS}(\aaa) \lra \CC_{\mathrm{simp}}(\aaa) \lra 0
\end{equation}
where the cokernel is obtained from the projection onto the bottom row $\CC^{p, 0}(\aaa)$ of the double complex $\CC^{p,q}(\aaa)$.

\begin{proposition}\label{propsplit}
Let $\aaa$ be a presheaf of commutative $k$-algebras. The bottom row $\CC^{\bullet,0}(\aaa) = \CC_{\mathrm{simp}}(\aaa)$ is a subcomplex of $\CC_\mathrm{GS}(\aaa)$ and the canonical projection $\CC_\mathrm{GS}(\aaa) \lra \CC_\mathrm{tGS}(\aaa)$ is a morphism of complexes, canonically splitting the sequence \eqref{seqhoch} as a sequence of complexes whence we have $$\CC_\mathrm{GS}(\aaa) =\CC_\mathrm{tGS}(\aaa) \oplus \CC_{\mathrm{simp}}(\aaa).$$
Similarly, we have
$$\bar{\CC}'_\mathrm{GS}(\aaa) = \bar{\CC}'_\mathrm{tGS}(\aaa) \oplus \CC'_{\mathrm{simp}}(\aaa).$$
\end{proposition}

\begin{proof}
In the double complex $\CC^{p,q}(\aaa)$, the maps from $\CC^{p,0}(\aaa)$ to $\CC^{p,1}(\aaa)$ $$\prod_{\sigma \in \nnn_p(\uuu)} \aaa(d\sigma) \lra \prod_{\sigma \in \nnn_p(\uuu)}\Hom_k(\aaa(c\sigma), \aaa(d\sigma))$$ are determined by the maps $d^0_{\mathrm{Hoch}}\colon \aaa(d\sigma) \lra \Hom_k(\aaa(c\sigma), \aaa(d\sigma))$ described in \eqref{eqcentr}. By commutativity of $\aaa(d\sigma)$, we have $d^0_{\mathrm{Hoch}} = 0$. The claim follows.
\end{proof}

Finally, we extend the operation of taking opposite cochains from \S \ref{paralg} to the setting of twisted presheaves. Let $\aaa$ be a twisted presheaf with opposite presheaf $\aaa^{\op}$ with $\aaa^{\op}(U) = \aaa(U)^{\op}$ and induced opposite restriction maps. For $p,q \geq 0$, we have an isomorphism
$$(-)^{\op}\colon \CC^{p,q}(\aaa) \lra \CC^{p,q}(\aaa^{\op}),\quad (\phi^{\sigma})_{\sigma} \longmapsto \bigl((\phi^{\sigma})^{\op}\bigr)_{\sigma}.$$
These isomorphisms are compatible with the Hochschild and simplicial differentials, whence they give rise to isomorphisms
\begin{equation}\label{gsop}
(-)^{\op}\colon \CC_\mathrm{GS}(\aaa) \lra \CC_\mathrm{GS}(\aaa^{\op})
\end{equation}
and $(-)^{\op}\colon \CC_\mathrm{tGS}(\aaa) \lra \CC_\mathrm{tGS}(\aaa^{\op})$, with the resulting cohomology isomorphisms.

\subsection{The situation in low degrees}\label{pardeglow}
Let $\aaa$ be a {presheaf} of $k$-algebras as before. Let us list the ingredients in $\CC^n(\aaa) = \CC^n_{\mathrm{GS}}(\aaa)$ for $n = 0$, $1$, $2$, $3$ and describe the differentials.

When $n=0$, we have $$\CC^0(\aaa) = \CC^{0,0}(\aaa) =  \prod_{U_0 \in \uuu}\aaa(U_0).$$

When $n=1$, $\CC^1(\aaa) = \CC^{0,1}(\aaa) \oplus \CC^{1,0}(\aaa)$ with
$$\CC^{0,1}(\aaa) = \prod_{U_0 \in \uuu} \Hom_k(\aaa(U_0), \aaa(U_0))$$
and
$$\CC^{1,0}(\aaa) = \prod_{u_1\colon U_0 \lra U_1} \aaa(U_0).$$

When $n=2$, we have $\CC^2(\aaa) = \CC^{0,2}(\aaa) \oplus \CC^{1,1}(\aaa) \oplus \CC^{2,0}(\aaa)$ with
\begin{gather*}
\CC^{0,2}(\aaa) = \prod_{U_0 \in \uuu}  \Hom_k(\aaa(U_0)^{\otimes 2}, \aaa(U_0)),\\
\CC^{1,1}(\aaa) = \prod_{u_1\colon U_0 \lra U_1} \Hom_k(\aaa(U_1), \aaa(U_0)),
\end{gather*}
and \[\CC^{2,0}(\aaa) = \prod_{\substack{u_1\colon U_0\lra U_1\\u_2\colon U_1\lra U_2}} \aaa(U_0).\]

When $n=3$, $\CC^3(\aaa) = \CC^{0,3}(\aaa) \oplus \CC^{1,2}(\aaa) \oplus \CC^{2,1}(\aaa)\oplus \CC^{3,0}(\aaa)$ with
\begin{gather*}
\CC^{0,3}(\aaa) = \prod_{U_0 \in \uuu}  \Hom_k(\aaa(U_0)^{\otimes 3}, \aaa(U_0)),\\
\CC^{1,2}(\aaa) = \prod_{u_1\colon U_0 \lra U_1} \Hom_k(\aaa(U_1)^{\otimes 2}, \aaa(U_0)),\\
\CC^{2,1}(\aaa) = \prod_{\substack{u_1\colon U_0\lra U_1\\u_2\colon U_1\lra U_2}} \Hom_k(\aaa(U_2),\aaa(U_0)),
\end{gather*}
and
\[\CC^{3,0}(\aaa) = \prod_{\substack{u_1\colon U_0\lra U_1\\u_2\colon U_1\lra U_2\\u_3\colon U_2\lra U_3}} \aaa(U_0).\]

The differentials are given by
\begin{gather*}
d^0_{\mathrm{GS}}(\theta_{0,0})=(d_{\mathrm{Hoch}}(\theta_{0,0}), -d_{\mathrm{simp}}(\theta_{0,0})),\\
d^1_{\mathrm{GS}}(\theta_{0,1},\theta_{1,0})=(d_{\mathrm{Hoch}}(\theta_{0,1}),d_{\mathrm{simp}}(\theta_{0,1})+d_{\mathrm{Hoch}}(\theta_{1,0}), d_{\mathrm{simp}}(\theta_{1,0})),\\
\begin{aligned}
d^2_{\mathrm{GS}}(\theta_{0,2},\theta_{1,1},\theta_{2,0})&=(d_{\mathrm{Hoch}}(\theta_{0,2}),-d_{\mathrm{simp}}(\theta_{0,2})+d_{\mathrm{Hoch}}(\theta_{1,1}),\\ &\phantom{\:=\:}-d_{\mathrm{simp}}(\theta_{1,1})+d_{\mathrm{Hoch}}(\theta_{2,0}),-d_{\mathrm{simp}}(\theta_{2,0})).
\end{aligned}
\end{gather*}

\subsection{Twisted presheaves}\label{partwistedpresheaf} Twisted presheaves of algebras are natural variants of presheaves, which only satisfy the composition law for restriction maps up to conjugation by invertible elements.

\begin{definition}\label{deftwisted}
A \emph{twisted presheaf of $k$-algebras} $\aaa=(\aaa, m, f, c, z)$ on $\uuu$ consists of the following data:
\begin{itemize}
\item for every $U \in \uuu$ a $k$-algebra $(\aaa(U), m^U)$;
\item for every $u\colon V \lra U$ in $\uuu$ a morphism of  $k$-algebras $f^u = u^{\ast}\colon \aaa(U) \lra \aaa(V)$;
\item for every pair $v\colon W \lra V$, $u\colon V  \lra U$ an invertible element $c^{u,v}\in \aaa(W) $ such that for every $a \in \aaa(U)$ we have
$$c^{u,v} v^*u^*(a) = (uv)^*(a) c^{u,v}.$$
\item for every $U \in \uuu$ an invertible element $z^U\in \aaa(U)$ such that for every $a \in \aaa(U)$ we have
$$z^U a = f^{1_U}(a) z^U.$$
\end{itemize}
These data further satisfy the following compatibility conditions:
\begin{gather}
c^{u,vw}c^{v,w}=c^{uv,w}{w^*}(c^{u,v}),\label{eq:twistcocycle}\\
c^{u,1_V}z^V=1,\quad c^{1_U,u}{u^*}(z^U)=1\notag
\end{gather}
for every triple $w\colon T\lra W$, $v\colon W\lra V$, $u\colon V\lra U$.
\end{definition}

\begin{remarks}
\begin{enumerate}
\item  An ordinary presheaf $(\aaa, m, f)$ is naturally interpreted as a twisted presheaf with $c^{u,v}=1$ and $z^U=1$.
\item A twisted presheaf $(\aaa, m, f, c, z)$ with $c^{u,v}$ and $z^U$ being central for all $(u,v)$ and $U$  is said to \textit{have central twists}. Such a twisted presheaf has an underlying ordinary presheaf which is denoted by  $\underline{\aaa}=(\aaa, m, f)$.
\item A twisted presheaf $\aaa$ can be seen as a prestack of $k$-linear categories (see Definition \ref{defpseudo}) by viewing the algebras $\aaa(U)$ as one-object categories.
\end{enumerate}
\end{remarks}

\begin{example}\label{exop1}
Consider a twisted presheaf $\aaa = (\aaa, m, f, c, z)$. We obtain the \emph{opposite twisted presheaf} $\aaa^{\op} = (\aaa^{\op}, m^{\op}, f^{\op}, c^{-1}, z^{-1})$ with $\aaa^{\op}(U) = \aaa(U)$ as a $k$-modules and $(m^{\op})^U = (m^U)^{\op}$, $(f^{\op})^u = (f^u)^{\op}$, $(c^{-1})^{u,v} = (c^{u,v})^{-1}$ and $(z^{-1})^U = (z^U)^{-1}$. One checks that this indeed defines a twisted presheaf. We have $(\aaa^{\op})^{\op} = \aaa$.
\end{example}

\begin{definition}
Consider twisted presheaves $(\aaa, m, f, c, z)$ and $(\aaa', m', f', c', z')$ on $\uuu$. A \emph{morphism of twisted presheaves} $(g,\tau)\colon\aaa \lra \aaa'$ consists of the following data:
\begin{itemize}
\item for  each $U\in \uuu$,  a $k$-linear map $g^U\colon\aaa(U)\lra\aaa'(U)$;
\item for each $u\colon V\lra U$, an invertible element $\tau^u\in\aaa'(V),$
\end{itemize}
These data further satisfy the following compatibility conditions: for any $v\colon W\lra V$, $u\colon V\lra U$ and $a\in\aaa(U)$,
\begin{enumerate}
\item  $g^Um^U=m'^U(g^U\otimes g^U)$;
\item  $g^U(1)=1'$;
\item  $m'^V(g^Vu^*(a),\tau^u)=m'^V(\tau^u,u'^* g^U(a))$;
\item  $m'^W(\tau^{uv},c'^{u,v})=m'^W(g^W(c^{u,v}),\tau^v,v'^*(\tau^u))$;
\item  $m'^U(\tau^{1_U},z'^U)=g^U(z^U)$.
\end{enumerate}
\end{definition}

Morphisms can be composed, and every twisted presheaf $\aaa$ has an identity morphism $1_{\aaa}$ with $g^U = 1_{\aaa(U)}$ and $\tau^u = 1 \in \aaa(V)$.
A morphism of twisted presheaves $(g,\tau)\colon\aaa \lra \aaa'$ is an isomorphism if and only if $g^U$ is an isomorphism of $k$-algebras for each $U$.

\begin{example}\label{exdualpre}
Let $\aaa$ be a twisted presheaf with opposite twisted presheaf $\aaa^{\op}$ as in Example \ref{exop1}. Putting $g^{U} = 1_{\aaa(U)}$ and $\tau^u = 1 \in \aaa(V)$ defines a morphism (whence, an isomorphism) of twisted presheaves $(g, \tau)\colon \aaa \lra \aaa^{\op}$ if and only if $\aaa$ is a twisted presheaf of commutative algebras with $c =1$ and $z = 1$, i.e. it is a presheaf of commutative algebras.
\end{example}

It is proved in \cite{lowenmap} that any twisted presheaf $(\aaa, m, f, c, z)$ is isomorphic to one of the form $(\aaa', m', f', c', 1)$. In this paper, we always work with twisted presheaves with $z=1$ and we write them as $(\aaa, m, f, c)$.

\begin{definition}\label{deftwistedpresheaf} (see Def 3.24 in \cite{lowenmap})
Let $(\aaa, m, f, c)$ be a twisted presheaf of $k$-algebras.
\begin{enumerate}
\item
\begin{enumerate}
\item A \emph{first order twisted deformation} of $\aaa$ is given by a twisted presheaf
$$(\bar{\aaa}, \bar{m}, \bar{f}, \bar{c}) = (\aaa[\epsilon], m + m_1\epsilon, f + f_1\epsilon, c + c_1\epsilon)$$
of $k[\epsilon]$-algebras such that $(\bar{\aaa}(U), \bar{m}^U)$ is a first order deformation of $(\aaa(U), m^U)$ for all $U$, where $(m_1, f_1, c_1)\in \CC^{0,2}(\aaa) \oplus \CC^{1,1}(\aaa) \oplus \CC^{2,0}(\aaa)$.

\item If $(\aaa, m, f)$ is a presheaf (i.e. we have $c = 1$), a \emph{first order presheaf deformation} of $\aaa$ is a twisted deformation which is itself a presheaf, i.e. $c_1 = 0$.
\end{enumerate}

\item
\begin{enumerate}
\item For two twisted deformations $(\bar{\aaa}, \bar{m}, \bar{f}, \bar{c})$ and $(\bar{\aaa}', \bar{m}', \bar{f}', \bar{c}')$ an \emph{equivalence of twisted deformations} is given by an isomorphism of the form $(g,\tau)=(1+g_1\epsilon,1+\tau_1\epsilon)$ where $(g_1,\tau_1)\in\CC^{0,1}(\aaa)\oplus\CC^{1,0}(\aaa)$.

\item If $(\bar{\aaa}, \bar{m}, \bar{f})$ and $(\bar{\aaa}', \bar{m}', \bar{f}')$ are two presheaf deformations of $(\aaa, m, f)$, then an \emph{equivalence of presheaf deformations} is given by  an isomorphism of presheaves of the form $g = 1 + g_1\epsilon$ for $g_1 \in \CC^{0,1}(\aaa)$.
\end{enumerate}
\end{enumerate}
\end{definition}

Let $\Def_{\mathrm{tw}}(\aaa)$ denote the set of twisted deformations of a twisted presheaf $\aaa$ up to equivalence of twisted deformations, and let $\Def_{\mathrm{pr}}(\aaa)$ denote the set of presheaf deformations of a presheaf $\aaa$ up to equivalence of presheaf deformations.

\begin{theorem}\label{proptwist}
Let $\aaa = (\aaa, m, f)$ be a presheaf of $k$-algebras with Gerstenhaber-Schack complex $\CC_\mathrm{GS}(\aaa)$.
\begin{enumerate}
\item
\begin{enumerate}
\item For $(m_1, f_1, c_1)$ in $\CC^{0,2}(\aaa) \oplus \CC^{1,1}(\aaa) \oplus \CC^{2,0}(\aaa)$, we have that $(\aaa[\epsilon], \bar{m}=m + m_1\epsilon,\bar{f}= f + f_1\epsilon, \bar{c}=1+c_1\epsilon)$ is a first order twisted deformation of $\aaa$ if and only if $(m_1,f_1,c_1)\in\bar{\CC}_\mathrm{GS}'^2(\aaa)$ and $d_\mathrm{GS}(m_1, f_1, c_1) = 0$.
\item For $(m_1, f_1)$ in $\CC^{0,2}(\aaa) \oplus \CC^{1,1}(\aaa)$, we have that $(\aaa[\epsilon], \bar{m}=m + m_1\epsilon,\bar{f}= f + f_1\epsilon)$ is a first order presheaf deformation of $\aaa$ if and only if $(m_1,f_1)\in\bar{\CC}_\mathrm{tGS}'^2(\aaa)$ and $d_\mathrm{GS}(m_1, f_1) = 0$.
\end{enumerate}
\item
\begin{enumerate}
\item For $(m_1, f_1, c_1)$ and $(m_1', f_1', c_1')$ in $Z^2\bar{\CC}_\mathrm{GS}'(\aaa)$, and $(g_1, -\tau_1) \in \CC^{0,1}(\aaa) \oplus \CC^{1,0}(\aaa)$, we have that $(g,\tau)=(1 + g_1\epsilon, 1 + \tau_1\epsilon)$ is an isomorphism between $\bar{\aaa}$ and $\bar{\aaa}'$ if and only if $(g_1,-\tau_1)\in \bar{\CC}_\mathrm{GS}'^1(\aaa)$ and $d_\mathrm{GS}(g_1, -\tau_1) = (m_1, f_1, c_1)-(m_1', f_1', c_1') $.
\item For $(m_1, f_1)$ and $(m_1', f_1')$ in $Z^2\bar{\CC}_\mathrm{tGS}'(\aaa)$, and $g_1 \in \CC^{0,1}(\aaa)$, we have that $g = 1 + g_1\epsilon$ is an isomorphism of presheaves between $\bar{\aaa}$ and $\bar{\aaa}'$ if and only if $g_1 \in \bar{\CC}_\mathrm{tGS}'^1(\aaa)$ and $d_\mathrm{GS}(g_1) = (m_1, f_1)-(m_1', f_1')$.
\end{enumerate}
\item
\begin{enumerate}
\item We have an isomorphism of sets
\begin{equation}\label{eqclass}
H^2\bar{\CC}_\mathrm{GS}'(\aaa) \lra \Def_{\mathrm{tw}}(\aaa).
\end{equation}
Hence, the second cohomology group $HH^2(\aaa) \cong H^2\bar{\CC}_\mathrm{GS}'(\aaa)$ classifies first order twisted deformations of $\aaa$ up to equivalence.
\item We have an isomorphism of sets
$$H^2\bar{\CC}_\mathrm{tGS}'(\aaa) \lra \Def_{\mathrm{pr}}(\aaa).$$
Hence, the second cohomology group $H^2\bar{\CC}_\mathrm{tGS}'(\aaa)$ classifies first order presheaf deformations of $\aaa$ up to equivalence.
\end{enumerate}
\end{enumerate}
\end{theorem}

\begin{proof}
We prove the (a) part. The (b) part is an easier variant and can be found in \cite{gerstenhaberschack1}.

(1) For each $U\in\uuu$, the associativity of $\bar{m}^U$ is equivalent to $d_{\mathrm{Hoch}}(m_1)^U=0$.  The unity condition $\bar{m}^U(a,1)=a=\bar{m}^U(1,a)$ holds if and only if $m_1(a,1)=m_1(1,a)=0$ for all $a\in\aaa(U)$.

We have $\bar{z}^U=1$ for all $U\in\uuu$. So for $1_U\colon U\lra U$ and $a\in\aaa(U)$, the condition $\bar{m}(\bar{z}^U,a)=\bar{m}(\bar{f}^{1_U}(a),\bar{z}^U)$  if and only if $f_1^{1_U}(a)=0$, equivalently $f_1^{\sigma}=0$ for any degenerate 1-simplex $\sigma$.

For any $v\colon W\lra V$, $u\colon V\lra U$, and $a$, $b\in\aaa(U)$, $\bar{f}^u$ is a morphism of algebras if $\bar{f}^u\bar{m}(a,b)=\bar{m}(\bar{f}^u(a),\bar{f}^u(b))$ and $\bar{f}^u(1)=1$, which are in turn equivalent to $d_{\mathrm{simp}}(m_1)-d_{\mathrm{Hoch}}(f_1)=0$ and $f^u_1(1)=0$. Moreover, $\bar{c}^{u,v} \bar{f}^v\bar{f}^u(a) = \bar{f}^{uv}(a) \bar{c}^{u,v}$ is equivalent to $ -d_{\mathrm{simp}}(f_1)+d_{\mathrm{Hoch}}(c_1)=0$.

The compatibility condition of $\bar{c}$ is satisfied if and only if $d_{\mathrm{simp}}(c_1)=0$ and $c_1^\sigma=0$ for any degenerate 2-simplex $\sigma$.

Recall the differential $d_\mathrm{GS}$ given in \S\ref{pardeglow}. These facts yield that $(m_1,f_1,c_1)$ gives rise to a twisted deformation if and only if it is a normalized reduced cocycle.

(2) The map $g=1+g_1\epsilon$ is a morphism of algebras if and only if $d_\mathrm{Hoch}(g_1)=m_1-m_1'$ and $g_1^U(1)=0$ for all $U$.

The equation $m'^V(g^Vu^*(a),\tau^u)=m'^V(\tau^u,u'^* g^U(a))$ is equivalent to $d_\mathrm{simp}(g_1)-d_\mathrm{Hoch}(\tau_1)=f_1-f_1'$, while  $m'^W(\tau^{uv},c'^{u,v})=m'^W(g^W(c^{u,v}),\tau^v,v'^*(\tau^u))$ is equivalent to $-d_\mathrm{simp}(\tau_1)=c_1-c_1'$, and the equation $m'^U(\tau^{1_U},z'^U)=g^U(z^U)$ is equivalent to $\tau_1^{1_U}=0$.

Thus $(g,\tau)=(1+g_1\epsilon,1+\tau_1\epsilon)$ is an isomorphism between $\aaa$ and $\aaa'$ if and only if $(g_1,-\tau_1)$ is a normalized reduced cochain and
\begin{align*}
d_{\mathrm{GS}}(g_1,-\tau_1)&=(d_\mathrm{Hoch}(g_1), d_{\mathrm{simp}}(g_1)+d_{\mathrm{Hoch}}(-\tau_1),d_{\mathrm{simp}}(-\tau_1))\\
&=(m_1,f_1,c_1)-(m'_1,f'_1,c'_1).
\end{align*}

(3) This follows from (1) and (2).
\end{proof}

\begin{remarks}
\begin{enumerate}
\item In \cite{lowenmap} (based upon \cite{lowenvandenberghCCT}), Proposition \ref{proptwist} (3) is obtained by making use of the Hochschild complex $\CC_{\uuu}(\tilde{\aaa})$ of the $\uuu$-graded category $\tilde{\aaa}$ associated to $\aaa$.
\item In contrast to $\CC_{\uuu}(\tilde{\aaa})$ from (1), the Gerstenhaber-Schack complex $\CC_\mathrm{GS}(\aaa)$ is not endowed with a $B_{\infty}$-algebra  or dg Lie algebra structure. 
\item In \cite{dinhvanlowen}, we introduce a Gerstenhaber-Schack complex for an arbitrary twisted presheaf (or prestack) $\aaa$. The necessary modification of the current complex is substantial, as to the simplicial and Hochschild components of the differential one needs to add a series of higher components in general. In \cite{dinhvanlowen}, we further show that this Gerstenhaber-Schack complex inherits an $L_{\infty}$-structure from the dg Lie structure present upon $\CC_{\uuu}(\tilde{\aaa})$, controlling the higher order twisted deformation theory.
\end{enumerate}
\end{remarks}

The relation between the isomorphism $(-)^{\op}\colon \CC_\mathrm{GS}(\aaa) \lra \CC_\mathrm{GS}(\aaa^{\op})$ and twisted deformations is as follows (see Example \ref{exop1}).

\begin{proposition}\label{propopdef}
Let $(\aaa, m, f)$ be a presheaf with opposite presheaf $(\aaa^{\op}, m^{\ast}, f^{\ast})$. For $\phi = (m_1, f_1, c_1) \in Z^2\bar{\CC}'_\mathrm{GS}(\aaa)$, let $\phi^{\op} = (m_1^{\ast}, f_1^{\ast}, -c_1) \in Z^2\bar{\CC}'_\mathrm{GS}(\aaa^{\op})$ be the associated opposite cocycle. The associated first order twisted deformations $(\aaa[\epsilon], m + m_1\epsilon, f + f_1\epsilon, 1 + c_1\epsilon)$ and $(\aaa^{\op}[\epsilon], m^{\ast} + m_1^{\ast}\epsilon, f^{\ast} + f_1^{\ast}\epsilon, 1 - c_1\epsilon)$ are opposite twisted presheaves.
\end{proposition}

\begin{proposition}\label{centralkey}
Let $(\aaa, m, f)$ be a presheaf of commutative $k$-algebras with $$\bar{\CC}'_\mathrm{GS}(\aaa) = \bar{\CC}'_\mathrm{tGS}(\aaa) \oplus \CC'_{\mathrm{simp}}(\aaa).$$
Consider $((m_1, f_1), c_1) \in Z^2\bar{\CC}'_\mathrm{GS}(\aaa)$ with corresponding first order twisted deformation
$$\bar{\aaa} = (\aaa[\epsilon], m + m_1\epsilon, f + f_1\epsilon, 1 + c_1\epsilon).$$
Then $\bar{\aaa}$ is a twisted presheaf with central twists and the underlying presheaf $\underline{\bar{\aaa}}$ corresponds to $(m_1, f_1) \in Z^2\bar{\CC}'_\mathrm{tGS}(\aaa)$.
\end{proposition}

\begin{proof}
This follows from Lemma \ref{lemcentr}.
\end{proof}

\section{From Hodge to HKR decomposition}\label{secHodge-HKR}\label{parparhodgeHKR}

Throughout this section, we assume that $\Q\subseteq k$. In \cite{HKR}, Hochschild, Kostant and Rosenberg proved the famous HKR theorem for a regular affine algebra $A$, which states that the anti-symmetrization map
$$\wedge^n_A \mathrm{Der}(A) \lra HH^n(A)$$
is an isomorphism.

Let $X$ be a quasi-compact separated scheme. The Hochschild cohomology of $X$ is defined as
\begin{equation}\label{HHscheme}
HH^n(X) = \Ext^n_{X \times X}(\ooo_{\Delta}, \ooo_{\Delta})
\end{equation}
for the diagonal $\Delta\colon X \lra X \times X$ and $\ooo_{\Delta} = \Delta_{\ast} \ooo_X$.
For a smooth scheme, we have the following HKR decomposition into sheaf cohomologies:
\begin{equation}\label{HKRscheme}
HH^n(X) \cong \bigoplus_{p + q = n} H^p(X, \wedge^q \ttt_X).
\end{equation}
Various proofs of this statement exist in the literature. In \cite{swan}, the decomposition is proved for smooth quasi-projective schemes, and in \cite{yekutieli} it is proved for smooth separated finite type schemes (under a somewhat weaker condition than $\Q \subseteq k$).

On the other hand, in \cite{gerstenhaberschack}, the authors relate the decomposition at the right hand side of \eqref{HKRscheme} to their Hochschild cohomology of presheaves of algebras in the following way.
Let $\uuu$ be an open affine cover of $X$ closed under intersections and consider the restriction $\ooo_X|_{\uuu}$ of the structure sheaf to this cover. If $X$ is a smooth complex projective variety, they prove the existence of a decomposition
\begin{equation} \label{HKRGS}
HH^n(\ooo_X|_{\uuu}) \cong \bigoplus_{p + q = n} H^p(X, \wedge^q \ttt_X).
\end{equation}
The decomposition \eqref{HKRGS} results from a combination of a purely algebraic Hodge decomposition for presheaves of commutative algebras with the classical HKR theorem.
According to \cite[Thm. 7.8.1]{lowenvandenberghhoch}, we further have
$$HH^n(X) \cong HH^n(\ooo_X|_{\uuu}),$$
so in combination with \eqref{HKRGS} this yields another proof of \eqref{HKRscheme}. This route to a proof of \eqref{HKRscheme} was also noted in \cite{kontsevich2}, in the context of Homological Mirror symmetry.

In \S \ref{parHodge}, we recall the Hodge decomposition for the Gerstenhaber-Schack complex of a presheaf $\aaa$ of commutative algebras from \cite{gerstenhaberschack}.
Under the additional assumption that the algebras are essentially of finite type and smooth (FS) and the restriction maps are flat epimorphisms of rings (FE), in \S \ref{parHodgeHKR} we deduce a HKR decomposition
\begin{equation}\label{eqpreHKR}
HH^n(\aaa) \cong \bigoplus_{p + q = n} H^p(\uuu, \wedge^q \ttt_{\aaa})
\end{equation}

where the cohomology groups on the right hand side are simplicial presheaf cohomologies, and $\ttt_{\aaa}$ is the tangent presheaf of $\aaa$. 
If $X$ is a quasi-compact semi-separated scheme with semi-separating cover $\uuu$, and $\aaa = \ooo_X|_{\uuu}$, the simplicial presheaf cohomologies at the right hand side of \eqref{eqpreHKR} are isomorphic to \v{C}ech cohomologies, as is detailed in \S \ref{parcechsimp} for further use later on, and to sheaf cohomologies by Leray's theorem.
For $X$ furthermore smooth, \eqref{eqpreHKR} thus translates into \eqref{HKRGS}, leading to a proof, in the separated case, of \eqref{HKRscheme}.

Note that even in the absence of condition (FS), the Hodge decomposition is a highly valuable tool in computing Hochschild cohomology, see the forthcoming paper \cite{liulowen} in which it is used to compute deformations of some particular singular schemes.

\subsection{Localization}\label{parloc}

In this section, we recall some facts on localization of abelian categories which will be used later on. For more details, see for instance \cite{popescu}, \cite{verschoren1}.

A fully faithful functor $\iota\colon \ccc' \lra \ccc$ between abelian categories is called a \emph{localization} (of $\ccc$) if $\iota$ has an exact left adjoint $a\colon \ccc \lra \ccc'$. The localization is called \emph{strict} if it is the inclusion of a full subcategory closed under isomorphisms. Two localizations of $\ccc$ are called \emph{equivalent} if they have the same essential image in $\ccc$. Every localization is equivalent to precisely one strict localization, namely the inclusion of its essential image.

Let $f\colon A \lra B$ be a morphism of rings. The morphism $f$ is called an \emph{epimorphism of rings} if it is an epimorphism in the category of non-commutative rings. The morphism $f$ gives rise to an induced restriction of scalars functor $(-)_A \colon\Mod(B)\lra\Mod(A)$ with left adjoint $-\otimes_AB\colon \Mod(A) \lra \Mod(B)$. The morphism $f$ is called \emph{(right) flat} if $- \otimes_A B\colon \Mod(A) \lra \Mod(B)$ is exact.

\begin{lemma}\label{lemloc}
The following are equivalent:
\begin{enumerate}
\item $(-)_A\colon\Mod(B)\lra\Mod(A)$ is fully faithful.
\item The multiplication $B \otimes_A B \lra B$ is an isomorphism.
\item $f$ is an epimorphism of rings.
\end{enumerate}
The following are equivalent:
\begin{enumerate}
\item $(-)_A\colon\Mod(B)\lra\Mod(A)$ is a localization.
\item $f$ is a (right) flat epimorphism of rings.
\end{enumerate}
\end{lemma}

Now consider two localizations $\iota_1\colon \ccc_1 \lra \ccc$ and $\iota_2\colon \ccc_2 \lra \ccc$ with respective left adjoints $a_1$ and $a_2$ and with $q_1 = \iota_1 a_1$ and $q_2 = \iota_2 a_2$.
The localizations are called \emph{compatible} if $q_1 q_2 \cong q_2 q_1$.
If the two compatible localizations are strict, the intersection defines a new strict localization $\iota\colon \ccc_1 \cap \ccc_2 \lra \ccc$ with left adjoint $a$ and $\iota a = q = q_1 q_2 \cong q_2 q_1$.
Consequently, we can define the \emph{intersection} of arbitrary localizations up to equivalence of localizations as the intersection of the two strict representatives.

\begin{lemma}\label{lemcomp}
Let  $f\colon A \lra B$, $f_1\colon A \lra B_1$ and $f_2\colon A \lra B_2$ be flat epimorphisms of rings.
The following are equivalent:
\begin{enumerate}
\item $(-)_{A,1}\colon \Mod(B_1) \lra \Mod(A)$ and $(-)_{A,2}\colon \Mod(B_2) \lra \Mod(A)$ are compatible localizations with $\Mod(B_1) \cap \Mod(B_2) \cong \Mod(B)$.
\item There are isomorphisms of $A$-bimodules $$B \cong B_1 \otimes_A B_2 \cong B_2 \otimes_A B_1.$$
\end{enumerate}
\end{lemma}

\subsection{The Hodge decomposition}\label{parHodge}
In this section, following \cite{gerstenhaberschack1}, \cite{gerstenhaberschackhodge}, we describe the Hodge decomposition of the Gerstenhaber-Schack complex of a presheaf of commutative algebras.

The Hodge decomposition is based upon the existence, for each $n \geq 1$, of a collection of pairwise orthogonal idempotents $e_n(r)$ for $1 \leq r \leq n$ in the group algebra $\Q S_n$ of the $n$-th symmetric group $S_n$. These idempotents satisfy $\sum_{r = 1}^n e_n(r) = 1$. We further put $e_n(0) = 0$ for $n \geq 1$, $e_n(r) = 0$ for $r > n$ and $e_0(0) = 1 \in \Q$.

Let $A$ be a commutative $k$-algebra and $M$ a symmetric bimodule. We obtain a subcomplex $\CC(A,M)_r \subseteq \CC(A,M)$ with $\CC^n(A,M)_r = \CC(A,M)e_n(r)$ and a Hodge decomposition of complexes
\begin{equation}\label{eqhodgealg}
\CC(A, M) = \oplus_{r \in \N} \CC(A,M)_r.
\end{equation}
Taking cohomology yields the Hodge decomposition of the Hochschild cohomology of a commutative algebra. We refer the reader to \cite{gerstenhaberschackhodge} for the details.

Now let $\uuu$ be a small category and $\aaa\colon \uuu^{\op} \lra \mathsf{CommAlg}(k)$ a presheaf of commutative $k$-algebras.
Recall the double complex given in \S \ref{paragersschackcomp},
\[
\CC^{p,q}(\aaa)=\prod_{\sigma\in\nnn_p(\uuu)}\Hom(\aaa(c\sigma)^{\otimes q}, \aaa(d\sigma)).
\]
Since $\aaa(U)$ is a commutative algebra for any $U\in\uuu$, there is the Hodge decomposition
\[
\Hom(\aaa(c\sigma)^{\otimes q}, \aaa(d\sigma))=\bigoplus_{r=0}^q\Hom(\aaa(c\sigma)^{\otimes q}, \aaa(d\sigma))_r.
\]

It follows that there is a decomposition of every $\CC^{p,q}(\aaa)$. Since the vertical differentials $d_{\mathrm{Hoch}}$ are pointwise, we have $d_{\mathrm{Hoch}}(\CC^{p,q}(\aaa)_r)\subseteq \CC^{p,q+1}(\aaa)_r$. In order to induce a decomposition of the total complex, it suffices to prove that the horizontal differentials $d_{\mathrm{simp}}$ also preserve the Hodge decomposition. When $\uuu$ is a poset, the proof can be found in \cite{gerstenhaberschack2}. In the general case, the argument is similar. Let us give a brief explanation.

Let $\phi=(\phi^\tau)_\tau\in \CC^{p,q}(\aaa)_r$. Then $\phi^\tau\in \Hom(\aaa(c\tau)^{\otimes q}, \aaa(d\tau))_r$ for all $\tau$, equivalently, $\phi^\tau e_q(r)=\phi^\tau$. Thus following the notations in \S \ref{parsimp},
\begin{align*}
(d_0\phi)^\sigma e_q(r)&=f^{u_1}\phi^{\partial_0\sigma}e_q(r)=f^{u_1}\phi^{\partial_0\sigma}=(d_0\phi)^\sigma,\\
(d_i\phi)^\sigma e_q(r)&=\phi^{\partial_i\sigma}e_q(r)=\phi^{\partial_i\sigma}=(d_i\phi)^\sigma,\quad i=1,\dots,p,\\
(d_{p+1}\phi)^\sigma e_q(r)&=\phi^{\partial_{p+1}\sigma}(f^{\otimes q})^{u_{p+1}}e_q(r)=\phi^{\partial_{p+1}\sigma}e_q(r)(f^{\otimes q})^{u_{p+1}}=(d_{p+1}\phi)^\sigma,
\end{align*}
and so $d_{\mathrm{simp}}(\CC^{p,q}(\aaa)_r)\subseteq \CC^{p+1,q}(\aaa)_r$.

Therefore, as a complex, $\CC_{\mathrm{GS}}(\aaa)$ admits a Hodge decomposition
\begin{equation}\label{eqhodge}
\CC_{\mathrm{GS}}(\aaa) = \oplus_{r \in \N} \CC_{\mathrm{GS}}(\aaa)_r.
\end{equation}
Taking cohomology yields the Hodge decomposition for $HH(\aaa)$.

Note that we have $\CC_{\mathrm{GS}}(\aaa)_0 = \CC_{\mathrm{simp}}(\aaa)$ and $\oplus_{r \in \N_0} \CC_{\mathrm{GS}}(\aaa)_r = \CC_{\mathrm{tGS}}(\aaa)$.

\subsection{From Hodge to simplicial HKR decomposition}\label{parHodgeHKR}
In \cite[\S 28]{gerstenhaberschack}, the authors combine the Hodge decomposition \eqref{eqhodge} with the classical HKR theorem \cite{HKR} in order to obtain the HKR decomposition \eqref{HKRGS} for smooth complex projective varieties.

In this section, we present the argument in a somewhat abstracted setting. First, we recall the HKR theorem following \cite[Theorem 9.4.7]{weibelbook}. Let $A$ be a commutative $k$-algebra.
Let $I$ be the kernel of the multiplication map $A \otimes_k A \lra A$ and $I^2$ the image of $I \otimes_k I \lra I$. We obtain the $A$-module of differentials $\Omega_A = I/I^2$ and its dual $\ttt_A = \Hom_A(\Omega_A, A) \cong \mathrm{Der}(A)$. For the Hochschild homology $HH_n(A)$ and cohomology $HH^n(A)$, there are natural $k$-linear anti-symmetrization morphisms $\Omega^n_A = \wedge^n \Omega_A \lra HH_n(A)$ and $\wedge^n \ttt_A \lra HH^n(A)$. If $A$ is an essentially of finite type, smooth algebra, these morphisms are isomorphisms, the so called HKR isomorphisms. Let $M$ be a symmetric $A$-bimodule. In the Hodge decomposition \eqref{eqhodgealg}, we further have $H^n\CC(A,M)_r = 0$ unless $n = r$, and $H^n\CC(A,M) = H^n\CC(A,M)_n$, i.e. the cohomology is concentrated in the ``top component''.

Now let $\uuu$ be a small category and $\aaa\colon \uuu^{\op} \lra \mathsf{CommAlg}(k)$ a presheaf of commutative algebras. We make the following two additional assumptions:
\begin{enumerate}
\item[(FS)] The algebra $\aaa(U)$ is an essentially of finite type, smooth algebra for every $U \in \uuu$.
\item[(FE)] The restriction map $u^{\ast}\colon \aaa(U) \lra \aaa(V)$ is a flat epimorphism of rings for every $u\colon V \lra U$ in $\uuu$.
\end{enumerate}

We obtain the \emph{presheaf of differentials} $\Omega_{\aaa}$ of $\aaa$ on $\uuu$ with $\Omega_{\aaa}(U) = \Omega_{\aaa(U)}$.
By (FE), every restriction map $u^{\ast}\colon\aaa(U) \lra \aaa(V)$ gives rise to a canonical isomorphism $\aaa(V) \otimes_{\aaa(U)} \Omega_{\aaa(U)} \cong \Omega_{\aaa(V)}$, and thus to a restriction map $\ttt_{\aaa(U)} \lra \ttt_{\aaa(V)}$. We thus obtain the \emph{tangent presheaf} $\ttt_{\aaa}$ of $\aaa$-modules on $\uuu$ with $\ttt_{\aaa}(U) = \ttt_{\aaa(U)}$.

In this section, we prove the following theorem:

\begin{theorem}\label{thmHKR}
Let $\uuu$ be a small category and $\aaa\colon \uuu^{\op} \lra \mathsf{CommAlg}(k)$ a presheaf of commutative algebras satisfying (FS) and (FE). There are canonical isomorphisms
$$HH^n(\aaa)= \bigoplus_{r=0}^n H^n\CC_{\mathrm{GS}}(\aaa)_r\cong \bigoplus_{p+q=n} H^p(\uuu,\wedge^q\ttt_{\aaa})$$
where the cohomologies on the right hand side are simplicial presheaf cohomologies.
\end{theorem}
Let us apply Theorem \ref{thmHKR} to smooth schemes, generalizing \eqref{HKRGS} and \eqref{HKRscheme} (in case $\Q \subseteq k$).
Recall that a scheme is called \emph{semi-separated} if the intersection of two affine open subsets is affine (that is, if the diagonal $\Delta\colon X \lra X \times X$ is an affine morphism).
A \emph{semi-separating cover} of a scheme $X$ is an open affine cover which is closed under finite intersections. A scheme $X$ is semi-separated if and only if it has a semi-separating cover.

\begin{corollary}\label{thmnewscheme}
Let $X$ be a smooth scheme with a semi-separating cover $\uuu$. Let $\ooo_X|_{\uuu}$ and $\ttt_X|_{\uuu}$ be the restrictions to $\uuu$ of the structure sheaf $\ooo_X$ and the tangent sheaf $\ttt_X$ of $X$ respectively.
There are canonical isomorphisms
$$\begin{aligned}\label{HKRnewscheme}
HH^n(\ooo_X|_{\uuu}) & = \bigoplus_{r=0}^n H^n\CC_{\mathrm{GS}}(\ooo_X|_{\uuu})_r \cong \bigoplus_{p+q=n} H^p(\uuu,\wedge^q\ttt_X|_{\uuu}) \\ & \cong \bigoplus_{p+q=n} \check{H}^p(\uuu, \wedge^q \ttt_X) \cong \bigoplus_{p+q=n} H^p(X,\wedge^q\ttt_X).
\end{aligned}$$
where the third, fourth and fifth expressions contain simplicial,  \v{C}ech and sheaf cohomologies respectively.
\end{corollary}

\begin{proof}
Put $\aaa = \ooo_X|_{\uuu}$. The condition (FS) obviously holds since $X$ is smooth. For $V\subseteq U$ in $\uuu$, the open immersion $(V,\ooo_X|_V)\lra (U,\ooo_X|_U)$ between affine opens corresponds to a flat epimorphism $\ooo_X(U)\lra\ooo_X(V)$ so (FE) also holds. Since $\uuu$ consists of affine opens, we have $\ttt_X|_{\uuu} \cong \ttt_{\aaa}$. Hence, Theorem \ref{thmHKR} applies, yielding the first line. Further, we have isomorphisms $H^p(\uuu, \wedge^q\ttt_X|_{\uuu}) \cong \check{H}^p(\uuu, \wedge^q\ttt_X)$ (see \S \ref{parcechsimp}), and $\check{H}^p(\uuu, \wedge^q \ttt_X) \cong H^p(X, \wedge^q \ttt_X)$ by Leray's theorem.
\end{proof}

\begin{corollary}
Let $X$ be a smooth quasi-compact separated scheme. There is an isomorphism
$$HH^n(X) \cong \bigoplus_{p+q=n} H^p(X,\wedge^q\ttt_X)$$
where $HH^n(X)$ is as defined in \eqref{HHscheme} and the cohomologies on the right hand side are sheaf cohomologies.
\end{corollary}

\begin{proof}
This follows from Corollary \ref{thmnewscheme} and the isomorphism $HH^n(X) \cong HH^n(\ooo_X|_{\uuu})$ proved in \cite[Thm. 7.8.1]{lowenvandenberghhoch}.
\end{proof}

\begin{remark}
Corollary \ref{thmnewscheme} shows that for a smooth, semi-separated scheme, the expression $HH^n(\ooo_X|_{\uuu})$ is independent of the choice of a semi-separating cover $\uuu$. For $X$ quasi-compact semi-separated (and not necessarily smooth), we actually have $HH^n(\ooo_X|_{\uuu}) \cong HH_{\mathrm{ab}}(\Qch(X))$ by \cite[Thm.\ 7.2.2 and Cor.\ 7.7.2]{lowenvandenberghhoch}, where $HH_{\mathrm{ab}}(\Qch(X))$ is the Hochschild cohomology of the abelian category $\Qch(X)$ (the proof in loc.\ cit.\ is given for a separated scheme, but only makes use of semi-separatedness). This confirms Gerstenhaber and Schack's point of view, identifying $HH^n(\ooo_X|_{\uuu})$ as  the correct cohomology one should associate to $X$.
\end{remark}

\begin{proof}[Proof of Theorem \ref{thmHKR}]
Let us consider the spectral sequence ${}^{\mathrm{I}}E^{p,q}_{\bullet,r}$ determined by the column filtration of $\bar{\CC}'^{\bullet,\bullet}(\aaa)_r$. According to \cite[Theorem 9.1.8]{weibelbook}, (FE) and Lemma \ref{lemloc}, we have
\begin{equation}\label{eq:Hoch1}
\aaa(V)\otimes_{\aaa(U)}HH_n( \aaa(U))\cong HH_n(\aaa(V)).
\end{equation}
By (FS), $H^q(\aaa(U), \aaa(V))_r=0$ unless $q=r$. So ${}^{\mathrm{I}}E^{p,q}_{1,r}$ is concentrated in the $r$th row, and
\begin{align*}
{}^{\mathrm{I}}E^{p,r}_{1,r}&=\sideset{}{'}\prod_{\sigma\in\nnn_p(\uuu)}H^r\bar{\CC}(\aaa(c\sigma), \aaa(d\sigma))_r\\
&=\sideset{}{'}\prod_{\sigma\in\nnn_p(\uuu)}H^r(\aaa(c\sigma), \aaa(d\sigma))\\
&\overset{\scriptscriptstyle(1)}{\cong}\sideset{}{'}\prod_{\sigma\in\nnn_p(\uuu)}\Hom_{\aaa(c\sigma)}(HH_r(\aaa(c\sigma)), \aaa(d\sigma))\\
&\cong\sideset{}{'}\prod_{\sigma\in\nnn_p(\uuu)}\Hom_{\aaa(d\sigma)}(\aaa(d\sigma)\otimes_{\aaa(c\sigma)}HH_r(\aaa(c\sigma)), \aaa(d\sigma))\\
&\overset{\scriptscriptstyle(2)}{\cong}\sideset{}{'}\prod_{\sigma\in\nnn_p(\uuu)}\Hom_{\aaa(d\sigma)}(HH_r(\aaa(d\sigma)), \aaa(d\sigma))\\
&\cong\sideset{}{'}\prod_{\sigma\in\nnn_p(\uuu)}\Hom_{\aaa(d\sigma)}(\Omega^r_{\aaa(d\sigma)}, \aaa(d\sigma))\\
&\cong\sideset{}{'}\prod_{\sigma\in\nnn_p(\uuu)}\wedge^r\ttt_{\aaa}(d\sigma)\\
&=\CC'^p_{\mathrm{simp}}(\wedge^r\ttt_{\aaa}).
\end{align*}
Here $\prod'*$ means the submodule of $\prod*$ consisting of reduced cochains. The above isomorphisms are pointwise, wherein (1) follows by \cite[Lemma 4.1]{HKR}, (2) follows from \eqref{eq:Hoch1}. We denote their composition by $T$.

Since $B^r\bar{\CC}(\aaa(c\sigma), \aaa(d\sigma))_r=0$, any class $\theta^\sigma\in H^r\bar{\CC}(\aaa(c\sigma), \aaa(d\sigma))_r$ is a normalized cocycle on the nose, satisfying $\theta^\sigma e_r(r)=\theta^\sigma$. We have in turn that $\theta^\sigma$ factors uniquely through $(f^{\sigma})^{\otimes r}\colon \aaa(c\sigma)^{\otimes r}\lra\aaa(d\sigma)^{\otimes r}$. Namely, there exists a unique $\Theta^\sigma$ such that the diagram
\[
\xymatrix{
\aaa(c\sigma)^{\otimes r} \ar[d]_-{(f^{\sigma})^{\otimes r}}\ar[r]^-{\theta^\sigma} & \aaa(d\sigma)\\
\aaa(d\sigma)^{\otimes r} \ar[ur]_-{\Theta^\sigma}
}
\]
is commutative. Since $\theta^\sigma e_r(r)=\theta^\sigma$ implies $\theta^\sigma=\Theta^\sigma e_r(r)(f^\sigma)^{\otimes r}$, by the uniqueness of $\Theta^\sigma$ we obtain $\Theta^\sigma e_r(r)=\Theta^\sigma$. Thus $\Theta^\sigma$ can be viewed as an anti-symmetric multi-derivation, i.e.\ $\Theta^\sigma\in \wedge^r\ttt_{\aaa}(d\sigma)$, and we have $T(\theta)=(\Theta^\sigma)_\sigma$. Consequently, the map $\CC'^p_{\mathrm{simp}}(\wedge^r\ttt_{\aaa})\lra \CC'^{p+1}_{\mathrm{simp}}(\wedge^r\ttt_{\aaa})$ induced by the differential ${}^{\mathrm{I}}E^{p,r}_{1,r}\lra {}^{\mathrm{I}}E^{p+1,r}_{1,r}$ is the same as $d_{\mathrm{simp}}$, and so
\[
{}^{\mathrm{I}}E^{p,q}_{2,r}=
\begin{cases}
H^p(\uuu,\wedge^r\ttt_{\aaa}), & q=r,\\
0, & q\neq r.
\end{cases}
\]
It follows that the spectral sequence collapses at $E_2$ stage. Hence $H^n\CC_{\mathrm{GS}}(\aaa)_r\cong {}^{\mathrm{I}}E^{n-r,r}_{2,r}$ and
\begin{equation}\label{HKRnew}
H^n\CC_{\mathrm{GS}}(\aaa)= \bigoplus_{r=0}^n H^n\CC_{\mathrm{GS}}(\aaa)_r\cong \bigoplus_{p+q=n} H^p(\uuu,\wedge^q\ttt_{\aaa}).\qedhere
\end{equation}
\end{proof}

According to the above proof and by the Hodge decomposition of $H^n\CC_{\mathrm{GS}}(\aaa)$, any GS cohomology class $\CCC_{\mathrm{GS}}$ admits a normalized reduced \emph{decomposable} representative $(\theta_{0,n},\theta_{1,n-1},\dots, \theta_{n,0})$ in the sense that $\theta_{n-r,r}\in \bar{\CC}^{n-r,r}(\aaa)_r$ and $\theta_{n-r,r}$ is reduced for $r=0,\dots,n$. In this way,  $(0,\dots,0, \theta_{n-r,r}, 0,\dots,0)$ are all normalized reduced cocycles.

Summarizing, under (FS) and (FE), the explicit Hodge to HKR transition is given as follows. Starting with any GS cohomology class $\CCC_{\mathrm{GS}}$, we choose a normalized reduced decomposable cocycle $(\theta_{0,n},\theta_{1,n-1},\dots, \theta_{n,0})$. Each $\theta_{n-r,r}=(\theta_{n-r,r}^\sigma)_{\sigma}$ lifts to a simplicial cocycle $\Theta_{n-r,r}=(\Theta_{n-r,r}^\sigma)_{\sigma}$ uniquely. Let $$\CCC_{\mathrm{simp}}\in\oplus_{p+q=n} H^p(\uuu,\wedge^q\ttt_\aaa)$$ be the class represented by $(\Theta_{0,n},\Theta_{1,n-1},\dots, \Theta_{n,0})$. The correspondence
\[H^n\CC_{\mathrm{GS}}(\aaa)\ni\CCC_{\mathrm{GS}}\longmapsto \CCC_{\mathrm{simp}}\in\bigoplus_{p+q=n} H^p(\uuu,\wedge^q\ttt_\aaa)\]
is bijective.

\subsection{Simplicial cohomology vs \v{C}ech cohomology}\label{parcechsimp}
In this section, we describe explicit quasi-isomorphisms between the natural complexes computing simplicial and \v{C}ech cohomology of presheaves respectively. This will be used in order to change from Toda's construction of first order deformations of schemes to a simplicial counterpart in \S\ref{partoda}.

Let $\uuu$ be a poset with binary meets. Meets in $\uuu$ are denoted by the symbol $\cap$. Let $\fff$ be a presheaf of $k$-modules on $\uuu$. The simplicial cohomology of $\fff$ is by Proposition \ref{propreduced} the cohomology of the reduced complex $\CC'^\bullet_{\mathrm{simp}}(\fff)$.

For any $p$-sequence $\tau=(U_0^\tau, U_1^\tau,\dots,U_p^\tau)\in \uuu^{p+1}$, denote by $\cap\tau$ the meet of all coordinates of $\tau$. Recall that the \v{C}ech complex $\check{\CC}^\bullet(\fff)$ is given by
\[\check{\CC}^\bullet(\fff)=\prod_{\tau\in \uuu^{\bullet+1}}\fff(\cap\tau).\]
A \v{C}ech $p$-cochain $\psi=(\psi^\tau)_\tau$ is said to be alternating if (1) $\psi^\tau=0$ whenever two coordinates of $\tau$ are equal, (2) $\psi^{\tau s}=(-1)^s\psi^\tau$ for any permutation $s$ of the set $\{0,\dots,p\}$. Here, we regard $\tau$ as a set-theoretic map $\{0,\dots,p\}\lra \uuu$, so $\tau s$ makes sense. Let $\check{\CC}'^\bullet(\fff)$ be the subcomplex of $\check{\CC}^\bullet(\fff)$ consisting of alternating cochains. It is well known that the \v{C}ech cohomology can be computed by both complexes.

As sets, $\nnn_p(\uuu)\subseteq \uuu^{p+1}$, so a simplex $\sigma$ can be regarded as a sequence $\tilde{\sigma}$ by forgetting the inclusions. Conversely, to a $p$-sequence $\tau$, we associate a $p$-simplex $\bar{\tau}$ by setting
\[U_i^{\bar{\tau}}=\cap_{j=i}^pU_j^\tau.\]
It is clear that $d\bar{\tau}=\cap\tau$. In \S \ref{parsimp}, the notation $\partial_i\sigma$ is given for any simplex $\sigma$. Similarly, $\partial_i\tau$ can be defined for any $\tau\in \uuu^{p+1}$. Here we define $\delta_i\tau$ by
\[
\delta_i\tau=(U_0^\tau, \ldots, U_{i-2}^\tau, U_{i-1}^\tau\cap U_i^\tau, U_{i+1}^\tau, \ldots, U_p^\tau)\in \uuu^p
\]
for $i=1,\dots,p$. Thus $\overline{\delta_i\tau}=\partial_i\bar{\tau}$.

Now for any reduced simplicial $p$-cochain $\phi$ and any alternating \v{C}ech $p$-cochain $\psi$, define
\begin{alignat*}{2}
\iota(\phi)^\tau&=\sum_{s\in S_{p+1}}(-1)^s\phi^{\overline{\tau s}},\quad& &\tau\in \uuu^{p+1},\\
\pi(\psi)^\sigma&=\psi^{\tilde{\sigma}},\quad& &\sigma\in \nnn_p(\uuu).
\end{alignat*}
It is easy to check that $\iota(\phi)\in \check{\CC}'^p(\fff)$ and $\pi(\psi)\in \CC'^p_{\mathrm{simp}}(\fff)$.

\begin{lemma}
The above $\iota$ and $\pi$ are morphisms of complexes.
\end{lemma}

\begin{proof}
The fact that $\pi$ commutes with differential is straightforward. Let us prove $\iota$ does also.

For any $\phi\in \CC'^p(\fff)$ and $(p+1)$-sequence $\tau$,
\begin{align*}
\iota d_{\mathrm{simp}} (\phi)^\tau&=\sum_{s\in S_{p+2}}(-1)^s d_{\mathrm{simp}}(\phi)^{\overline{\tau s}}\\
&=\sum_{s\in S_{p+2}}(-1)^s\biggl(\phi^{\partial_0\overline{\tau s}}|_{d \overline{\tau s}}+\sum_{i=1}^{p+1}(-1)^i \phi^{\partial_i\overline{\tau s}}\biggr)\\
&=\sum_{s\in S_{p+2}}(-1)^s\phi^{\partial_0\overline{\tau s}}|_{\cap\tau}+\sum_{i=1}^{p+1}(-1)^i\sum_{s\in S_{p+2}}(-1)^s\phi^{\overline{\delta_i(\tau s)}}.
\end{align*}
We first consider $\partial_0\overline{\tau s}$. If $s(0)=j$, then $\partial_0\overline{\tau s}=\overline{(\partial_j\tau) s'}$ with $s'\in S_{p+1}$ given by
\[
s'(i)=\begin{cases}
s(i+1), & \text{if $s(i+1)<j$},\\
s(i+1)-1, & \text{if $s(i+1)>j$}.
\end{cases}
\]
Furthermore, $s'$ ranges over $S_{p+1}$ if $s$ ranges over the set $\{s\in S_{p+2}\mid s(0)=j\}$, and $(-1)^s=(-1)^{s'}(-1)^j$. Thus
\[
\sum_{s\in S_{p+2}}(-1)^s\phi^{\partial_0\overline{\tau s}}|_{\cap\tau}=\sum_{j=0}^{p+1}(-1)^j\sum_{s'\in S_{p+1}}(-1)^{s'}\phi^{\overline{(\partial_j\tau) s'}}|_{\cap\tau}.
\]
Next we consider $\overline{\delta_i(\tau s)}$. It is easy to see that $\overline{\delta_i(\tau s)}=\overline{\delta_i(\tau r)}$ if $s$ equals $r$ composing the transposition $(i-1,i)$. It follows immediately that $\sum_{s\in S_{p+2}}(-1)^s\phi^{\overline{\delta_i(\tau s)}}=0$ for all $i\geq 1$, since $(-1)^s=-(-1)^r$.

On the other hand,
\[
d_{\mathrm{\check{C}ech}}\iota (\phi)^\tau=\sum_{i=0}^{p+1}(-1)^i\iota (\phi)^{\partial_i\tau}|_{\cap\tau}=\sum_{i=0}^{p+1}(-1)^i\sum_{s\in S_{p+1}}(-1)^s\phi^{\overline{(\partial_i\tau)s}}|_{\cap\tau}.
\]
So $\iota d_{\mathrm{simp}}=d_{\mathrm{\check{C}ech}}\iota$.
\end{proof}

Given any $\tau\in\uuu^{p+1}$ and $0\leq i\leq p$, define $\theta_i(\tau)$ by
\[
\theta_i(\tau)=\bigl(\cap_{j=0}^{p}U^\tau_j,\, \cap_{j=1}^{p}U^\tau_j,\, \ldots,\, \cap_{j=i}^{p}U^\tau_j,\, U^\tau_i,\,\ldots,\, U^\tau_p\bigr)\in \uuu^{p+2}.
\]

\begin{lemma}\label{lemma:simp-cech}
One has
\begin{enumerate}
\item $\partial_0\theta_i(\tau)=\theta_{i-1}\partial_0(\tau)$ for $1\leq i\leq p$,
\item $\partial_j\theta_i(\tau)=\theta_{i-1}\delta_j(\tau)$ for $1\leq j<i\leq p$,
\item $\partial_j\theta_i(\tau t)=\partial_{i+1}\theta_i(\tau)$ for some permutation $t$ (depending on $i$, $j$) with $(-1)^t=(-1)^{j-i-1}$, if $0\leq i<j\leq p+1$,
\item $\partial_{i+1}\theta_i(\tau)=\partial_{i+1}\theta_{i+1}(\tau)$ for $0\leq i\leq p-1$,
\item $\partial_{p+1}\theta_p(\tau)=\tilde{\bar{\tau}}$.
\end{enumerate}
\end{lemma}

\begin{proof}
(3) can be proved by induction on $j-i$. The others are straightforward.
\end{proof}

By the definitions of $\iota$, $\pi$, we have $\pi\iota=1$. In order to show they are quasi-isomorphisms, we construct a family of maps $h=\{h^p\colon \check{\CC}'^p(\fff)\to \check{\CC}'^{p-1}(\fff)\}_p$ by for all $\psi\in \check{\CC}'^p(\fff)$ and $\tau\in \uuu^{p}$,
\[
h^p(\psi)^\tau=\sum_{i=0}^{p-1}\frac{(-1)^i}{(p-i)!}\sum_{s\in S_{p}}(-1)^s\psi^{\theta_i(\tau s)}.
\]

\begin{lemma}
The map $h$ is a homotopy from $1$ to $\iota\pi$.
\end{lemma}

\begin{proof}
By computation, we have
\begin{align*}
\iota\pi (\psi)^\tau&=\sum_{s\in S_{p+1}}(-1)^s\pi (\psi)^{\overline{\tau s}}=\sum_{s\in S_{p+1}}(-1)^s\psi^{\widetilde{\overline{\tau s}}},\\
d_{\mathrm{\check{C}ech}}h^{p}(\psi)^\tau&=\sum_{j=0}^{p}(-1)^jh^{p}(\psi)^{\partial_j\tau}|_{\cap\tau}\\
&=\sum_{i=0}^{p-1}\sum_{j=0}^{p}\frac{(-1)^{i+j}}{(p-i)!}\sum_{s'\in S_{p}}(-1)^{s'}\psi^{\theta_i((\partial_j\tau)s')}|_{\cap\tau},\\
h^{p+1} d_{\mathrm{\check{C}ech}} (\psi)^\tau&=\sum_{i=0}^{p}\frac{(-1)^i}{(p+1-i)!}\sum_{s\in S_{p+1}}(-1)^s d_{\mathrm{\check{C}ech}}(\psi)^{\theta_i(\tau s)}\\
&=\sum_{i=0}^{p}\sum_{j=0}^{p+1}\frac{(-1)^{i+j}}{(p+1-i)!}\sum_{s\in S_{p+1}}(-1)^s\psi^{\partial_j\theta_i(\tau s)}|_{\cap\theta_i(\tau s)}.
\end{align*}

Let us consider the third equation. We separate the sum as $\mathrm{I}+\mathrm{II}$ where $\mathrm{I}$ corresponds to $j=0$ and $\mathrm{II}$ to $j\geq 1$. Then
\begin{align*}
\mathrm{I}&=\sum_{i=0}^{p}\frac{(-1)^{i}}{(p+1-i)!}\sum_{s\in S_{p+1}}(-1)^s\psi^{\partial_0\theta_i(\tau s)}|_{\cap\tau}\\
&=\sum_{s\in S_{p+1}}\frac{(-1)^s}{(p+1)!}\psi^{\partial_0\theta_0(\tau s)}|_{\cap\tau}+\sum_{i=0}^{p-1}\frac{(-1)^{i+1}}{(p-i)!}\sum_{s\in S_{p+1}}(-1)^s\psi^{\partial_0\theta_{i+1}(\tau s)}|_{\cap\tau}\\
&\overset{\scriptscriptstyle(1)}{=}\sum_{s\in S_{p+1}}\frac{(-1)^s}{(p+1)!}\psi^{\tau s}-\sum_{i=0}^{p-1}\frac{(-1)^{i}}{(p-i)!}\sum_{s\in S_{p+1}}(-1)^s\psi^{\theta_{i}\partial_0(\tau s)}|_{\cap\tau}\\
&=\psi^{\tau}-\sum_{i=0}^{p-1}\frac{(-1)^{i}}{(p-i)!}\sum_{j=0}^{p}\sum_{\substack{s\in S_{p+1}\\s(0)=j}}(-1)^s\psi^{\theta_{i}\partial_0(\tau s)}|_{\cap\tau}\\
&=\psi^{\tau}-\sum_{i=0}^{p-1}\frac{(-1)^{i}}{(p-i)!}\sum_{j=0}^{p}\sum_{s'\in S_{p}}(-1)^j(-1)^{s'}\psi^{\theta_{i}(\partial_j\tau) s')}|_{\cap\tau}\\
&=\psi^{\tau}-d_{\mathrm{\check{C}ech}}h^{p}(\psi)^\tau,
\end{align*}
and
\begin{align*}
\mathrm{II}&=\sum_{i=0}^{p}\sum_{j=1}^{p+1}\frac{(-1)^{i+j}}{(p+1-i)!}\sum_{s\in S_{p+1}}(-1)^s\psi^{\partial_j\theta_i(\tau s)}|_{\cap\tau}\\
&=\sum_{i=2}^{p}\sum_{j=1}^{i-1}\sum_{s\in S_{p+1}}\frac{(-1)^{i+j}(-1)^s}{(p+1-i)!}\psi^{\partial_j\theta_i(\tau s)}+\sum_{i=1}^{p}\sum_{s\in S_{p+1}}\frac{(-1)^{2i}(-1)^s}{(p+1-i)!}\psi^{\partial_i\theta_i(\tau s)}\\
&\phantom{\:=\:}{}+\sum_{i=0}^{p}\sum_{j=i+1}^{p+1}\sum_{s\in S_{p+1}}\frac{(-1)^{i+j}(-1)^s}{(p+1-i)!}\psi^{\partial_j\theta_i(\tau s)}\\
&\overset{\scriptscriptstyle(2)}{=}\sum_{i=2}^{p}\sum_{j=1}^{i-1}\sum_{s\in S_{p+1}}\frac{(-1)^{i+j}(-1)^s}{(p+1-i)!}\psi^{\theta_{i-1}\delta_j(\tau s)}+\sum_{i=1}^{p}\sum_{s\in S_{p+1}}\frac{(-1)^s}{(p+1-i)!}\psi^{\partial_i\theta_i(\tau s)}\\
&\phantom{\:=\:}{}+\sum_{i=0}^{p}\sum_{j=i+1}^{p+1}\sum_{s\in S_{p+1}}\frac{(-1)^{i+j}(-1)^{st}}{(p+1-i)!}\psi^{\partial_j\theta_i(\tau st)}\\
&\overset{\scriptscriptstyle(3)}{=}\sum_{i=1}^{p}\sum_{s\in S_{p+1}}\frac{(-1)^s}{(p+1-i)!}\psi^{\partial_i\theta_i(\tau s)}+\sum_{i=0}^{p}\sum_{j=i+1}^{p+1}\sum_{s\in S_{p+1}}\frac{(-1)(-1)^{s}}{(p+1-i)!}\psi^{\partial_{i+1}\theta_{i}(\tau s)}\\
&\overset{\scriptscriptstyle(4)}{=}\sum_{i=1}^{p}\sum_{s\in S_{p+1}}\frac{(-1)^s}{(p+1-i)!}\psi^{\partial_i\theta_i(\tau s)}-\sum_{i=0}^{p-1}\sum_{s\in S_{p+1}}\frac{(-1)^{s}}{(p-i)!}\psi^{\partial_{i+1}\theta_{i+1}(\tau s)}\\
&\phantom{\:=\:}{}-\sum_{s\in S_{p+1}}(-1)^{s}\psi^{\partial_{p+1}\theta_{p}(\tau s)}\\
&\overset{\scriptscriptstyle(5)}{=}-\iota\pi (\psi)^\tau
\end{align*}
where the numbers (1)--(5) over the signs of equality mean that the equations hold by the corresponding items of Lemma~\ref{lemma:simp-cech}.

It follows that
\[
h^{p+1} d_{\mathrm{\check{C}ech}} (\psi)^\tau=\mathrm{I}+\mathrm{II}=\psi^{\tau}-d_{\mathrm{\check{C}ech}}h^{p}(\psi)^\tau-\iota\pi (\psi)^\tau,
\]
and hence $1-\iota\pi=h d_{\mathrm{\check{C}ech}}+d_{\mathrm{\check{C}ech}}h$.
\end{proof}

Therefore, $\iota$, $\pi$ induce mutually inverse isomorphisms between $H^\bullet(\uuu, \fff)$ and $\check{H}^\bullet(\uuu, \fff)$.

When $\uuu$ is a semi-separating cover of a scheme $X$, and $\fff$ is a sheaf on $X$, we regard $\fff|_{\uuu}$ as a presheaf on $\uuu$. Leray's Theorem tells us $\check{H}^\bullet(\uuu, \fff)\cong H^\bullet(X, \fff)$. We furthermore have isomorphisms
\[
H^\bullet(\uuu, \fff|_\uuu)\cong \check{H}^\bullet(\uuu, \fff|_\uuu)=\check{H}^\bullet(\uuu, \fff)\cong H^\bullet(X, \fff).
\]

\section{Quasi-coherent modules}\label{secabeltoda}\label{parparqcoh}

For a scheme $X$, the category $\Qch(X)$ of quasi-coherent sheaves on $X$ is of fundamental importance in algebraic geometry. In non-commutative contexts, an important task is to find suitable replacements for this category. For a possibly non-commutative $k$-algebra $A$, the category $\Mod(A) = \Mod^r(A)$ of right $A$-modules is a natural replacement (the choice between left and right is a matter of convention, the category $\Mod^l(A)$ of left modules is just as valid as a choice, and will now be associated to $A^{\op}$ as $\Mod^r(A^{\op}) = \Mod^l(A)$).
For an arbitrary small $k$-linear category $\AAA$, this definition is generalized by putting $\Mod(\AAA) = \Mod^r(\AAA)$ equal to the category of $k$-linear functors $\AAA^{\op} \lra \Mod(k)$ (and $\Mod^l(\AAA) = \Mod^r(\AAA^{\op})$).

In non-commutative projective geometry, to a sufficiently nice $\Z$-graded algebra $A$, one associates a category $\mathsf{QGr}(A)$ of ``quasi-coherent graded modules'', obtained as the quotient of the graded modules by the torsion modules (see \cite{staffordvandenbergh}, \cite{vandenbergh}, \cite{vandenbergh2}, \cite{dedekenlowen} for more details and generalizations).

In this paper, we take the ``local approach'' to non-commutative schemes and consider a presheaf $\aaa\colon \uuu^{\op} \lra \Alg(k)$ on a small category $\uuu$ as a kind of ``non-commutative structure sheaf on affine opens''. Repeating the process by which a quasi-coherent sheaf is obtained by glueing modules on affine opens, we define the category of quasi-coherent modules over $\aaa$ to be
$$\Qch({\aaa}) = \Des(\Mod_\aaa),$$
the descent category of the prestack $\Mod_{\aaa}$ of module categories on $\aaa$. 

After recalling the definitions of prestacks and their morphisms, and the construction of the descent category $\Des(\ccc)$ of an arbitrary prestack $\ccc$ on $\uuu$ in \S \ref{pardescent}, in \S \ref{parqcoh} we introduce the prestacks of the form $\Mod_{\aaa}$ over a prestack $\aaa\colon \uuu^{\op} \lra \Cat(k)$ of small $k$-linear categories, as well as their descent categories $\Qch(\aaa)$. Here $\Mod_{\aaa}$ has the tensor functors $\Mod(\aaa(U)) \lra \Mod(\aaa(V))$ for $\aaa(U) \lra \aaa(V)$ as restriction functors (for $V \lra U$ in $\uuu$). Twisted presheaves and prestacks naturally occur as deformations (see \S \ref{partwistedpresheaf} and \S \ref{partwistab}). They play an important role in the context of deformation quantization, see  \cite{kashiwara}, \cite{kontsevich1}, \cite{vandenbergh4}, \cite{calaquevandenbergh}, \cite{yekutieli2}, \cite{yekutieli3}.

In \S \ref{parqchmodule} we turn our attention to twisted presheaves $\aaa$ with central twists. In this case, an alternative replacement of the category of quasi-coherent sheaves is described. Since $\aaa$ has an underlying presheaf $\underline{\aaa}$ (forgetting the twists), we can construct a prestack $\Pre_{\aaa}$ whose values are ordinary presheaf categories on the restrictions of $\underline{\aaa}$, but where the twists of $\aaa$ are naturally built into the resulting prestack. As descent category, we obtain the category $\Pre(\aaa)$ of twisted presheaves, in the spirit of the categories of twisted sheaves considered for instance in \cite{caldararu}. A natural sub-prestack $\QPr_{\aaa}$ gives rise to the descent category
$$\QPr(\aaa) = \Des(\QPr_{\aaa})$$ of quasi-coherent presheaves. This category is in fact equivalent to $\Qch(\aaa)$ (Theorem \ref{thmqmodqpr}).

In \S \ref{parpargroth}, generalizing the scheme case based upon \cite{enochsestrada}, we prove that if a prestack $\aaa$ gives rise to exact restriction functors
$$- \otimes_u \aaa(V): \Mod(\aaa(U)) \lra \Mod(\aaa(V)),$$
the category $\Qch(\aaa)$ is a Grothendieck abelian category.

 In \S \ref{parflat}, we further restrict out attention to what we call a \emph{quasi-compact semi-separated prestack} $\aaa$. Roughly speaking, this means that the restriction functors of $\Mod_{\aaa}$ are the left adjoints to compatible localization functors, which are themselves exact. In this case, $\Qch(\aaa)$ inherits the homological property of flatness \cite{lowenvandenberghab}, which is crucial for deformation theory (Proposition \ref{propdesflat1}).

\subsection{Descent categories}\label{pardescent}\label{parstar}

Let $k$ be a fixed commutative ground ring. Let $\uuu$ be a small category. A prestack on $\uuu$ is the categorical version of a twisted presheaf of $k$-algebras. 

\begin{definition}\label{defpseudo}
A \emph{prestack} $\aaa$ on $\uuu$ is a pseudofunctor on $\uuu$ taking values in $k$-linear categories. Precisely, $\aaa$ consists of the following data:
\begin{itemize}
\item for every $U \in \uuu$ a $k$-linear category $\aaa(U)$;
\item for every $u\colon V \lra U$ in $\uuu$ a $k$-linear functor $f^u = u^{\ast}\colon \aaa(U) \lra \aaa(V)$;
\item for every pair $v\colon W \lra V$, $u\colon V  \lra U$ a natural isomorphism
$$c^{u,v}\colon v^{\ast}u^{\ast} \lra (uv)^{\ast};$$
\item for every $U \in \uuu$ a natural isomorphism
$$z^U\colon 1_{\aaa(U)} \lra f^{1_U}.$$
\end{itemize}
These data further satisfy the following compatibility conditions for every triple $w\colon T\lra W$, $v\colon W\lra V$, $u\colon V\lra U$:
\begin{gather*}
c^{u,vw}(c^{v,w} \circ u^{\ast})=c^{uv,w}(w^{\ast} \circ c^{u,v}),\\
c^{u,1_V}(z^V \circ u^{\ast})=1,\quad c^{1_U,u}(u^{\ast} \circ z^U)=1.
\end{gather*}
\end{definition}

\begin{definition}
For prestacks $(\aaa, m, f, c, z)$, $(\aaa', m', f', c', z')$ on $\uuu$, a \emph{morphism of prestacks} is a pseudonatural transformation $\varphi\colon \aaa \lra \aaa'$. Precisely, $\varphi$ consists of the following data:
\begin{itemize}
\item for every $U \in \uuu$ a $k$-linear functor $g^U\colon \aaa(U) \lra \aaa'(U)$;
\item for every $u\colon V \lra U$ in $\uuu$ a natural isomorphism
$$\tau^u\colon {f'}^u g^U \lra g^V f^u.$$
\end{itemize}
These data further satisfy the following compatibility conditions for every couple $v\colon W\lra V$, $u\colon V\lra U$:
\begin{gather}
\tau^{uv}({c'}^{u,v} \circ g^U) = (g^W \circ c^{u,v})(\tau^v \circ f^u)({f'}^v \circ \tau^u),\label{eqpseudonattran1}\\
\tau^{1_U}(z'^U \circ g^U) = g^U \circ z^U.\label{eqpseudonattran2}
\end{gather}

\end{definition}

\begin{example}
A twisted presheaf $\aaa$ as in Definition \ref{deftwisted} corresponds to a prestack by viewing the algebras $\aaa(U)$ as one-object $k$-linear categories. A morphisms of twisted presheaves corresponds to a morphism of prestacks.
\end{example}

\begin{example}\label{exopps}
Consider a prestack $\aaa$. We obtain the \emph{opposite prestack} $\aaa^{\op}$ with $\aaa^{\op}(U) = (\aaa(U))^{\op}$, with restriction maps $(u^{\ast})^{\op}\colon \aaa(U)^{\op} \lra \aaa(V)^{\op}$ for $u\colon V \lra U$ in $\uuu$, and with $(c^{-1})^{u,v} = (c^{u,v})^{-1}$ and $(z^{-1})^U = (z^U)^{-1}$. One checks that this indeed defines a prestack with $(\aaa^{\op})^{\op} = \aaa$. This extends the definition of opposite twisted presheaf from Example \ref{exop1}.
\end{example}

One goes on to define, for morphisms of prestacks $\varphi$, $\varphi'\colon \aaa \lra \aaa'$, the notion of a \emph{modification} $\mu\colon \varphi \lra \varphi'$. As such, one obtains a 2-category $\mathsf{Prestack}(\uuu, k)$ of prestacks, morphisms of prestacks and modifications. See for instance \cite{leinster} for further details.

Let $\bbb$ be a prestack on $\uuu$. A \emph{pre-descent datum} in $\bbb$ consists of a collection of objects $B = (B_U)_U$ with $B_U \in \bbb(U)$ together with, for every $u\colon V \lra U$ in $\uuu$, a morphism
$$\varphi_u\colon u^{\ast} B_U \lra B_V$$
in $\bbb(V)$. These morphisms have to satisfy the obvious compatibility with the twist isomorphisms of $\bbb$ when considering an additional $v\colon W \lra V$, namely,
\[
 \varphi_vv^*(\varphi_u)=\varphi_{uv}c^{u,v,B_U}.
\]
The pre-descent datum $B$ is a \emph{descent datum} if $\varphi_u$ is an isomorphism for all $u \in \uuu$.

A \emph{morphism of pre-descent data} $g\colon (B_U)_U \lra (B'_U)_U$ consists of compatible morphisms $g_U\colon B_U \lra B'_U$ in $\bbb(U)$ for $U \in \uuu$.

Pre-descent data in $\bbb$ and morphisms of pre-descent data constitute a category $\PDes(\bbb)$, with a full subcategory $\Des(\bbb)$ of descent data.

The category $(\mathrm{P})\Des(\bbb)$ comes equipped with canonical functors $\pi_V\colon (\mathrm{P})\Des(\bbb) \lra \bbb(V)$, $(B_U)_U \longmapsto B_V$.

The (pre-)descent category is most useful for prestacks of sufficiently large $k$-linear categories.

\begin{proposition}\label{lemqchabelian}
Let $\bbb$ be a prestack on $\uuu$.
\begin{enumerate}
\item All limits which exist in $\bbb(U)$ for every $U$, exist in $\PDes(\bbb)$ and are computed pointwise (i.e., $\pi_V\colon \PDes(\bbb) \lra \bbb(V)$ preserves them for every $V$).
\item All colimits which exist in $\bbb(U)$ for every $U$, and are preserved by $u^{\ast}$ for every $u$, exist in $\PDes(\bbb)$ and are computed pointwise.
\item All limits and colimits which exist in $\bbb(U)$ for every $U$, and are preserved by $u^{\ast}$ for every $u$, exist in $\Des(\bbb)$ and are computed pointwise (i.e., $\pi_V\colon \Des(\bbb) \lra \bbb(V)$ preserves them for every $V$).
\end{enumerate}
\end{proposition}

\begin{proof}
We first look at the statements concerning limits. Consider a diagram $(B_{i, U})_{U, i}$ of pre-descent data. For every $U$, we obtain the object $\lim_i B_{i, U} \in \bbb(U)$. The morphisms $\varphi_{i, u}\colon u^{\ast} B_{i, U} \lra B_{i, V}$ for $u\colon V \lra U$ give rise to a limit morphism $\lim_i \varphi_{i, u}\colon \lim_i u^{\ast} B_{i, U} \lra \lim_i B_{i, V}$. Composed with the canonical map $\psi_u\colon u^{\ast} \lim_i B_{i, U} \lra \lim_i u^{\ast} B_{i,U}$, we obtain $\varphi_u = \lim_i \varphi_{i, u}\psi_u \colon u^{\ast} \lim_i B_{i, U} \lra \lim_i B_{i, V}$. The resulting $(\lim_i B_{i, U})_U$ endowed with the maps $\varphi_u$ is seen to define a pre-descent datum. If $(B_{i, U})_{U, i}$ is a diagram of descent data, and $u^{\ast}$ preserves the limit $\lim_i$, then $\psi_u$ is an isomorphism and so is $\varphi_u$.

Next, we look at the statements concerning colimits. Consider a diagram $(B_{i, U})_{U, i}$ of pre-descent data. For every $U$, we obtain the object $\colim_i B_{i, U} \in \bbb(U)$. The morphisms $\varphi_{i, u}\colon u^{\ast} B_{i, U} \lra B_{i, V}$ for $u\colon V \lra U$ give rise to a colimit morphism $\colim_i \varphi_{i, u}\colon \colim_i u^{\ast} B_{i, U} \lra \colim_i B_{i, V}$. This time, we have a canonical morphism $\psi_u\colon \colim_i u^{\ast} B_{i, U} \lra u^{\ast}\colim_i B_{i, U}$. If $u^{\ast}$ preserves the colimit $\colim_i$, then $\psi_u$ is an isomorphism, and we put $\varphi_u = \colim_i \varphi_{i, u}\psi^{-1}_u \colon u^{\ast} \colim_i B_{i, U} \lra \colim_i B_{i, V}$. The resulting $(\colim_i B_{i, U})_U$ endowed with the maps $\varphi_u$ is seen to define a pre-descent datum. If $(B_{i, U})_{U, i}$ is a diagram of descent data, then $\varphi_u$ is an isomorphism.
\end{proof}

\begin{corollary}
Let $\bbb$ be a prestack on $\uuu$. If all the categories $\bbb(U)$ are abelian and all the functors $u^*\colon \bbb(U)\lra\bbb(V)$ are exact, then $\Des(\bbb)$ is abelian and the canonical functors $\pi_V\colon \Des(\bbb) \lra \bbb(V)$ are exact.
\end{corollary}

A morphism of prestacks $\varphi\colon \aaa \lra \bbb$ induces a $k$-linear functor
$$\Des(\varphi)\colon \Des(\aaa) \lra \Des(\bbb).$$
We will make use of the following:

\begin{proposition}\label{propdeseq}
Let $\varphi\colon \aaa \lra \bbb$ be a morphism of prestacks on $\uuu$. The following are equivalent:
\begin{enumerate}
\item $\varphi$ is an equivalence in the 2-category $\mathsf{Prestack}(\uuu, k)$;
\item $\varphi^U\colon \aaa(U) \lra \bbb(U)$ is an equivalence of categories for every $U \in \uuu$.
\end{enumerate}
In this case, the induced functor $\Des(\varphi)\colon \Des(\aaa) \lra \Des(\bbb)$ is an equivalence of categories.
\end{proposition}

\subsection{Quasi-coherent modules}\label{parqcoh}
Next we describe our main source of prestacks of affine localizations.

Let $\Cat(k)$ be the category of small $k$-linear categories and $k$-linear functors.
Let $\uuu$ be a small category and let $\aaa\colon \uuu^{\op} \lra \Cat(k)$ be a prestack of small $k$-linear categories. For instance, $\aaa$ could be a twisted presheaf of $k$-algebras (using the natural inclusion $\Alg(k) \subseteq \Cat(k)$).

We construct an associated prestack $$\Mod_\aaa = \Mod^r_\aaa$$ as follows.
\begin{itemize}
\item For $U \in \uuu$, $\Mod_\aaa(U) = \Mod(\aaa(U))$, the category of $k$-linear functors $\aaa(U)^{\op} \lra \Mod(k)$;
\item For $u\colon V \lra U$, the restriction functor is $- \otimes_{u}\aaa(V) \colon \Mod(\aaa(U)) \lra \Mod(\aaa(V))$, the unique colimit preserving functor extending $u^{\ast}\colon \aaa(U) \lra \aaa(V)$;
\item The twist isomorphisms ($c$ and $z$) combine the twists in $\aaa$ with the natural isomorphisms between tensor functors.
\end{itemize}

Let us make the construction explicit. Let $F\in\Mod(\aaa(U))$. Define $F\otimes_u\aaa(V)$ by for all $B\in\aaa(V)$,
\[
F\otimes_u\aaa(V)(B)=\bigoplus_{A\in\aaa(U)}F(A)\otimes_k\aaa(V)(B,u^*A)/\sim
\]
where $\sim$ is defined as the equivalence relation generated as follows. For $a\colon A\lra A'$ in $\aaa(U)$, $x\in F(A')$, $y\in \aaa(V)(B,u^*A)$, we put
\[
F(a)(x)\otimes y\sim x\otimes u^*(a)y.
\]
This expresses that the tensor product is taken over $\aaa(U)$, so the action by morphisms $a$ in $\aaa(U)$ can be moved through the tensor symbol. In particular, if $F=\aaa(U)(-,A')$ for some $A'\in\aaa(U)$, $u^*$ induces an isomorphism
\[
\theta^{u}_{A'}\colon\aaa(U)(-,A')\otimes_u\aaa(V)\cong\aaa(V)(-, u^*A'),
\]
and if $u = 1_U$, $(z^U)^{-1}$ induces an isomorphism $F \otimes_{1_U} \aaa(U) \cong F.$ This is equivalent to say that the isomorphism $\Mod(z)^U\colon 1_{\Mod_{\aaa}(U)} \lra - \otimes_{1_U} \aaa(U)$ is induced by $z^U$.

Suppose that the prestack $\aaa$ gives rise to the isomorphism $c^{u,v}\colon v^*u^*\lra (uv)^*$ for $v\colon W\lra V$, $u\colon V\lra U$. We will give the corresponding isomorphism $\Mod(c)^{u,v}\colon -\otimes_u\aaa(V)\otimes_v\aaa(W)\lra -\otimes_{uv}\aaa(W)$. In fact, according to the above construction, we have
\begin{align*}
F\otimes_u\aaa(V)\otimes_v\aaa(W)(C)&=\bigoplus_{B\in\aaa(V)}F\otimes_u\aaa(V)(B)\otimes_k\aaa(W)(C,v^*B)/\sim\\
&=\bigoplus_{\substack{A\in\aaa(U)\\B\in\aaa(V)}}F(A)\otimes_k\aaa(V)(B,u^*A)\otimes_k\aaa(W)(C,v^*B)/\sim\\
&\cong\bigoplus_{A\in\aaa(U)}F(A)\otimes_k\aaa(W)(C,v^*u^*A)/\sim\\
&\cong\bigoplus_{A\in\aaa(U)}F(A)\otimes_k\aaa(W)(C,(uv)^*A)/\sim
\end{align*}
where the first isomorphism is induced by $\theta^{v}_{u^*A}$ and the second by $c^{u,v}$.

One can also define another prestack $\Mod^l_\aaa$, the ``left version'' of $\Mod_\aaa$. Namely, $\Mod^l_\aaa(U)$ is defined to be the category of $k$-linear functors $\aaa(U) \lra \Mod(k)$; the restriction maps $\aaa(V)\otimes_u-$ can be defined analogously. In this case, the twists of $\Mod^l_\aaa$ are induced by $c^{-1}$ and $z^{-1}$.

\begin{remark}\label{remark:mod}
When $\aaa$ is a twisted presheaf of $k$-algebras, the category $\Mod(\aaa(U))$ (resp.\ $\Mod^l(\aaa(U))$) is the same as the category of right (resp.\ left) modules over $\aaa(U)$ in the usual sense, and the restriction map $-\otimes_u\aaa(V)$ (resp.\ $\aaa(V)\otimes_u-$) also coincides with the usual tensor product. So our notations will not lead to misunderstanding. Furthermore, the twists $\Mod(c)^{u,v}$, $\Mod(z)^U$ and $\Mod^l(c)^{u,v}$, $\Mod^l(z)^U$ are given by
\begin{align*}
&\begin{aligned}
\Mod(c)^{u,v}_M\colon M\otimes_u\aaa(V)\otimes_v\aaa(W)&\lra M\otimes_{uv}\aaa(W),\\
m\otimes a\otimes b&\longmapsto m\otimes c^{u,v}v^*(a)b,
\end{aligned}\\
&\begin{aligned}
\Mod(z)^{U}_M\colon M&\lra M\otimes_{1_U}\aaa(U),\\
m&\longmapsto m\otimes z^{U},
\end{aligned}\\
&\begin{aligned}
\Mod^l(c)^{u,v}_M\colon \aaa(W)\otimes_v\aaa(V)\otimes_uM&\lra \aaa(W)\otimes_{uv}M,\\
b\otimes a\otimes m&\longmapsto bv^*(a)(c^{u,v})^{-1}\otimes m,
\end{aligned}\\
&\begin{aligned}
\Mod^l(z)^{U}_M\colon M&\lra \aaa(U)\otimes_{1_U}M,\\
m&\longmapsto (z^{U})^{-1} \otimes m
\end{aligned}
\end{align*}
for any $\aaa(U)$-module $M$.
\end{remark}

\begin{definition}\label{defqch}
Let $\aaa\colon \uuu^{\op} \lra \Cat(k)$ be a prestack.
We define the category of \emph{(right) modules} over $\aaa$ to be
$$\Mod(\aaa) = \Mod^r(\aaa) = \PDes(\Mod^r_\aaa),$$
and the category of
\emph{(right) quasi-coherent modules} over $\aaa$ to be
$$\Qch(\aaa) = \Qch^r(\aaa) = \Des(\Mod^r_\aaa).$$
We define the category of \emph{left modules} over $\aaa$ to be
$$\Mod^l(\aaa) = \PDes(\Mod^l_\aaa),$$
and the category of
\emph{left quasi-coherent modules} over $\aaa$ to be
$$\Qch^l(\aaa) = \Des(\Mod^l_\aaa).$$
\end{definition}

Since $\Mod^l(\aaa) \cong \Mod^r(\aaa^{\op})$ and $\Qch^l(\aaa) \cong \Qch^r(\aaa^{\op})$, it suffices to study right (quasi-coherent) modules.

\subsection{Quasi-coherent presheaves}\label{parqchmodule}
In \S \ref{parqcoh},  we define a category of quasi-coherent modules over an arbitrary prestack of small $k$-linear categories. In particular, the definition applies to a twisted presheaf of algebras with central twists. In this section, we define a category of ``quasi-coherent presheaves'' in this special case and we prove that both categories are equivalent.

Let $\aaa$ be a presheaf of $k$-algebras on $\uuu$. Denote by $\Pre(\aaa|_U)$ the category of presheaves of right modules over $\aaa|_U$. Every morphism $u\colon V\lra U$ in $\uuu$ induces a functor $u^*_\Pre\colon \Pre(\aaa|_U)\lra\Pre(\aaa|_V)$. It is obvious that $v^*_\Pre u^*_\Pre=(uv)^*_\Pre$ for an additional $v\colon W\lra V$ in $\uuu$ and $(1_U)^*_\Pre=1_{\Pre(\aaa|_U)}$. Thus we obtain a functor
\begin{align*}
\Pre(\aaa)\colon \uuu^{\op}&\lra\Cat(k),\\
U&\longmapsto \Pre(\aaa|_U)\\
u&\longmapsto u^*_\Pre.
\end{align*}

Let $M$ be a right $\aaa(U)$-module. Then for any $u\colon V\lra U$, $\widetilde{M}(u):=M\otimes_u\aaa(V)$ is a right $\aaa(V)$-module where $\aaa(V)$ is regarded as a left $\aaa(U)$-module via $u^*$ and $M\otimes_u\aaa(V):=M\otimes_{\aaa(U)}\aaa(V)$. Let $u'\colon V'\lra U$ and $v\colon V'\lra V$ be such that $uv=v'$. Then $v^*\colon \aaa(V)\lra\aaa(V')$ is an $\aaa(U)$-bimodule homomorphism, and thus $\widetilde{M}(v):=1_M\otimes v^*\colon \widetilde{M}(u)\lra \widetilde{M}(u')$ is a right $\aaa(U)$-module homomorphism. We can check the functorial property of $\widetilde{M}$ and hence we obtain a presheaf $\widetilde{M}$ of right modules over $\aaa|_U$ on $\uuu/U$. Consider the following assignment
\begin{align*}
Q^U\colon \Mod(\aaa(U))&\lra \Pre(\aaa|_U),\\
M&\longmapsto \widetilde{M}.
\end{align*}
In order to make $Q^U$ into a functor, associate to an $\aaa(U)$-module homomorphism $g\colon M\lra N$ the natural transformation $\tilde{g}=\{\tilde{g}^u \}_{u\colon V\lra U}$ defined by
\[
\tilde{g}^u:=g\otimes 1_{\aaa(V)}\colon M\otimes_u\aaa(V)\lra N\otimes_u\aaa(V).
\]
It is easy to verify that $Q^U$ is indeed a functor.

Observe that we have the following canonical isomorphism
\[
\can^{u,v}_{M}\colon M\otimes_u\aaa(V)\otimes_v\aaa(W)\lra M\otimes_{uv}\aaa(W), \quad m\otimes a\otimes b\longmapsto m\otimes v^*(a)b.
\]

\begin{lemma}\label{lemqch1}
\begin{enumerate}
\item The functor $Q^U$ is fully faithful for each $U$.
\item The square
\[
\xymatrix{
\Mod(\aaa(U)) \ar[r]^-{Q^U}\ar[d]_-{-\otimes_u\aaa(V)} & \Pre(\aaa|_U) \ar[d]^-{u_\Pre^*} \\
\Mod(\aaa(V)) \ar[r]^-{Q^V} & \Pre(\aaa|_V)
}
\]
is 2-commutative. More precisely, there is a natural isomorphism
\[
\tau^u\colon u_{\Pre}^* Q^U\lra Q^V(-\otimes_u\aaa(V))
\]
induced by $(\can^{u,v}_M)^{-1}$.
\end{enumerate}
\end{lemma}

\begin{proof}
(1) Note that $F\in\Pre(\aaa|_U)(\widetilde{M},\widetilde{N})$ is completely determined by its component $F^{1_U}\in\Mod(\aaa(U))(M\otimes_{1_U}\aaa(U), N\otimes_{1_U}\aaa(U))$ which can be identified with an $\aaa(U)$-module homomorphism $M\lra N$.

(2) Straightforward.
\end{proof}

\begin{definition}\label{defqchmodule1}
A presheaf $\fff\in\Pre(\aaa|_U)$ is called a \textit{quasi-coherent presheaf} over $\aaa|_U$ if $\fff\cong \widetilde{M}$ for some $\aaa(U)$-module $M$.
\end{definition}

Denote by $\mathsf{QPr}(\aaa|_U)$ the category of quasi-coherent presheaves over $\aaa|_U$, i.e. the essential image of $Q_U$, which is a full subcategory of $\Pre(\aaa|_U)$. By Lemma \ref{lemqch1} (2), $u^*_\Pre$ preserves quasi-coherent modules. Let $u^*_\QPr\colon \QPr(\aaa|_U)\lra \QPr(\aaa|_V)$ be the restriction of $u^*_\Pre$. By abuse of the notation, the restricted isomorphism
\begin{equation}\label{eqtau}
u_{\QPr}^* Q^U\lra Q^V(-\otimes_u\aaa(V))
\end{equation}
is still denoted by $\tau^u$.

Our next task is to define quasi-coherent presheaves over a twisted presheaf $\aaa$ with central twists $c$. In order to adapt to the twist $\Mod(c)^{u,v}$, we have to define an appropriate twist $\Pre(c)^{u,v}\colon v^*_\Pre u^*_\Pre\lra(uv)^*_\Pre$ artificially. Let $\fff\in\Pre(\underline{\aaa}|_U)$ and $w\colon T\lra W\in\uuu/W$. We have $v_\Pre^*u_\Pre^*(\fff)(w)=(uv)_\Pre^*(\fff)(w)=\fff(uvw)\in\Mod(\aaa(T))$. By \eqref{eq:twistcocycle},  $w^*(c^{u,v})=(c^{uv,w})^{-1}c^{u,vw}c^{v,w}$ is a central invertible element in $\aaa(T)$, and hence determines an automorphism $w^*(c^{u,v})_r\colon \fff(uvw)\lra \fff(uvw)$, $m\longmapsto mw^*(c^{u,v})$. Hence we obtain an isomorphism
\[
\Pre(c)^{u,v}_{\fff}\colon v_\Pre^*u_\Pre^*(\fff)\lra (uv)_\Pre^*(\fff)
\]
in $\Pre(\underline{\aaa}|_W)$ as well as a natural transformation $\Pre(c)^{u,v}=\bigl(\Pre(c)^{u,v}_{\fff}\bigr)_{\fff}$. A direct computation shows that the condition
\[
\Pre(c)^{u,vw}\Pre(c)^{v,w}=\Pre(c)^{uv,w}w^*_\Pre(\Pre(c)^{u,v})
\]
is satisfied. Therefore we obtain two prestacks
\begin{align*}
\Pre_{\aaa}\colon \uuu^{\op}&\lra\Cat(k),& \QPr_{\aaa}\colon \uuu^{\op}&\lra\Cat(k),\\
U&\longmapsto \Pre(\underline{\aaa}|_U) & U&\longmapsto \QPr(\underline{\aaa}|_U)\\
u&\longmapsto u^*_\Pre & u&\longmapsto u^*_\QPr
\end{align*}
whose twist functors $c$ are both given by $\Pre(c)$, and $z$ given by identity.

We define the category of
 \textit{(right) twisted presheaves} over a twisted presheaf $\aaa$ with central twists by
\[
\Pre(\aaa)=\Des(\Pre_{\aaa}).
\]
and the category of
 \textit{(right) twisted quasi-coherent presheaves} by
\[
\QPr(\aaa)=\Des(\QPr_{\aaa}).
\]

Recall that we define $\Qch(\aaa)$ in \S \ref{parqcoh} for a general prestack $\aaa$. When restricted to the above situation, let us prove that it is equivalent to $\QPr(\aaa)$.

\begin{theorem}\label{thmqmodqpr}
Let $\aaa$ be a twisted presheaf of algebras on $\uuu$  having central twists. Then $Q=(Q^U,\tau^u)_{U,u}$ gives an equivalence $$\Qch(\aaa)\cong\QPr(\aaa)$$ of $k$-linear categories where $\tau^u$ is given in \eqref{eqtau}.
\end{theorem}

\begin{proof}
By Lemma \ref{lemqch1} and Proposition \ref{propdeseq}, it suffices to check that $Q\colon \Mod_{\aaa} \lra \QPr_{\aaa}$ is a pseudonatural transformation.

Let us first prove \eqref{eqpseudonattran1}. For any $M\in\Mod(\aaa(U))$ and $w\colon T\lra W\in\uuu/W$, we have that
\[
\bigl(\tau^{uv}(\Pre(c)^{u,v}\circ Q^U)\bigr)_M^w=\tau^{uv,w}_M\circ\Pre(c)^{u,v,w}_{\widetilde{M}}=(\can_M^{uv,w})^{-1}\circ w^*(c^{u,v})_r
\]
as a homomorphism $M\otimes_{uvw}\aaa(T)\lra M\otimes_{uv}\aaa(W)\otimes_w\aaa(T)$ which maps $m\otimes a$ to $m\otimes 1 \otimes aw^*(c^{u,v})$. Likewise, the three natural transformations $v^*_{\QPr}\circ \tau^u$, $\tau^v\circ(-\otimes_u\aaa(V))$, $Q^W\circ\Mod(c)^{u,v}$, when acting on $M$ and $w$, are equal to the homomorphisms
\begin{gather*}
\begin{aligned}
(\can_M^{u,vw})^{-1}\colon M\otimes_{uvw}\aaa(T)&\lra M\otimes_{u}\aaa(V)\otimes_{vw}\aaa(T)\\
m\otimes a&\longmapsto m\otimes 1 \otimes a,
\end{aligned}\\
\begin{aligned}
(\can_{M\otimes_u\aaa(U)}^{v,w})^{-1}\colon M\otimes_{u}\aaa(V)\otimes_{vw}\aaa(T)&\lra M\otimes_{u}\aaa(V)\otimes_v\aaa(W)\otimes_w\aaa(T)\\
m\otimes a'\otimes a&\longmapsto m\otimes a' \otimes 1 \otimes a,
\end{aligned}\\
\begin{aligned}
\Mod(c)_M^{u,v}\otimes 1_{\aaa(T)}\colon M\otimes_{u}\aaa(V)\otimes_v\aaa(W)\otimes_w\aaa(T)&\lra M\otimes_{uv}\aaa(W)\otimes_{w}\aaa(T)\\
m\otimes a'' \otimes a' \otimes a&\longmapsto m\otimes c^{u,v}v^*(a'')a' \otimes a,
\end{aligned}
\end{gather*}
respectively. After composing them, we have
\[
\tau^{uv}(\Pre(c)^{u,v}\circ Q^U)=(Q^W\circ\Mod(c)^{u,v})(\tau^v\circ(-\otimes_u\aaa(V)))(v^*_{\QPr}\circ \tau^u),
\]
finishing the verification of \eqref{eqpseudonattran1}.

Next we will prove \eqref{eqpseudonattran2}. This follows from the fact that the twists $z^U$ of $\QPr_\aaa$ are all identity transformations.
\end{proof}

Likewise, by considering left modules, we can also introduce $\Pre^l_\aaa$,  $\QPr^l_\aaa$, and so on. In particular, we obtain $\Qch^l(\aaa)\cong\QPr^l(\aaa)$ by applying Theorem \ref{thmqmodqpr} to $\aaa^{\op}$.

\subsection{The Grothendieck property}\label{parpargroth}
One of the important properties of categories of quasi-coherent sheaves on schemes is the Grothendieck property. This property was originally established for quasi-compact quasi-separated schemes in \cite{EGA}. More recently, a proof for arbitrary schemes was given by Gabber (see \cite[Lemma 2.1.7]{conrad}) and in \cite{enochsestrada}, the authors present a general proof making use of the category of quasi-coherent modules over a ring representation of a quiver.
Recall that a Grothendieck category is a cocomplete abelian category with exact filtered colimits and a set of generators. By the Gabriel-Popescu theorem, Grothendieck categories can be considered as additive topoi (see \cite{lowen1}). The terminology in the following definition is inspired by topos theory:

\begin{definition}\label{defgeom}
Let $\uuu$ be a small category. A prestack  $\aaa: \uuu^{\op} \lra \Cat(k)$ is called \emph{geometric} if for every $u: V \lra U$ in $\uuu$, the restriction functor
$$- \otimes_u \aaa(V): \Mod(\aaa(U)) \lra \Mod(\aaa(V))$$
is exact.
\end{definition}

In this section, building on the argument from \cite{enochsestrada}, we prove the following:

\begin{theorem}\label{thmgroth}
Let $\uuu$ be a small category and let $\aaa$ be a geometric prestack on $\uuu$. The category $\Qch(\aaa)$ is a Grothendieck abelian category.
\end{theorem}
We make use of cardinalities.
For a small category $\uuu$, put
$|\uuu| = \sum_{U,V \in \uuu} |\uuu(V,U)|$.
For a small $k$-linear category $\AAA$ and $M \in \Mod(\AAA)$, a subset $X \subseteq M$ by definition consists of subsets $X(A) \subseteq M(A)$ for all $A \in \AAA$. Put $|X| = \sum_{A \in \AAA} |X(A)|$.
For a prestack $\aaa: \uuu \lra \Cat(k)$ on $\uuu$, put $|\aaa| = \sum_{U \in \uuu}|\aaa(U)|$. For $M = (M_U)_U \in \Mod(\aaa)$, a subset $X \subseteq M$ by definition consists of subsets $X_U \subseteq M_U$ for every $U \in \uuu$. Put $|X| = \sum_{U \in \uuu}|X_U|$.

\begin{proof}
The restriction functor is exact by assumption and preserves all colimits. By Proposition \ref{lemqchabelian} (3), it follows that $\Qch(\aaa)$ is an abelian category with exact filtered colimits. It remains to show that $\Qch(\aaa)$ has a set of generators. Let $\kappa$ be a cardinal with $\kappa \geq \sup \{ \aleph_0, |\uuu|,  |\aaa|\}$.
By Proposition \ref{propgroth}, for every $M \in \Qch(\aaa)$, $U \in \uuu$, $A \in \aaa(U)$ and $x \in M_U(A)$, there exists a subobject $N \subseteq M$ in $\Qch(\aaa)$ with $x \in N_U(A)$ and $|N| \leq \kappa$. Let $\Qch(\aaa)_0$ be a skeletal subcategory of $\Qch(\aaa)$. Then the objects $N \in \Qch(\aaa)_0$ with $|N| \leq \kappa$ constitute a set of generators for $\Qch(\aaa)$. 
\end{proof}

\begin{remarks}
\begin{enumerate}
\item A simplification of the argument shows that for an arbitrary prestack on $\uuu$, the category $\Mod(\aaa)$ is Grothendieck. In fact, it is well known that a stronger property holds, namely $\Mod(\aaa)$ has a set of finitely generated projective generators $P^A_U$ for $U \in \uuu$, $A \in \aaa(U)$ with $(P^A_U)_V(B) = \oplus_{u: V \lra U}\aaa(V)(B, u^{\ast} A)$ (see \cite{lowenvandenberghab}, \cite{lowenvandenberghCCT}). 
\item
Making use of cardinalities of objects in Grothendieck categories in the sense of \cite{lowenvandenberghab}, it is possible to investigate the more general question when a descent category of Grothendieck categories inherits the Grothendieck property. This question will be addressed in \cite{dinhvan}.
\end{enumerate}
\end{remarks}

The main result used in the proof of Theorem \ref{thmgroth} is the following:

\begin{proposition}\label{propgroth}
Let $\uuu$ be a small category and let $\aaa$ be a geometric prestack on $\uuu$. Let $\kappa$ be an cardinal with $\kappa \geq \sup\{ \aleph_0, |\uuu|, |\aaa|\}$. Consider $M \in \Qch(\aaa)$ and suppose that $X \subseteq M$ is a subset with $|X|\leq\kappa$. Then there is a subobject $N \subseteq M$ in $\Qch(\aaa)$ with $X \subseteq N$ and $|N|\leq\kappa$.
\end{proposition}

The proof of Proposition \ref{propgroth}, which will be completed at the end of this section, makes use of the notions of quivers, $\Cat(k)$-representations of quivers, and (quasi-coherent) representations over $\Cat(k)$-repre\-sen\-tations. These notions correspond to \emph{part of the data, with part of the axioms} defining categories, prestacks on categories, and (quasi-coherent) modules over prestacks respectively. The weaker notions are not used elsewhere in the paper. 

A \emph{quiver} $\uuu$ consists of a set of objects $\Ob(\uuu)$ and for $U, V \in \uuu$, a set of morphisms $\uuu(U,V)$. For a quiver $\uuu$, a \emph{$\Cat(k)$-representation of $\uuu$} consists of the following data:
\begin{itemize}
\item for every $U \in \uuu$ a small $k$-linear category $\aaa(U)$;
\item for every $u: V \lra U$ in $\uuu$ a $k$-linear functor $u^{\ast}: \aaa(U) \lra \aaa(V)$.
\end{itemize}
If for every $u: V \lra U$ in $\uuu$, the restriction functor
$$- \otimes_u \aaa(V): \Mod(\aaa(U)) \lra \Mod(\aaa(V))$$
is exact, the $\Cat(k)$-representation $\aaa$ is called \emph{geometric}.

For a $\Cat(k)$-representation $\aaa$ of $\uuu$, a \emph{representation} $M = (M_U)_U$ over $\aaa$ consists of the following data:
\begin{itemize}
\item for every $U \in \uuu$ an object $M_U \in \Mod(\aaa(U))$;
\item for every $u: V \lra U$ in $\uuu$ a morphism $\varphi_u: M_U \otimes_u \aaa(V) \lra M_V$.
\end{itemize}
If the morphism $\varphi_u$ is an isomorphism for all $u \in \uuu$, $M$ is a \emph{quasi-coherent representation}. A morphism of (quasi-coherent) representations $f: (N_U)_U \lra (M_U)_U$ consists of a compatible collection of morphisms $f_U: N_U \lra M_U$. We thus obtain the category $\mathsf{Rep}(\aaa)$ of representations over $\aaa$ and its full subcategory $\mathsf{QRep}(\aaa)$ of quasi-coherent representations.
A subset $X \subseteq M$ of a representation $M$ consists of subsets $X_U \subseteq M_U$ for all $U \in \uuu$.

Cardinalities of quivers, $\Cat(k)$-representations, and their (subsets of) representations are defined in complete analogy with cardinalities of small categories, prestacks and their (subsets of) modules.

Obviously, a small category $\uuu$ can be seen as a quiver by forgetting about the composition. The resulting forgetful functor from small categories to quivers has a left adjoint path category functor. For a quiver $\uuu$, the path category $\ppp(\uuu)$ has $\Ob(\ppp(U)) = \Ob(\uuu)$ and $\ppp(\uuu)(V, U) = \cup_{n \in \N} \ppp(\uuu)_n(V,U)$  where $\ppp(\uuu)_n(V,U)$ consists of paths of length $n$
\begin{equation}\label{eqp0}
p = (\xymatrix{ {V = U_0} \ar[r]^-{u_1} & {U_{1}} \ar[r]^-{u_2} & {\cdots} \ar[r]^-{u_{n-1}} & {U_{n-1}} \ar[r]^-{u_n} & {U_n = U}}).
\end{equation}
We have 
$$\ppp_0(\uuu)(V,U) = \begin{cases} \{1_U\} & \text{if} \,\,\, V = U \\ \varnothing & \text{otherwise} \end{cases}$$
where $1_U$ denotes the unique path of length $0$ from $U$ to $U$. The composition in $\ppp(\uuu)$ is given by concatenation of paths, and the paths $1_U$ are the identity morphisms.

If $\aaa: \uuu^{\op} \lra \Cat(k)$ is a prestack on a small category $\uuu$, then $\aaa$ can be seen as a $\Cat(k)$-representation of the quiver $\uuu$.
If $\uuu$ is a quiver with $\Cat(k)$-representation $\aaa$, we define a prestack $\ppp(\aaa)$ on $\ppp(\uuu)$ as follows. For $U \in \uuu$, we put $\ppp(\aaa)(U) = \aaa(U)$. For $p \in \ppp(\uuu)(V,U)$ as in \eqref{eqp0}, we put 
$$p^{\ast} = u_1^{\ast} u_2^{\ast} \dots u_{n-1}^{\ast} u_n^{\ast}: \aaa(U) \lra \aaa(V)$$
and $1_U^{\ast} = 1_{\aaa(U)}$.
Then $\ppp(\aaa)$ becomes a prestack by taking all the twist isomorphisms to be identity morphisms, i.e. $\ppp(\aaa)$ is a presheaf of $k$-linear categories. We obtain an equivalences of categories $\mathsf{Rep}(\aaa) \cong \Mod(\ppp(\aaa))$ which restricts to an equivalence
\begin{equation}\label{repmod}
\mathsf{QRep}(\aaa) \cong \Qch(\ppp(\aaa)).
\end{equation}
Under these equivalences, a representation $M = (M_U)_U$ of $\aaa$ is uniquely extended to a $\ppp(\aaa)$-module $\ppp(M) = (M_U)_U$ by imposing the compatibility relation between the morphisms $\varphi_u: M_U \otimes_u \aaa(V) \lra M_V$ with respect to concatenation. 

If $\aaa: \uuu \lra \Cat(k)$ is a prestack on a small category $\uuu$, then the category $\Mod(\aaa)$ is a full subcategory of $\mathsf{Rep}(\aaa)$, and $\Qch(\aaa)$ is a full subcategory of $\mathsf{QRep}(\aaa)$.

We will prove Proposition \ref{propgroth} in the following steps:
\begin{enumerate}
\item In Lemma \ref{lemgroth}, we show that Proposition \ref{propgroth} holds if $\aaa$ is a presheaf of $k$-linear categories on $\uuu$.
\item In Proposition \ref{grothcor}, by \eqref{repmod} we deduce from Lemma \ref{lemgroth} applied to $\ppp(\aaa)$ the corresponding statement for the category $\mathsf{QRep}(\aaa)$ of a $\Cat(k)$-re\-pre\-sen\-ta\-tion $\aaa$.
\item In Lemma \ref{lemsub}, we show that a subrepresentation of a module over a prestack $\aaa$ is automatically a submodule, whence Proposition \ref{grothcor} implies Proposition \ref{propgroth}.
\end{enumerate}

\begin{lemma}\label{lemmodgen1}
Let $\AAA$ be a small $k$-linear category. Consider $M \in \Mod(\AAA)$ and a subset $X \subseteq M$. The smallest submodule $N \subseteq M$ with $X \subseteq N$ satisfies $|N| \leq |X|\cdot |\AAA|$.
\end{lemma}

\begin{proof}
For $f: A \lra B$, let $N^f(A) \subseteq M(A)$ be the image of $X(B)$ under $M(f): M(B) \lra M(A)$. 
The submodule $N \subseteq M$ satisfies $N(A) = \sum_{f: A \lra B} N^f(A)$.
\end{proof}

\begin{lemma}\label{lemmodgen}
Let $\uuu$ be a small category and let $\aaa$ be a presheaf of small $k$-linear categories on $\uuu$. Let $\kappa$ be a cardinal with $\kappa \geq \sup\{ \aleph_0, |\uuu|, |\aaa|\}$. Consider $M \in \mathsf{Mod}(\aaa)$ and a subset $X \subseteq M$. The smallest subobject $N \subseteq M$ in $\mathsf{Mod}(\aaa)$ with $X \subseteq N$ satisfies $|N|\leq |X| \cdot \kappa$.
\end{lemma}

\begin{proof}
For $U \in \uuu$, let $Y_U \subseteq M_U$ be the smallest submodule of $M_U$ with $X_U \subseteq Y_U$. For $u: V \lra U$, let $N_V^u \subseteq M_V$ be the image of $Y_U \otimes_u \aaa(V)$ under $\varphi^M_u: M_U \otimes_u \aaa(V) \lra M_V$. 
The submodule $N \subseteq M$ satisfies $N_V = \sum_{u: V \lra U} N_V^u$.
\end{proof}

\begin{lemma}\label{lemqchgen}
Proposition \ref{propgroth} holds true for $\aaa: \uuu^{\op} \lra \Cat(k)$ a geometric presheaf of categories, i.e. a geometric prestack with all twists given by identity morphisms, on $\uuu = \ppp(e: 2 \lra 1)$.
\end{lemma}

\begin{proof}
Consider the isomorphism $\varphi^M_e: M_1 \otimes_e \aaa(2) \lra M_2$. We have $\kappa \geq \{ \aleph_0, |\aaa(1)|, |\aaa(2)|\}$. From the existence of epimorphisms
$$\bigoplus_{B \in \aaa(1)} M_1(B) \otimes_k \aaa_2(A, \aaa(e)(B)) \lra M_2(A)$$
we see that there exists a submodule $Y_1 \subseteq M_1$ with $|Y_1| \leq \kappa$ and $X_2 \subseteq \varphi^M_e(Y_1 \otimes_e \aaa(2))$. Let $N_1$ be the smallest submodule of $M_1$ with $X_1 \cup Y_1 \subseteq N_1$. Since $- \otimes_e \aaa(2)$ is exact, we have $N_1 \otimes_e \aaa(2) \subseteq M_1 \otimes_e \aaa(2)$ and putting $N_2 = \varphi^M_e(N_1 \otimes_e \aaa(2)) \subseteq M_2$ yields the desired quasi-coherent submodule $N \subseteq M$.
\end{proof}

\begin{lemma}\label{lemgroth}
Proposition \ref{propgroth} holds true for $\aaa: \uuu^{\op} \lra \Cat(k)$ a geometric presheaf of categories, i.e. a geometric prestack with all twists given by identity morphisms.
\end{lemma}

\begin{proof}
This is proven along the lines of \cite[Prop. 3.3]{enochsestrada}, by combining Lemmas \ref{lemmodgen} and \ref{lemqchgen} in a transfinite induction on $\N \times \Mor(\uuu)$ (after well-ordering $\Mor(\uuu) = \coprod_{V, U \in \uuu}\uuu(V,U)$).
\end{proof}

\begin{proposition}\label{grothcor}
Let $\uuu$ be a quiver and let $\aaa$ be a geometric $\Cat(k)$-repre\-sen\-ta\-tion of $\uuu$.  Let $\kappa$ be a cardinal with $\kappa \geq \sup\{ \aleph_0, |\uuu|, |\aaa|\}$. Consider $M \in \mathsf{QRep}(\aaa)$ and suppose that $X \subseteq M$ is a subset with $|X |\leq\kappa$. Then there is a subobject $N \subseteq M$ in $\mathsf{QRep}(\aaa)$ with $X \subseteq N$ and $|N|\leq\kappa$.
\end{proposition}

\begin{proof}
Let $\ppp(\uuu)$ be the path category of $\uuu$, and let $\ppp(\aaa)$ be the associated geometric prestack on $\ppp(\aaa)$. Note that $\sup\{ \aleph_0, |\ppp(\uuu)|, |\ppp(\aaa)|\} \leq \kappa$. 
Consider the unique extension of $M$ to $M = (M_U)_U \in \Qch(\ppp(\aaa))$ under the equivalence \eqref{repmod}. According to Lemma \ref{lemgroth}, there is a subobject $N \subseteq M$ in $\Qch(\ppp(\aaa))$ with $X_U \subseteq N_U$ for all $U$ and with $|N| \leq \kappa$. Under the equivalence \eqref{repmod}, $N$ corresponds to a subobject $N \subseteq M$ in $\mathsf{QRep}(\aaa)$ with the desired properties.
\end{proof}

\begin{lemma}\label{lemsub}
Let $\uuu$ be a small category and let $\aaa$ be a geometric prestack on $\uuu$. The full subcategory $\Mod(\aaa) \subseteq \mathsf{Rep}(\aaa)$ is closed under subobjects, i.e. if we have $N \subseteq M$ in $\mathsf{Rep}(\aaa)$ with $M \in \Mod(\aaa)$, then we also have $N \in \Mod(\aaa)$. Similarly, the full subcategory $\Qch(\aaa) \subseteq \mathsf{QRep}(\aaa)$ is closed under subobjects. 
\end{lemma}

\begin{proof}
Since $M=(M_U,\varphi^M_u) \in \Mod(\aaa)$, we have a commutative diagram
 \begin{equation}\label{diagram}
  \xymatrix{
  M_U\otimes_u\aaa(V)\otimes_v\aaa(W) \ar[r]^-{\varphi^M_u\otimes 1}\ar[d]_-{\Mod(c)^{u,v}_{M_U}} & M_V\otimes_v\aaa(W) \ar[d]^-{\varphi^M_v} \\
  M_U\otimes_{uv}\aaa(W) \ar[r]_-{\varphi^M_{uv}} & M_W
  }
  \end{equation}
  in $\Mod(\aaa(W))$ for any $u\colon V\lra U$, $v\colon W\lra V$. 
 Since $N = (N_U, \varphi_u^N)$ is a subrepresentation of $M$, and since $- \otimes_u \aaa(V)$ is exact by assumption, we have a commutative diagram
  $$\xymatrix{ {M_U \otimes_u \aaa(V)} \ar[r]^-{\varphi^M_u} & {M_V} \\  {N_U \otimes_u \aaa(V)} \ar[u] \ar[r]_-{\varphi^N_u} & {N_V} \ar[u] }$$
  in which the vertical arrows are monomorphisms, i.e. we have $\varphi^N_u = \varphi^M_u|_N$. Similarly, we have $\varphi^N_{uv} = \varphi^M_{uv}|_N$ and $\varphi^N_u \otimes 1 = (\varphi^M_u \otimes 1)|_N$. 
  Further, since the twist $\Mod(c)^{u,v}$ is a natural transformation, we also have $\Mod(c)^{u,v}_{N_U} = \Mod(c)^{u,v}_{M_U}|_N$. 
  It follows that diagram \eqref{diagram} for $N$ instead of $M$ is obtained as a subdiagram of \eqref{diagram}, whence it commutes.  
 We conclude that $N\in \Mod(\aaa)$. 
\end{proof}

\begin{proof}[Proof of Proposition \ref{propgroth}]
Consider $\uuu$ as a quiver, $\aaa$ as a geometric $\Cat(k)$-repre\-sen\-ta\-tion of $\uuu$, and $M$ as a quasi-coherent representation of $\aaa$. By Proposition \ref{grothcor}, there is an $N \subseteq M$ in $\mathsf{QRep}(\aaa)$ with $X \subseteq N$ and $|N|\leq\kappa$. By Lemma \ref{lemsub}, we have $N \in \Qch(\aaa)$.
\end{proof}

\subsection{Flatness}\label{parflat}
In the context of deformation theory, the notion of flatness is fundamental. An appropriate notion of flatness for abelian categories was introduced in \cite[\S 3]{lowenvandenberghab}.

Let $\ccc$ be an abelian category. Let $\mmod(k)$ be the category of finitely presented $k$-modules. Recall first that an object $C \in \ccc$ is called \emph{flat} if the natural finite colimit preserving functor $- \otimes_k C\colon \mmod(k) \lra \ccc$, $k \longmapsto C$ is exact, and \emph{coflat} if the natural finite limit preserving  functor $\Hom_k(-, C)\colon \mmod(k) \lra \ccc$ is exact.
Flatness for abelian categories is a selfdual notion which naturally extends the usual notion of flatness for $k$-algebras, as the following proposition shows. More generally, it can easily be characterized for an arbitrary abelian category with enough projectives (or, dually, with enough injectives).

\begin{proposition} \cite[\S 3]{lowenvandenberghab}.
\begin{enumerate}
\item Let $\AAA$ be a $k$-linear category. The abelian category $\Mod(\AAA)$ is flat if and only if the modules $\AAA(A,A')$ are flat for all $A, A' \in \AAA$.
\item Let $\ccc$ be an abelian category with enough injectives. Then $\ccc$ is flat if and only if every injective object in $\ccc$ is coflat.
\end{enumerate}

\end{proposition}

In order that the descent category inherits flatness, we restrict our  setting.

\begin{definition}\label{defaf}\label{defafps}
Let $\ccc$ be a prestack on a small category $\uuu$ with pullbacks and binary products. For $U, V, W \in \uuu$, $v\colon V \lra U$, $w\colon W \lra U$ in $\uuu$, consider the pullback
\begin{equation}\label{eqpull}
\xymatrix{ {V} \ar[r]^v & U \\ {V \cap W} \ar[u]^{w_1} \ar[r]_-{v_1} & W \ar[u]_w }
\end{equation}
with $u = vw_1 = wv_1$.
We call $\ccc$ a \emph{prestack of affine localizations} if the following conditions are fulfilled for all objects and morphisms involved:
\begin{enumerate}
\item The category $\ccc(U)$ is a Grothendieck abelian category.
\item The functor $v^{\ast}\colon \ccc(U) \lra \ccc(V)$ is exact.
\item The functor $v^{\ast}$ has a right adjoint $v_{\ast}\colon \ccc(V) \lra \ccc(U)$ such that:
\begin{enumerate}
\item $v_{\ast}\colon \ccc(V) \lra \ccc(U)$ is fully faithful;
\item $v_{\ast}\colon \ccc(V) \lra \ccc(U)$ is exact.
\end{enumerate}
\item We have natural isomorphisms $$(v_{\ast}v^{\ast})(w_{\ast}w^{\ast}) \cong u_{\ast} u^{\ast} \cong (w_{\ast}w^{\ast})(v_{\ast}v^{\ast}).$$
\end{enumerate}
\end{definition}

\begin{remark}
In Definition \ref{defafps}, conditions (2) and (3a) together mean that $v_{\ast}$ is a localization, and (4) implies that any two such localizations $v_{\ast}$ and $w_{\ast}$ are compatible.
\end{remark}

Let $\uuu$ be a finite poset with binary meets and let $\uuu^{\star} = \uuu \cup \{ \star\}$ be the poset $\uuu$ with top $\star$ adjoined. Note that $\uuu^{\star}$ is a lattice. Conversely, let $\uuu^{\star}$ be a finite lattice with top $\star$ and $\uuu = \uuu^{\star} \setminus \{\star\}$. Then $\uuu$ is a finite poset with binary meets. Posets are considered as categories in the usual way.
For $U, V \in \uuu$, the meet of $U$ and $V$ is denoted by $U \cap V$.

\begin{proposition}\label{propdesflat1}
Let $\uuu^{\star}$ be a finite lattice with top $\star$ and $\uuu = \uuu^{\star} \setminus \{\star\}$. Let $\ccc$ be a prestack of affine localizations on $\uuu$.
\begin{enumerate}
\item There is a prestack of affine localizations $\ccc^{\star}$ on $\uuu^{\star}$ with $\ccc^{\star}|_{\uuu} = \ccc$ and $\ccc^{\star}(\star) = \Des(\ccc)$. For the unique map $u_U\colon U \lra \star$, we have $$u^{\ast}_U = \pi_U\colon \Des(\ccc) \lra \ccc(U).$$
Consider the pullback in $\uuu^{\star}$:
$$\xymatrix{ U \ar[r]^{u_U} & {\star} \\ {U \cap V} \ar[u]^{u^U_{U \cap V}} \ar[r]_{u^V_{U \cap V}} & V \ar[u]_{u_V} }$$
The right adjoint $u_{U, \ast}$ of $u_U^{\ast}$ satisfies
\begin{equation} \label{eqkey}
u_V^{\ast} u_{U, \ast} = u_{U \cap V, \ast}^V u_{U \cap V}^{U, \ast}.
\end{equation}
\item If the abelian categories $\ccc(U)$ are flat over $k$ for $U \in \uuu$, then so is $\ccc^{\star}(\star) = \Des(\ccc)$.
\end{enumerate}
\end{proposition}

\begin{proof}
Consider a pullback \eqref{eqpull}. By condition (4), we obtain canonical natural isomorphisms $w^{\ast} v_{\ast} \cong v_{1, \ast} w_1^{\ast}$ compatible with the pseu\-do\-functor isomorphisms of $\ccc$. For $C \in \ccc(U)$, this allows the definition of a descent datum $u_{U, \ast}(C)$ based upon \eqref{eqkey}. The resulting functor $u_{U, \ast}$ is exact by conditions (2) and (3b) and fully faithful by formula \eqref{eqkey} applied to $V = U$ since $U \cap U = U$, $u^U_U = 1_U$ and $u^{U, \ast}_U \cong 1_{\ccc(U)} \cong u^U_{U, \ast}$.

Since $\uuu$ is finite, the fact that $\Des(\ccc)$ inherits flatness is proven like \cite[Proposition 3.12]{lowenprestacks}.
\end{proof}

\begin{definition}\label{defafpre}
Let $\uuu$ be a poset with binary meets and let $\aaa\colon \uuu^{\op} \lra \Cat(k)$ be a prestack on $\uuu$.

We call $\aaa$ a \emph{(right) semi-separated prestack} (resp. a \emph{left semi-separated prestack}) if the associated prestack $\Mod^r_\aaa$ (resp. $\Mod^l_\aaa$) is a prestack of affine localizations in the sense of Definition \ref{defaf}.

If $\aaa\colon \uuu^{\op} \lra \Alg(k)$ is a (twisted) presheaf of $k$-algebras, we call $\aaa$ a \emph{(right) semi-separated (twisted) presheaf} (resp. a \emph{left semi-separated (twisted) presheaf}) if the corresponding prestack is right (resp. left) semi-separated.

If the poset $\uuu$ is finite, we call $\aaa$ a \emph{quasi-compact} prestack (or (twisted) presheaf).
\end{definition}

Note that a semi-separated prestack is geometric in the sense of Definition \ref{defgeom}.
Clearly, in Definition \ref{defaf}, conditions (1) and (3b) are automatically fulfilled for $\ccc = \Mod^r_\aaa$, and the remaining conditions can be made explicit.
In particular, for twisted presheaves of $k$-algebras we obtain:

\begin{proposition}\label{propafftw}
Let $\uuu$ be a poset with binary meets and let $\aaa\colon \uuu^{\op} \lra \Alg(k)$ be a twisted presheaf of $k$-algebras on $\uuu$. Then $\aaa$ is a right semi-separated twisted presheaf if and only if the following conditions are fulfilled for all $U$, $V$, $W \in \uuu$ with $V \leq U$, $W \leq U$:
\begin{enumerate}
\item[(i)] The restriction map $\aaa(U) \lra \aaa(V)$ is a right flat epimorphism of rings.
\item[(ii)] We have isomorphisms of $\aaa(U)$-bimodules $$\aaa(V) \otimes_{\aaa(U)} \aaa(W) \cong \aaa(V \cap W) \cong \aaa(W) \otimes_{\aaa(U)} \aaa(V).$$
\end{enumerate}
\end{proposition}

\begin{proof}
According to Lemmas \ref{lemloc} and \ref{lemcomp}, condition (i) (resp. (ii)) is equivalent to condition (2) + (3a) (resp. (4)) from Definition \ref{defaf}.
\end{proof}

\begin{remark}
Proposition \ref{propafftw} can easily be adapted to the case of a general prestack $\aaa\colon \uuu^{\op} \lra \Cat(k)$ based upon \cite{krauseepi}, where a linear functor $f: \AAA \lra \BBB$ inducing a fully faithful forgetful functor $\Mod(\BBB) \lra \Mod(\AAA)$ is characterized as a so called ``epimorphism up to direct factors''.
\end{remark}

We have the following corollary of Proposition \ref{propdesflat1} and Theorem \ref{thmgroth}:

\begin{proposition}\label{propdesflat}
Let $\uuu^{\star}$ be a finite lattice with top $\star$ and $\uuu = \uuu^{\star} \setminus \{\star\}$. Let $\aaa$ be quasi-compact semi-separated prestack on $\uuu$.
\begin{enumerate}
\item We obtain a prestack of affine localizations $\Mod_\aaa^{\star}$ on $\uuu^{\star}$ with $\Mod_\aaa^{\star}|_{\uuu} = \Mod_\aaa$. The category $\Mod_\aaa^{\star}(\star) = \Qch(\aaa)$ is a Grothendieck abelian category.
\item If the modules $\aaa(U)(A,A')$ are flat over $k$ for $A,A' \in \aaa(U)$ and  $U \in \uuu$, then the abelian category $\Mod_\aaa^{\star}(\star) = \Qch(\aaa)$ is flat over $k$.
\end{enumerate}
\end{proposition}

\begin{example}
Let $X$ be a scheme over $k$ with a finite semi-separating cover $\uuu$ and let $\uuu^{\star}$ be the same cover with $\star = X$ adjoined. The restricted structure sheaf $\ooo = \ooo_X|_{\uuu}$ constitutes a quasi-compact semi-separated presheaf of $k$-algebras. We obtain the prestack of affine localizations $\Mod_\ooo^{\star}$ on $\uuu^{\star}$ with $\Mod_\ooo^{\star} |_{\uuu} = \Mod_\ooo$ and $\Mod_\ooo^{\star}(X) = \Qch(\ooo) \cong \Qch(X)$, which is Grothendieck. The prestack $\Mod_\ooo^{\star}$ is equivalent to the prestack $V \longmapsto \Qch(V)$ on $X$, restricted to $\uuu^{\star}$. If the algebras $\ooo(U)$ are flat over $k$ for $U \in \uuu$, the abelian category $\Qch(X)$ is flat over $k$.
\end{example}

\section{Comparison of abelian deformations}\label{parparcompab}

The infinitesimal deformation theory of abelian categories was developed in \cite{lowenvandenberghab} as a natural extension of the classical Gerstenhaber deformation theory of algebras \cite{gerstenhaber}. The relation between the two theories is best understood by means of the following basic result from \cite{lowenvandenberghab}: if $A$ is an algebra with module category $\Mod(A)$, then there is an equivalence between algebra deformations of $A$ and abelian deformations of $\Mod(A)$. This equivalence is given by:
$$
\Def_{\mathrm{alg}}(A) \lra \Def_{\mathrm{ab}}(\Mod(A)),\quad B \longmapsto \Mod(B).
$$
Our first aim in this section is to prove a counterpart of this theorem for twisted presheaves. Precisely, in \S \ref{partwistab}, we prove that for $\aaa$ a quasi-compact semi-separated twisted presheaf, there is an equivalence between twisted deformations of $\aaa$ and abelian deformations of $\Qch(\aaa)$ (Theorem \ref{defeq}). This equivalence is given by:
\begin{equation} \label{eqtwabintro1}
\Def_{\mathrm{tw}}(\aaa) \lra \Def_{\mathrm{ab}}(\Qch(\aaa)),\quad \bbb \longmapsto \Qch(\bbb).
\end{equation}
In \S \ref{pardefab}, we briefly recall the necessary background from \cite{lowenvandenberghab} in order to make our presentation of the proof of Theorem \ref{defeq} self contained.

If we compose \eqref{eqtwabintro1} with the isomorphism \eqref{eqclass}, we thus obtain a canonical isomorphism
$$
\Psi^{\aaa}\colon HH^2(\aaa) \lra \Def_{\mathrm{tw}}(\aaa) \lra \Def_{\mathrm{ab}}(\Qch(\aaa)).
$$
For our purpose, we instead look at the isomorphism
$$
\Psi^{\aaa^{\op}}(-)^{\op}\colon HH^2(\aaa) \lra HH^2(\aaa^{\op}) \lra \Def_{\mathrm{tw}}(\aaa^{\op}) \lra \Def_{\mathrm{ab}}(\Qch^l(\aaa)).
$$
For $X$ a scheme with a finite semi-separating cover $\uuu$, taking $\aaa = \ooo_X|_{\uuu}$, the morphism $\Psi^{\aaa^{\op}}(-)^{\op}$ translates into a canonical isomorphism
\begin{equation}\label{intbas20}
\Psi^X\colon HH^2(\ooo_X|_{\uuu}) \lra \Def_{\mathrm{ab}}(\Qch(X)).
\end{equation}
On the other hand, if $X$ is furthermore smooth, in \cite{toda}, Toda associates to an element
$$u \in \check{H}^2(X, \ooo_X) \oplus \check{H}^1(X, \ttt_X) \oplus \check{H}^0(X, \wedge^2 \ttt_X)$$
a certain ``first order deformation'' of the abelian category $\Qch(X)$. This is an abelian $k[\epsilon]$-linear category $\Qch^l(X, u)$, which is not a priori a deformation in the sense of  \S \ref{pardefab}.

We know from \S \ref{secHodge-HKR} that
\[
HH^2(X) \cong HH^2(\ooo_X|_{\uuu}) \cong \check{H}^2(X, \ooo_X) \oplus \check{H}^1(X, \ttt_X) \oplus \check{H}^0(X, \wedge^2 \ttt_X),
\]
so both \eqref{intbas20} and Toda's construction give concrete ways to interpret a Hochschild $2$-class of a smooth scheme in terms of a certain $k[\epsilon]$-linear abelian category.
In \S \ref{partoda}, we use the isomorphisms
\begin{equation}\label{transfos0}
H^2\CC_\mathrm{GS}(\ooo_X|_{\uuu}) \cong \oplus_{p + q = 2} \check{H}^p(\uuu, \Lambda^q\ttt_X)
\end{equation}
from \S \ref{parHodgeHKR} and \S \ref{parcechsimp} in order to show that both constructions are equivalent, which in fact holds true for an arbitrary smooth scheme $X$ with semi-separating cover $\uuu$. Precisely, in Theorem \ref{thmequiv},
we show that if $\phi \in H^2\CC_\mathrm{GS}(\ooo_X)$ corresponds to $u \in \oplus_{p + q = 2} \check{H}^p(X, \Lambda^q\ttt_X)$ under \eqref{transfos0}, then for the image $\Psi^X(\phi) = \Qch^l(\bar{\ooo})$ of $\phi$ under \eqref{intbas20}, there is an equivalence of abelian categories
$$\Qch^l(\bar{\ooo}) \cong \Qch^l(X, u).$$
The proof of the theorem makes use of the intermediate category $\QPr^l(\bar{\ooo})$, which was shown to be equivalent to $\Qch^l(\bar{\ooo})$ in Theorem \ref{thmqmodqpr}.

In case $X$ is quasi-compact semi-separated, it further follows that $\Qch^l(X,u)$ is an abelian deformation of $\Qch(X)$ in the sense of \cite{lowenvandenberghab} and the general theory, including the obstruction theory for lifting objects \cite{lowen2}, applies. This is used for instance by Macr\`i and Stellari in the context of an infinitesimal derived Torelli theorem for K3 surfaces \cite[\S 3]{macristellari}.

\subsection{Deformations of abelian categories}\label{pardefab}

The infinitesimal deformation theory of abelian categories as developed in \cite{lowenvandenberghab} constitutes a natural extension of Gerstenhaber's deformation theory of algebras \cite{gerstenhaber}. In the current paper, we are concerned with first order deformations, i.e.\ we deform in the direction of the dual numbers $k[\epsilon]$. For abelian categories, we use the notion of flatness from \cite[Definition 3.2]{lowenvandenberghab} (see also \S \ref{pardescent}).

 For a $k[\epsilon]$-linear category $\ddd$, we define the full subcategory of $k$-linear objects
$$\ddd_k = \{ D \in \ddd \,\, |\,\, \epsilon 1_D = 0\} \subseteq \ddd.$$
Taking $k$-linear objects defines a functor from $k[\epsilon]$-linear categories to $k$-linear categories, which is right adjoint to the forgetful functor. The functor is best behaved when applied to abelian categories.

Let $\ccc$ be a flat $k$-linear abelian category. A \emph{first order deformation} of $\ccc$ is a flat $k[\epsilon]$-linear abelian category $\ddd$ with an equivalence $\ccc \cong \ddd_k$. An equivalence of deformations $\ddd_1$ and $\ddd_2$ is an equivalence of categories $\ddd_1 \cong \ddd_2$ such that the induced $(\ddd_1)_k \cong (\ddd_2)_k$ is compatible with the equivalences with $\ccc$. Let $\Def_{\mathrm{ab}}(\ccc)$ denote the set of first order abelian deformations of $\ccc$ up to equivalence.

In order to understand the structure of abelian deformations, a fundamental tool is the lifting of special objects to a deformation.

For a deformation $\ddd$ of $\ccc$, the natural inclusion functor $\ccc \lra \ddd$ has both a left adjoint $k \otimes_{k[\epsilon]} -\colon \ddd \lra \ccc$ and a right adjoint $\Hom_{k[\epsilon]}(k, -)\colon \ddd \lra \ccc$. There is an obstruction theory controlling the lifting of flat objects along $k \otimes_{k[\epsilon]} -$ (and, dually, for lifting coflat objects along $\Hom_{k[\epsilon]}(k, -)$) which was developed in \cite{lowen2}. In both cases, the obstruction against lifting is in $\Ext^2_{\ccc}(C,C)$, and the freedom is given by $\Ext^1_{\ccc}(C,C)$.
Important properties of objects are sometimes preserved under deformation, as the following proposition shows:

\begin{proposition}\label{liftproj}
Let $\ccc \lra \ddd$ be an abelian deformation. Let $\aaa \subseteq \ccc$ be a collection of finitely generated projective generators of $\ccc$. Let $\bar{\aaa}$ be any collection of flat lifts of objects in $\aaa$ along $k \otimes_{k[\epsilon]} -\colon \ddd \lra \ccc$, such that $\bar{\aaa}$ contains at least one flat lift of every object in $\aaa$. Then $\bar{\aaa}$ is a collection of finitely generated projective generators of $\ddd$.
\end{proposition}

For a flat $k$-linear category $\AAA$, let $\Def_{\mathrm{lin}}(\AAA)$ be the set of first order linear deformations of $\AAA$ up to equivalence. Here, linear deformations are obtained by simply treating a linear category as an algebra with several objects in the sense of \cite{mitchell} (in particular, linear deformations keep the object set fixed, and reduce to algebra deformations when applied to an algebra).
The following fundamental result is proven based upon Proposition \ref{liftproj}:
\begin{proposition}\label{propbas}\cite{lowenvandenberghab}
Let $\AAA$ be a $k$-linear category. There is an isomorphism
$$\Def_{\mathrm{lin}}(\AAA) \lra \Def_{\mathrm{ab}}(\Mod(\AAA)),\quad \BBB \longmapsto \Mod(\BBB).$$
\end{proposition}

In particular, the deformation of a module category is again a module category. More generally, the following was shown in \cite{lowenvandenberghab}:

\begin{proposition} \cite{lowenvandenberghab} \label{liftgroth}
Let $\ccc \lra \ddd$ be an abelian deformation. If $\ccc$ is a Grothendieck category, then so is $\ddd$.
\end{proposition}

Deformations of Grothendieck categories behave well with respect to localization in the following sense. Let $\iota\colon \ddd' \lra \ddd$ be a localization of $k[\epsilon]$-linear categories with left adjoint $a$. Then the induced functor $\iota_k\colon \ddd'_k \lra \ddd_k$ is a localization with left adjoint $a_k$. Suppose $\ddd$ is a first order deformation of a Grothendieck category $\ccc \cong \ddd_k$.  Let $\Lambda(\ddd)$ (resp. $\Lambda(\ddd_k)$) denote the set of localizations of $\ddd$ (resp. $\ddd_k$) up to equivalence. There is a natural map
\begin{equation}\label{eqres}
\Lambda(\ddd) \lra \Lambda(\ddd_k),\quad \iota \longmapsto \iota_k.
\end{equation}

The proof of the following theorem is based upon lifting localizing Serre subcategories to a deformation.
\begin{theorem}\cite{lowenvandenberghab} \label{thmloc11}
The map \eqref{eqres} is bijective.
\end{theorem}

Further, compatibility of localizations lifts under deformation.

\begin{proposition}\cite[Prop. 3.8]{lowenprestacks}\label{propliftcomp}
Consider localizations $\iota_1\colon \ddd_1 \lra \ddd$ and $\iota_2\colon \ddd_2 \lra \ddd$ of flat $k[\epsilon]$-linear Grothendieck categories. If $\iota_{1,k}\colon \ddd_{1,k} \lra \ddd_k$ and $\iota_{2,k}\colon \ddd_{2,k} \lra \ddd_k$ are compatible, then so are $\iota_1$ and $\iota_2$. Further, for strict localizations, we have $(\ddd_1 \cap \ddd_2)_k \cong \ddd_{1,k} \cap \ddd_{2,k}$.
\end{proposition}

According to \cite{lowenvandenberghab}, the property of a functor being a localization can itself be lifted under deformation under some circumstances. We have the following particular case:

\begin{proposition}\cite[Thm. 7.3]{lowenvandenberghab} \label{proploclift}
Let $\iota\colon \ddd' \lra \ddd$ be a right adjoint functor between flat $k[\epsilon]$-linear categories such that $\iota$ maps injective objects to coflat objects. If $\iota_k\colon \ddd'_k \lra \ddd_k$ is a localization between Grothendieck categories, then so is $\iota$.

In particular, the result applies if $\iota$ is a functor between flat $k[\epsilon]$-linear categories with an exact left adjoint.
\end{proposition}

Consider a commutative square
\begin{equation}\label{eqsquare}
\xymatrix{ {\BBB} \ar[r] \ar[d] & {\BBB'} \ar[d] \\ {\AAA} \ar[r] & {\AAA'} }
\end{equation}
in which $\BBB$ (resp. $\BBB'$) is a linear deformation of $\AAA$ (resp. $\AAA'$). We have the following corollary of Proposition \ref{proploclift}:

\begin{proposition}\label{propliftloc}
If in \eqref{eqsquare} $\Mod(\AAA') \lra \Mod(\AAA)$ is a localization, then so is $\Mod(\BBB') \lra \Mod(\BBB)$.
\end{proposition}

\subsection{From twisted to abelian deformations}\label{partwistab}

In this section, we present a counterpart to Proposition \ref{propbas} for twisted presheaves of algebras $\aaa\colon \uuu^{\op} \lra \Alg(k)$, or, more generally, for prestacks $\aaa\colon \uuu^{\op} \lra \Cat(k)$. In Theorem \ref{defeq}, we will adopt the setting of Proposition \ref{propdesflat} in order to do so.

Let $\uuu$ be a small category. For a prestack $\aaa\colon \uuu^{\op} \lra \Cat(k)$ of small $k$-linear categories, first order deformations and equivalences of deformations can be defined in complete analogy with the case of twisted presheaves in Definition \ref{deftwistedpresheaf} (see also Def. 3.24 in \cite{lowenmap}). We will refer to these deformations and equivalences as \emph{strict deformations} and \emph{strict equivalences}. Let $\Def^s_{\mathrm{tw}}(\aaa)$ be the set of strict deformations of $\aaa$ up to strict equivalence.

It will be convenient to consider a more relaxed notion of twisted deformations as well, based upon equivalences rather than isomorphisms of prestacks (see Proposition \ref{propdeseq}).

A prestack $\aaa\colon \uuu^{\op} \lra \Cat(R)$ of $R$-linear categories is called \emph{flat} if for every $U \in \uuu$ and $A, A' \in \aaa(U)$, the $R$-module $\aaa(U)(A,A')$ is flat.

\begin{definition}
Let $\aaa\colon \uuu^{\op} \lra \Cat(k)$ be a flat prestack of small $k$-linear categories.
\begin{enumerate}
\item A \emph{first order deformation} of $\aaa$ is a flat prestack $\bbb$ of $k[\epsilon]$-linear categories with an equivalence of prestacks $k \otimes_{k[\epsilon]} \bbb \cong \aaa$.
\item An \emph{equivalence} of deformations $\bbb_1$ and $\bbb_2$ is an equivalence of prestacks $\bbb_1 \cong \bbb_2$ such that the induced equivalence $k \otimes_{k[\epsilon]} \bbb_1 \cong k \otimes_{k[\epsilon]} \bbb_2$ is compatible with the equivalences with $\aaa$.
\end{enumerate}
\end{definition}

Let $\Def_{\mathrm{tw}}(\aaa)$ be the set of first order deformations of $\aaa$ up to equivalence. Obviously, there is a natural map
\begin{equation}\label{eqstrict}
\Def^s_{\mathrm{tw}}(\aaa) \lra \Def_{\mathrm{tw}}(\aaa).
\end{equation}

The following proposition can be proven along the lines of the parallel statement for linear categories \cite[Theorem B.3]{lowenvandenberghab}.

\begin{proposition}\label{propstrict}
The map \eqref{eqstrict} is an isomorphism.
\end{proposition}

The following theorem generalizes part of \cite[Theorem 3.26]{lowenprestacks}.

\begin{theorem}\label{defeq}
Let $\uuu$ be a finite poset with binary meets. Let $\aaa\colon \uuu^{\op} \lra \Cat(k)$ be a quasi-compact semi-separated prestack on $\uuu$. Every first order deformation of $\aaa$ is a quasi-compact semi-separated prestack on $\uuu$. There is an isomorphism
\begin{equation} \label{eqtwab}
\Def_{\mathrm{tw}}(\aaa) \lra \Def_{\mathrm{ab}}(\Qch(\aaa)),\quad \bbb \longmapsto \Qch(\bbb).
\end{equation}
\end{theorem}

\begin{proof}
Put $\uuu^{\star} = \uuu \cup \{ \star \}$ with top $\star$ as before.

Consider a first order deformation $\bbb$ of $\aaa$. By Propositions \ref{propliftloc} and \ref{propliftcomp}, $\bbb\colon \uuu^{\op} \lra \Cat(k[\epsilon])$ is a quasi-compact semi-separated prestack on $\uuu$. By Proposition \ref{propdesflat} (2), $\Qch(\bbb)$ is a flat abelian category over $k[\epsilon]$. It is easily seen that $\Qch(\bbb)$ is an abelian deformation of $\Qch(\aaa)$ so the map \eqref{eqtwab} is well defined.

An inverse to \eqref{eqtwab} is constructed as follows. Let $\ddd(\star)$ be an abelian deformation of $\Qch(\aaa)$. By Theorem \ref{thmgroth}, $\Qch(\aaa)$ is Grothendieck, whence by Proposition \ref{liftgroth}, $\ddd(\star)$ is Grothendieck as well.
By Proposition \ref{propdesflat} (1), we have localizations $\Mod(\aaa(U)) \lra \Qch(\aaa)$ for every $U \in \uuu$. Hence by Theorem \ref{thmloc11}, we obtain a unique corresponding localization $u_{U, \ast}\colon \ddd(U) \lra \ddd(\star)$ for every $U \in \uuu$, which is such that $\ddd(U)$ is a deformation of $\Mod(\aaa(U))$. For $u\colon V \lra U$ in $\uuu$, there is a unique localization $\ddd_u(V) \lra \ddd(U)$ corresponding to the localization $\Mod(\aaa(V)) \lra \Mod(\aaa(U))$. Consequently, the localizations $\ddd_u(V) \lra \ddd(U) \lra \ddd(\star)$ and $\ddd(V) \lra \ddd(\star)$ are equivalent. We may suppose all localizations are strict, whence we have $\ddd_u(V) = \ddd(V) \subseteq \ddd(U) \subseteq \ddd(\star)$. An arbitrary choice of left adjoints $u_U^{\ast}\colon \ddd(\star) \lra \ddd(U)$ and $u^{\ast}\colon \ddd(U) \lra \ddd(V)$ to these localizations can be organized into a prestack $\ddd^{\star}$ on $\uuu^{\star}$ with restriction $\ddd$ to $\uuu$.

Let $\bar{\aaa}(U) \subseteq \ddd(U)$ be the full subcategory of all flat lifts of objects in $\aaa(U)$ along $k \otimes_{k[\epsilon]} -\colon \ddd(U) \lra \Mod(\aaa(U))$. By Proposition \ref{liftproj}, $\bar{\aaa}(U)$ consists of a collection of finitely generated projective generators of $\ddd(U)$. We obtain a prestack $\bar{\aaa}$ and an equivalence of prestacks $k \otimes_{k[\epsilon]} \bar{\aaa} \cong \aaa$. Thus, $\bar{\aaa}$ is a deformation of $\aaa$ and there is an equivalence of prestacks $\ddd \cong \Mod_{\bar{\aaa}}$ and the category $\Des(\ddd) \cong \Qch(\bar{\aaa})$ is Grothendieck.
The canonical functor $\varphi\colon \ddd(\star) \lra \Des(\ddd)$ is exact and has a right adjoint $\lambda\colon \Des(\ddd) \lra \ddd(\star)$, $(M_U)_U \longmapsto \lim_U u_{U, \ast} M_U$. By assumption, $\lambda_k$ and $\varphi_k$ constitute an equivalence of categories. By Proposition \ref{proploclift}, $\lambda$ is a localization. Further, by Theorem \ref{thmloc11}, $\lambda$ and $\varphi$ necessarily constitute an equivalence.

Conversely, when $\bbb$ is a first order deformation of $\aaa$, the inverse construction applied to a $\Qch(\bbb)$ yields back a deformation $\bar{\aaa}$ of $\aaa$ which is equivalent to $\bbb$.
\end{proof}

\begin{remark}
Theorem \ref{defeq} holds true for arbitrary infinitesimal deformations in the deformation setup of \cite{lowenvandenberghab}, as can be shown inductively.
\end{remark}

\subsection{Comparison with Toda's construction}\label{partoda}

Let $\uuu$ be a finite poset with binary meets.
Let $\aaa\colon \uuu^{\op} \lra \Alg(k)$ be a quasi-compact semi-separated presheaf of algebras in the sense of Definition \ref{defafpre}.
Combining \eqref{gsop}, Propositions \ref{proptwist}, \ref{propstrict} and Theorem \ref{defeq}, we obtain a canonical isomorphism
\begin{equation}\label{intbas}
HH^2(\aaa) \lra HH^2(\aaa^{\op}) \lra \Def_{\mathrm{tw}}(\aaa^{\op}) \lra \Def_{\mathrm{ab}}(\Qch^l(\aaa)).
\end{equation}
Let $X$ be a scheme with a finite semi-separating cover $\uuu$. From \eqref{intbas} we thus obtain a canonical isomorphism
\begin{equation}\label{intbas2}
HH^2(\ooo_X|_{\uuu}) \lra \Def_{\mathrm{ab}}(\Qch(X)).
\end{equation}\

For a scheme $X$ with semi-separating cover $\uuu$, put
$$HH_\mathrm{HKR}^2(X, \uuu) = \check{H}^2(\uuu, \ooo_X) \oplus \check{H}^1(\uuu, \ttt_X) \oplus \check{H}^0(\uuu, \wedge^2 \ttt_X).$$
In \cite{toda}, for a quasi-compact, separated smooth scheme $X$ with semi-separating cover $\uuu$, Toda associates to an element
$u \in HH_\mathrm{HKR}^2(X)$
a certain ``first order deformation'' of the abelian category $\Qch(X)$. This is an abelian $k[\epsilon]$-linear category, which is not a priori a deformation in the sense of  \S \ref{pardefab}.
Let us give a brief review of Toda's construction. Suppose that $u$ is represented by a triple of cocycles $(\alpha, \beta, \gamma)$. He first constructs a sheaf $\ooo_X^{(\beta,\gamma)}$ of non-commutative $k[\epsilon]$-algebras on $X$ which depends upon $\beta$ and $\gamma$. Next, the element $\alpha$ gives rise to an element $\tilde{\alpha} = 1- \alpha \epsilon \in \check{H}^2(X, (\ooo_X^{(\beta,\gamma)})^{\times})$, and he defines
$$\Qch^l(X,u) = \Qch^l(\ooo_X^{(\beta,\gamma)}, \tilde{\alpha})$$ as the category of quasi-coherent $\tilde{\alpha}$-twisted left $\ooo_X^{(\beta,\gamma)}$-modules (defined in analogy with the case of quasi-coherent twisted modules over a scheme, see for instance \cite{caldararu}).
The category is independent of the choice of the cover $\uuu$ and the triple $(\alpha, \beta, \gamma)$, up to equivalence.

We know from \S \ref{secHodge-HKR} that
\[
HH^2(X) \cong HH^2(\ooo_X|_{\uuu}) \cong HH_\mathrm{HKR}^2(X, \uuu),
\]
so both \eqref{intbas2} and Toda's construction give concrete ways to interpret a Hochschild $2$-class of a smooth quasi-compact separated scheme in terms of a certain $k[\epsilon]$-linear abelian category.
In fact, the conditions on $X$ are necessary for the interpretation in terms of $HH^2(X)$, but not for the construction of the category $\Qch(X, u)$.

In this section, we use the isomorphisms
\begin{equation}\label{transfos}
H^2\CC_\mathrm{GS}(\ooo_X|_{\uuu}) \cong \oplus_{p + q = 2} H^p(\uuu, \wedge^q\ttt_X|_{\uuu}) \cong \oplus_{p + q = 2} \check{H}^p(\uuu, \wedge^q\ttt_X)
\end{equation}
from \S \ref{parHodgeHKR} and \S \ref{parcechsimp} in order to show that both constructions are equivalent for a smooth scheme $X$ with semi-separating cover $\uuu$.

We first return to the more general setup of a presheaf of commutative algebras, and we associate to a GS cocycle a Toda-type deformation.
Let $\aaa$ be a presheaf of commutative $k$-algebras on $\uuu$, and $\CCC_{\mathrm{GS}}\in H^2\CC_{\mathrm{GS}}(\aaa)$. Suppose that $\CCC_\mathrm{GS}$ is represented by a cocycle $$(m_1,f_1,c_1) \in  \CC^{0,2}(\aaa) \oplus \CC^{1,1}(\aaa) \oplus \CC^{2,0}(\aaa).$$ As before, we assume the cocycle is normalized and reduced. Since $\aaa(U)$ is not necessarily a smooth algebra, we cannot assume that $(m_1,f_1,c_1)$ is decomposable. Instead, by Proposition \ref{propsplit} we have a weaker decomposition of normalized reduced cocycles,
$$
(m_1,f_1,c_1) = (m_1,f_1,0) + (0,0,c_1).
$$

Let $\bar{\aaa}$ be the first order twisted deformation of $\aaa$ determined by the cocycle $(m_1,f_1,c_1)$. By Proposition \ref{centralkey}, $\bar{\aaa}$ has central twists, and the underlying presheaf $\underline{\bar{\aaa}}$ is the first order presheaf deformation of $\aaa$ determined by the cocycle $(m_1, f_1, 0)$.

Following Toda's idea, let us give a characterization of $\underline{\bar{\aaa}}$. Recall the complex of presheaves $(\aaa^\bullet,\varphi^\bullet)$ defined in \S \ref{paracechlikecomp}. We define a map $F\colon \aaa\oplus\aaa^0\lra\aaa^1$ of presheaves by $(f_1,\varphi^0)$, namely, for all $U\in\uuu$,
\begin{align*}
F^U\colon \aaa(U)\oplus\prod_{u\colon V\lra U}\aaa(V)&\lra\prod_{\substack{u\colon V\lra U\\v\colon W\lra V}}\aaa(W)\\
(a,(b^u)_u)&\longmapsto (f_1^vu^*(a)+v^*(b^u)-b^{uv})_{u,v}.
\end{align*}
Equip $\aaa(U)\oplus\aaa^0(U)$ with the multiplication by
\[
(a,(b^u)_u)\cdot(a',(b'^u)_u)=(aa',(u^*(a)b'^u+b^uu^*(a')+m_1(u^*(a), u^*(a')))_u).
\]
Then $\aaa(U)\oplus\aaa^0(U)$ becomes a $k[\epsilon]$-algebra by setting
\[
(\lambda+\kappa\epsilon)(a,(b^u)_u)=(\lambda a, (\kappa u^*(a)+\lambda b^u)_u),\quad \lambda,\kappa\in k.
\]
Let $G^U\colon \underline{\bar{\aaa}}(U)\lra \aaa(U)\oplus\aaa^0(U)$ be the map $a+b\epsilon\longmapsto (a, (f_1^u(a)+u^*(b))_u)$. It is easy to verify that $G^U$ is a homomorphism of $k[\epsilon]$-algebras, and the sequence
\[
0\lra \underline{\bar{\aaa}}\xrightarrow[\quad]{G} \aaa\oplus\aaa^0 \xrightarrow[\quad]{F} \aaa^1
\]
is exact.

We obtain categories $\QPr(\bar{\aaa})$, $\QPr^l(\bar{\aaa})$ as defined in \S \ref{parqchmodule}, which are independent of the choice of cocycle, up to equivalence.

Now consider the case $\aaa= \ooo = \ooo_X|_{\uuu}$ for $X$ a smooth scheme with a semi-separating cover $\uuu$. Assume that $(m_1,f_1,c_1)$ is a normalized reduced decomposable cocycle. Let $\bar{\ooo}$ be the corresponding first order twisted deformation.

By the correspondence $\CCC_{\mathrm{GS}}\longmapsto\CCC_{\mathrm{simp}}$ given in \S \ref{parHodgeHKR}, we obtain three reduced simplicial cocycles $\Theta_{0,2},\Theta_{1,1},\Theta_{2,0}$. Let $(\alpha,\beta,\gamma)=(\iota(\Theta_{2,0}),\iota(\Theta_{1,1}),\iota(\Theta_{0,2}))$ consist of the corresponding \v{C}ech cocycles, representing a class $u$.

\begin{theorem}\label{thmequiv}
There are equivalences $$\Qch^l(\bar{\ooo}) \cong \QPr^l(\bar{\ooo})\cong\Qch^l(X, u)$$ of Grothendieck abelian categories.

If $\uuu$ is finite, then the abelian categories are flat over $k[\epsilon]$ and constitute (equivalent) first order abelian deformations of $\Qch(X)$.
\end{theorem}

\begin{proof}
The first equivalence $\Qch^l(\bar{\ooo}) \cong \QPr^l(\bar{\ooo})$ was shown in Theorem \ref{thmqmodqpr} (applied to $\bar{\ooo}^{\op}$), and the category $\Qch^l(\bar{\ooo})$ (and hence also $\QPr^l(\bar{\ooo})$) is abelian by Proposition \ref{lemqchabelian}.
According to the quasi-isomorphism given in \S \ref{parcechsimp}, $m_1^U=\Theta_{2,0}^U=\gamma^U$, so the multiplications on $\ooo(U)$ deformed by $m_1$ and $\gamma$ are the same. Furthermore, if $V\subseteq U$, then $\beta^{V,U}=\Theta_{1,1}^{V\subseteq U}-\Theta_{1,1}^{U\subseteq U}=\Theta_{1,1}^{V\subseteq U}$ lifts $f_1^{V\subseteq U}$ through the restriction map $\ooo(U)\lra\ooo(V)$. This yields that as a presheaf on $\uuu$, $$\underline{\bar{\ooo}} = \ooo_X^{(\beta,\gamma)}|_\uuu.$$

Consider $(\fff_U,\phi_u)\in\QPr^l(\bar{\ooo})$. For any pair $(U_i,U_j)\in\uuu\times\uuu$, define
\[
\psi_{ij}=(\phi_{U_{ij}\subseteq U_j})^{-1}\phi_{U_{ij}\subseteq U_i}\colon \fff_{U_i}|_{U_{ij}}\lra \fff_{U_j}|_{U_{ij}}
\]
where $U_{ij}=U_i\cap U_j$. We know $\Qch^l({\bar{\ooo}}|_U)\cong\Mod^l(\bar{\ooo}(U))$ for all $U$. On the other hand, Toda has proved $\Qch^l(\ooo_X^{(\beta,\gamma)}|_U)\cong \Mod^l(\bar{\ooo}(U))$. Thus we obtain an equivalence $\Xi\colon \Qch^l(\bar{\ooo}|_U)\lra \Qch^l(\ooo_X^{(\beta,\gamma)}|_U)$ with $\Xi(\widetilde{M})$ equal to the sheaf associated to $\widetilde{M}$.
In order to show that $(\fff_U,\phi_u)\longmapsto(\Xi(\fff_U), \Xi(\psi_{ij}))$ gives an equivalence, we need to show that the collection $(\Xi(\fff_U), \Xi(\psi_{ij}))$ forms a twisted quasi-coherent module in Toda's sense (see \cite[\S 4]{toda}). To this end, it suffices to check that the twisted cocycle condition
\[
\psi_{jk}\psi_{ij}=(1-\alpha^{U_i,U_j,U_k}\epsilon)\psi_{ik}
\]
holds for any $U_i$, $U_j$, $U_k\in \uuu$.

It follows from the chains $U_{ijk}\subseteq U_{ij}\subseteq U_j$ and $U_{ijk}\subseteq U_{jk}\subseteq U_j$ that
\begin{align*}
\phi_{U_{ijk}\subseteq U_{ij}}\phi_{U_{ij}\subseteq U_j}=(1+c_1^{U_{ijk}\subseteq U_{ij}\subseteq U_j}\epsilon)\phi_{U_{ijk}\subseteq U_j},\\
\phi_{U_{ijk}\subseteq U_{jk}}\phi_{U_{jk}\subseteq U_j}=(1+c_1^{U_{ijk}\subseteq U_{jk}\subseteq U_j}\epsilon)\phi_{U_{ijk}\subseteq U_j}.
\end{align*}
After canceling $\phi_{U_{ijk}\subseteq U_j}$, we obtain
\[
\phi_{U_{jk}\subseteq U_j}(\phi_{U_{ij}\subseteq U_j})^{-1}=\bigl(1-(c_1^{U_{ijk}\subseteq U_{ij}\subseteq U_j}-c_1^{U_{ijk}\subseteq U_{jk}\subseteq U_j})\epsilon\bigr)(\phi_{U_{ijk}\subseteq U_{jk}})^{-1}\phi_{U_{ijk}\subseteq U_{ij}}.
\]
So
\begin{align*}
&\phantom{\;=\;}\psi_{jk}\psi_{ij}=(\phi_{U_{jk}\subseteq U_k})^{-1}\phi_{U_{jk}\subseteq U_j}(\phi_{U_{ij}\subseteq U_j})^{-1}\phi_{U_{ij}\subseteq U_i}\\
&=\bigl(1-(c_1^{U_{ijk}\subseteq U_{ij}\subseteq U_j}-c_1^{U_{ijk}\subseteq U_{jk}\subseteq U_j})\epsilon\bigr)(\phi_{U_{jk}\subseteq U_k})^{-1}(\phi_{U_{ijk}\subseteq U_{jk}})^{-1}\phi_{U_{ijk}\subseteq U_{ij}}\phi_{U_{ij}\subseteq U_i}\\
&=\bigl(1-(c_1^{U_{ijk}\subseteq U_{ij}\subseteq U_j}-c_1^{U_{ijk}\subseteq U_{jk}\subseteq U_j})\epsilon\bigr)(\phi_{U_{ijk}\subseteq U_{jk}}\phi_{U_{jk}\subseteq U_k})^{-1}\phi_{U_{ijk}\subseteq U_{ij}}\phi_{U_{ij}\subseteq U_i}\\
&=\bigl(1-(c_1^{U_{ijk}\subseteq U_{ij}\subseteq U_j}-c_1^{U_{ijk}\subseteq U_{jk}\subseteq U_j})\epsilon\bigr)(1-c_1^{U_{ijk}\subseteq U_{jk}\subseteq U_k}\epsilon)(\phi_{U_{ijk}\subseteq U_{k}})^{-1}\\
&\phantom{\;=\;}(1+c_1^{U_{ijk}\subseteq U_{ij}\subseteq U_i}\epsilon)\phi_{U_{ijk}\subseteq U_{i}}\\
&=\bigl(1-(c_1^{U_{ijk}\subseteq U_{ij}\subseteq U_j}-c_1^{U_{ijk}\subseteq U_{jk}\subseteq U_j}+c_1^{U_{ijk}\subseteq U_{jk}\subseteq U_k}-c_1^{U_{ijk}\subseteq U_{ij}\subseteq U_i})\epsilon\bigr)\\
&\phantom{\;=\;}(\phi_{U_{ijk}\subseteq U_{k}})^{-1}\phi_{U_{ijk}\subseteq U_{i}}\\
&=\bigl(1-(c_1^{U_{ijk}\subseteq U_{ij}\subseteq U_j}-c_1^{U_{ijk}\subseteq U_{jk}\subseteq U_j}+c_1^{U_{ijk}\subseteq U_{jk}\subseteq U_k}-c_1^{U_{ijk}\subseteq U_{ij}\subseteq U_i})\epsilon\bigr)\\
&\phantom{\;=\;}(1+c_1^{U_{ijk}\subseteq U_{ik}\subseteq U_k}\epsilon)(\phi_{U_{ijk}\subseteq U_{ik}}\phi_{U_{ik}\subseteq U_k})^{-1}(1-c_1^{U_{ijk}\subseteq U_{ik}\subseteq U_i}\epsilon)\phi_{U_{ijk}\subseteq U_{ik}}\phi_{U_{ik}\subseteq U_i}\\
&=\bigl(1-(c_1^{U_{ijk}\subseteq U_{ij}\subseteq U_j}-c_1^{U_{ijk}\subseteq U_{jk}\subseteq U_j}+c_1^{U_{ijk}\subseteq U_{jk}\subseteq U_k}-c_1^{U_{ijk}\subseteq U_{ij}\subseteq U_i}\\
&\phantom{\;=\;}{}-c_1^{U_{ijk}\subseteq U_{ik}\subseteq U_k}+c_1^{U_{ijk}\subseteq U_{ik}\subseteq U_i})\epsilon\bigr)(\phi_{U_{ik}\subseteq U_k})^{-1}\phi_{U_{ik}\subseteq U_i}\\
&=(1-\iota(\Theta_{2,0})^{U_i,U_j,U_k}\epsilon)\psi_{ik}\\
&=(1-\alpha^{U_i,U_j,U_k}\epsilon)\psi_{ik}.
\end{align*}

Therefore, the corresponding $(\fff_U,\phi_u)\longmapsto (\Xi(\fff_U),\Xi(\psi_{ij}))$ gives an equivalence $\QPr^l(\bar{\ooo})\cong\Qch^l(\ooo_X^{(\beta,\gamma)},\tilde{\alpha})\cong\Qch^l(X,u)$ of abelian categories.

If $\uuu$ is finite, the additional statement is contained in the existence of \eqref{intbas} which was shown earlier on.
\end{proof}

\def\cprime{$'$} \def\cprime{$'$}
\providecommand{\bysame}{\leavevmode\hbox to3em{\hrulefill}\thinspace}
\providecommand{\MR}{\relax\ifhmode\unskip\space\fi MR }
\providecommand{\MRhref}[2]{%
  \href{http://www.ams.org/mathscinet-getitem?mr=#1}{#2}
}
\providecommand{\href}[2]{#2}

\end{document}